\documentclass[12pt]{article}
\usepackage{bbding}
\usepackage{bm, commath}
\usepackage{natbib}
\usepackage{caption}
\usepackage{graphicx}
\usepackage{subfig}
\usepackage{amsmath, amsfonts, amsthm}
\usepackage{float}
\usepackage{booktabs,siunitx}
\usepackage{hyperref}
\usepackage{url}
\usepackage{multirow}
\usepackage{amssymb}
\usepackage{blkarray}
\usepackage{bbm}
\usepackage{multicol}
\usepackage{authblk}
\usepackage{natbib}
\usepackage[shortlabels]{enumitem}
\usepackage{longtable}
\usepackage{adjustbox}
\usepackage{threeparttable}
\usepackage{hhline}
\usepackage{colortbl}
\usepackage{lscape}

\usepackage{mathtools}
\usepackage{mathabx}
\usepackage{rotating}
\usepackage{caption}      
\usepackage{subcaption}   
\usepackage{booktabs} 

\usepackage{pifont}
\newcommand{\cmark}{\ding{51}}
\newcommand{\xmark}{\ding{55}}

\usepackage{listings}
\usepackage{bbm}
\usepackage{textcomp}
\usepackage{authblk}
\usepackage[ruled]{algorithm2e}
\SetKwInput{KwParam}{Parameter}
\SetAlgoCaptionLayout{centerline}

\usepackage{sectsty}
\usepackage[doublespacing]{setspace}

\usepackage[utf8]{inputenc}
\usepackage{float}
\usepackage{tikz}
\usetikzlibrary{shapes,decorations,arrows,calc,arrows.meta,fit,positioning}

\usepackage{mwe}
\usepackage{graphicx}

\usepackage[textsize=scriptsize, textwidth = 2cm, shadow]{todonotes}

\usepackage[textsize=scriptsize, textwidth = 2cm, shadow]{todonotes}

\newcommand{\indep}{\mathrel{\perp\mspace{-10mu}\perp}}

\newtheorem{theorem}{Theorem}
\newtheorem{definition}{Definition}

\newtheorem{remark}{Remark}

\newtheorem{assumption}{Assumption}
\newtheorem*{assumption*}{Assumption}

\newtheorem{proposition}{Proposition}

\newtheorem{lemma}{Lemma}

\newtheorem{innercustomgeneric}{\customgenericname}
\providecommand{\customgenericname}{}
\newcommand{\newcustomtheorem}[2]{%
  \newenvironment{#1}[1]
  {%
   \renewcommand\customgenericname{#2}%
   \renewcommand\theinnercustomgeneric{##1}%
   \innercustomgeneric
  }
  {\endinnercustomgeneric}
}

\newcustomtheorem{customthm}{Theorem}
\newcustomtheorem{customassumption}{Assumption}

\usepackage{mathtools}

\usepackage{xcolor}

\makeatletter
\renewcommand{\algocf@captiontext}[2]{#1\algocf@typo. \AlCapFnt{}#2} 
\def\@algocf@capt@plain{top}
\renewcommand{\algocf@makecaption}[2]{%
  \addtolength{\hsize}{\algomargin}%
  \sbox\@tempboxa{\algocf@captiontext{#1}{#2}}%
  \ifdim\wd\@tempboxa >\hsize
  \hskip .5\algomargin%
  \parbox[t]{\hsize}{\algocf@captiontext{#1}{#2}}
  \else%
  \global\@minipagefalse%
  \hbox to\hsize{\box\@tempboxa}
  \fi%
  \addtolength{\hsize}{-\algomargin}%
}
\makeatother

\allowdisplaybreaks

\begin{document}

\def\spacingset#1{\renewcommand{\baselinestretch}%
{#1}\small\normalsize} \spacingset{1}

\sectionfont{\bfseries\large\sffamily}%

\subsectionfont{\bfseries\sffamily\normalsize}%

\title{GMM with Many Weak Moment Conditions and Nuisance Parameters: General Theory and Applications to Causal Inference}

\author[1]{Rui Wang}
\author[1]{Kwun Chuen Gary Chan$^*$}
\author[1]{Ting Ye\thanks{Corresponding to: {\tt tingye1@uw.edu and kcgchan@uw.edu}}}

\affil[1]{Department of Biostatistics, University of Washington}
\date{}
\maketitle

\begin{abstract}

{ Weak identification arises in many statistical problems when key variables exhibit weak correlations—for example, when instrumental variables correlate weakly with treatment, or when proxy variables correlate weakly with unmeasured confounders. Under weak identification, standard estimation methods such as the generalized method of moments (GMM) can produce substantial bias, both in finite samples and asymptotically. This challenge is compounded in modern applications that require estimating many nuisance parameters. This paper develops a framework for estimation and inference of a finite-dimensional target parameter in general moment models with the number of weak moment conditions and nuisance parameters growing with sample size. We analyze a general two-step debiasing estimator that accommodates flexible, possibly nonparametric first-step estimation of nuisance parameters, in which Neyman orthogonality plays a more critical role in obtaining debiased inference than in conventional settings with strong identification. Under a many-weak-moment asymptotic regime, we establish the estimator's consistency and asymptotic normality. We provide high-level conditions for the general setting and demonstrate their application to two important special cases: inference with weak instruments and inference with weak proxies.}

\end{abstract}
\vspace{0.3 cm}

\newpage
\spacingset{1.6}

\setlength\abovedisplayskip{2pt}
\setlength\belowdisplayskip{1pt}%

\section{Introduction}
\subsection{Weak identification problem}
Weak identification is a common issue for many statistical applications. A well-known example is the problem of weak instrumental variables (IVs). IV methods are widely used to identify and estimate causal effects when there is unmeasured confounding \citep{angrist1996JASA,hernan2006instruments}. These methods typically rely on three standard assumptions:  (i) IV independence: the IVs are independent of unmeasured confounders; (ii) Exclusion restriction: the IVs do not have direct effect on the outcome variable; and (iii) IV relevance: the IVs are correlated with the treatment \citep{Baiocchi2014SIM}. While the IV relevance assumption can be empirically assessed and justified in many examples, the correlation between the instruments and the treatment is often weak. For example, in epidemiology, Mendelian randomization is a popular strategy to study the causal effects of some exposures in the presence of unmeasured confounding, using single-nucleotide polymorphisms (SNPs) as instruments. However, SNPs are usually weakly correlated with the exposure of interest, which leads to substantial bias using traditional methods \citep{sanderson2022mendelian}, such as two-stage least square or generalized method of moments (GMM) \citep{hansen1982GMM}. As another example, labor economists  used quarter of birth and
its interactions with year of birth as instruments for estimating the return to education  \citep{Angrist1991QJE}. Even if the total sample size is large, weak instruments problem might produce biased point and variance estimates, as discussed in \citet{Bound1995JASA}. 

Another important but less recognized example of weak identification arises in proximal causal inference \citep{ett2024proximal} with weak proxies. Like IV methods, proximal causal inference enables identification and estimation of causal effects in the presence of unmeasured confounding. This approach typically requires two types of proxies for the unmeasured confounders: treatment proxies and  outcome proxies \citep{Miao2018BKA,ett2024proximal}, where the treatment proxies are assumed to have no direct causal effect on the outcome, and the exposure is assumed to have no causal effect on the outcome proxies. When the proxies are only weakly correlated with the unmeasured confounders, they may fail to provide sufficient information about these confounders. As a result, standard methods (e.g., those proposed by \citet{ett2024proximal}) may yield inconsistent estimates and invalid inference, even in large samples, as demonstrated in our simulation study.

The weak identification problem is usually modeled through a drifting data-generating process where identification fails in the limit \citep{stock2000gmm,newey2009generalized}. Under weak identification, standard estimators—such as estimating equation estimators \citep{van2000asymptotic} and generalized method of moments (GMM) estimators with a fixed number of moment conditions \citep{hansen1982GMM}—can be biased, and their asymptotic distributions can become complex and nonstandard \citep{stock2000gmm}. Nevertheless, information about the target parameter may accumulate as the number of moment conditions increases. This has motivated the development of \emph{many weak moment asymptotics} \citep{han2006gmm,newey2009generalized}. Under this regime, traditional GMM estimators remain biased in general \citep{han2006gmm}, but consistent and asymptotically normal estimates can be obtained in the absence of nuisance parameters, using the continuous updating estimator (CUE) \citep{newey2009generalized}, which belongs to the broader family of estimators for moment conditions. However, modern semiparametric inference often involves flexibly estimating function-valued nuisance parameters. For example, the structural mean model approach for IV analysis involves modeling the conditional expectations of the instruments, treatment, and outcome given baseline covariates \citep{hernan2006instruments}. To mitigate model misspecification, flexible nonparametric estimation methods are frequently used to estimate function-valued nuisance parameters. The inclusion of such nuisance models complicates the estimation and inference of the target parameter, especially under weak identification.



\subsection{Previous works}
In this paper, we address two major estimation challenges that appear simultaneously, where each gathered a body of literature.  

The first challenge is to handle nuisance parameters. Function-valued nuisance parameters are very common in modern semiparametric models, especially in those efficient influence function based methods \citep{bickel1993efficient}.  Usually, a two-step estimator is considered, where the first-step is to estimate the nuisance parameters and the second step is to construct an efficient influence function-based estimator by plugging in the nuisance estimator. In the earlier literature, parametric estimation for nuisance parameters is widely adopted for constructing doubly robust \citep{Robins1994JASA} or multiple robust estimators \citep{ett2012aos}. More recently, based on the orthogonality property of efficient influence function \citep{chernozhukov2018double}, flexible nonparametric nuisance estimation coupled with cross-fitting was considered in one-step estimation \citep{Pfanzagl1990,andrea2021BKA}, estimating equation \citep{chernozhukov2018double}, targeted maximum likelihood estimation \citep{van2011targeted}, GMM \citep{Chernozhukov2022locallyrobust} and empirical likelihood \citep{Bravo2020AoS}. 

Another challenge is to deal with many weak moment conditions. The related weak IVs problem has been extensively studied in econometrics literature \citep{Staiger1996ECA,stock2000gmm,han2006gmm,newey2009generalized,andrews2019weak}. See  \citet{andrews2019weak} for a comprehensive review of weak IV literature. \citet{newey2009generalized} considered the case of high-dimensional weak moment conditions in the absence of nuisance parameters. A few recent papers consider weak moments with nuisance parameters. \citet{ye2024genius} considered many weak moment conditions with nuisance parameters under specific linear moment conditions and a one-dimensional paramter of interest in the setting of weak instrumental variables and Mendelian randomization. \citet{andrews2018valid} constructed valid confidence intervals under weak identification with first-step nuisance estimates but focused on a fixed number of moment conditions and did not consider point estimation alongside inference. Around the time we completed this paper, we became aware of a preprint by \citet{zhang2025debiased} on arXiv that proposed a similar idea for constructing an estimator but under a partially linear IV model. 

In contrast to these existing approaches, our work develops a more general theoretical framework that applies to both linear and nonlinear weak-identification problems and is not restricted to IV problems. See Table \ref{tab:contributions} for a comprehensive comparison with the previous weak IV literature and semiparametric estimation literature.

\begin{table}[ht]
    \centering
        \small
     \caption{Comparison of our contribution with existing literature, where \textbf{Growing $m$} means if they allow for growing number of moment conditions; \textbf{ML nuisance} means if they allow for general machine-learning based nuisance estimators; \textbf{General $g$} means if they allow for general linear and nonlinear moment conditions or if they focus on a specific form of $g$; \textbf{Weak $g$} means if they allow for weak moment conditions.}
    \label{tab:contributions}

    \begin{tabular}{lcccc}
    \toprule
    \textbf{Scholarship} & \textbf{Growing $m$} & \textbf{ML nuisance} & \textbf{General $g$} & \textbf{Weak $g$}\\
    \midrule
    \citet{chernozhukov2018double} & \xmark & \cmark & \cmark & \xmark \\
    \citet{newey2009generalized}& \cmark & \xmark & \cmark  & \cmark\\
    \citet{ye2024genius} & \cmark & \xmark & \xmark & \cmark \\
     \citet{zhang2025debiased} & \cmark & \cmark & \xmark &\cmark \\
     Our proposal & \cmark & \cmark & \cmark & \cmark \\
    \bottomrule
    \end{tabular}
\end{table}

\subsection{Main contributions}
{In this paper, we develop a framework
for estimation and inference of a finite-dimensional target parameter in general moment models with many weak moment conditions and many nuisance parameters. Our framework accommodates flexible first-step nuisance parameter estimators. In contrast, the two most closely related works \citep{ye2024genius,zhang2025debiased} focus exclusively on specific linear moment condition models for instrumental variable applications. By establishing a general theoretical framework, we provide new insights into the problem while offering a unified approach that can be applied to different settings without requiring new theory for each application.}

{ 
Specifically, our framework reveals that the interplay between function-valued nuisance parameters and many weak moments produces several unique features that set our theory apart from the existing semiparametric inference literature:}

{
\begin{enumerate}
    \item \textbf{Enhanced role of Neyman orthogonality.} Many weak moments necessitate a \emph{global Neyman orthogonality condition} (Definition \ref{Definition: Neyman Orthogonality}). Unlike standard semiparametric inference theory, where orthogonality primarily facilitates asymptotic normality, this global condition is essential for \emph{both} consistency and asymptotic normality. Figure \ref{fig: comparison of the estimators} demonstrates this importance: the continuously updated estimator (CUE) with non-orthogonal moment conditions remains consistent under strong identification (panel (d)) but becomes substantially biased under weak identification (panel (c)). With orthogonal moment conditions, however, the CUE maintains unbiasedness (panel (a)). Further details appear in Sections \ref{subsection: Global Neyman orthogonality} and \ref{subsection: Two key features}.
    \item \textbf{Stricter nuisance parameter requirements.} Many weak moments demand faster convergence rates than in classical settings. While standard theory requires convergence rates faster than $N^{-1/4}$, where $N$ is the sample size \citep{chernozhukov2018double}, our setting imposes more stringent conditions. We elaborate on these refined rate requirements in Sections \ref{subsection: consistency} and \ref{subsection: ASN}.
    \item \textbf{Non-standard asymptotics and variance estimation.} Our proposed estimator converges at a rate slower than $\sqrt{N}$, rendering it irregular in the classical sense \citep{van2000asymptotic}. The asymptotic expansion contains both the standard first-order term corresponding to the influence function of classical GMM and an additional non-negligible U-statistic term induced by the many weak moments. This necessitates careful control over how estimated nuisance parameters contribute to these U-statistic terms. Moreover, the classical influence function-based variance estimator becomes invalid in this setting. Details are provided in Section \ref{subsection: ASN}.
\end{enumerate}
}

\begin{figure}
    \centering
    \includegraphics[width=0.9\linewidth]{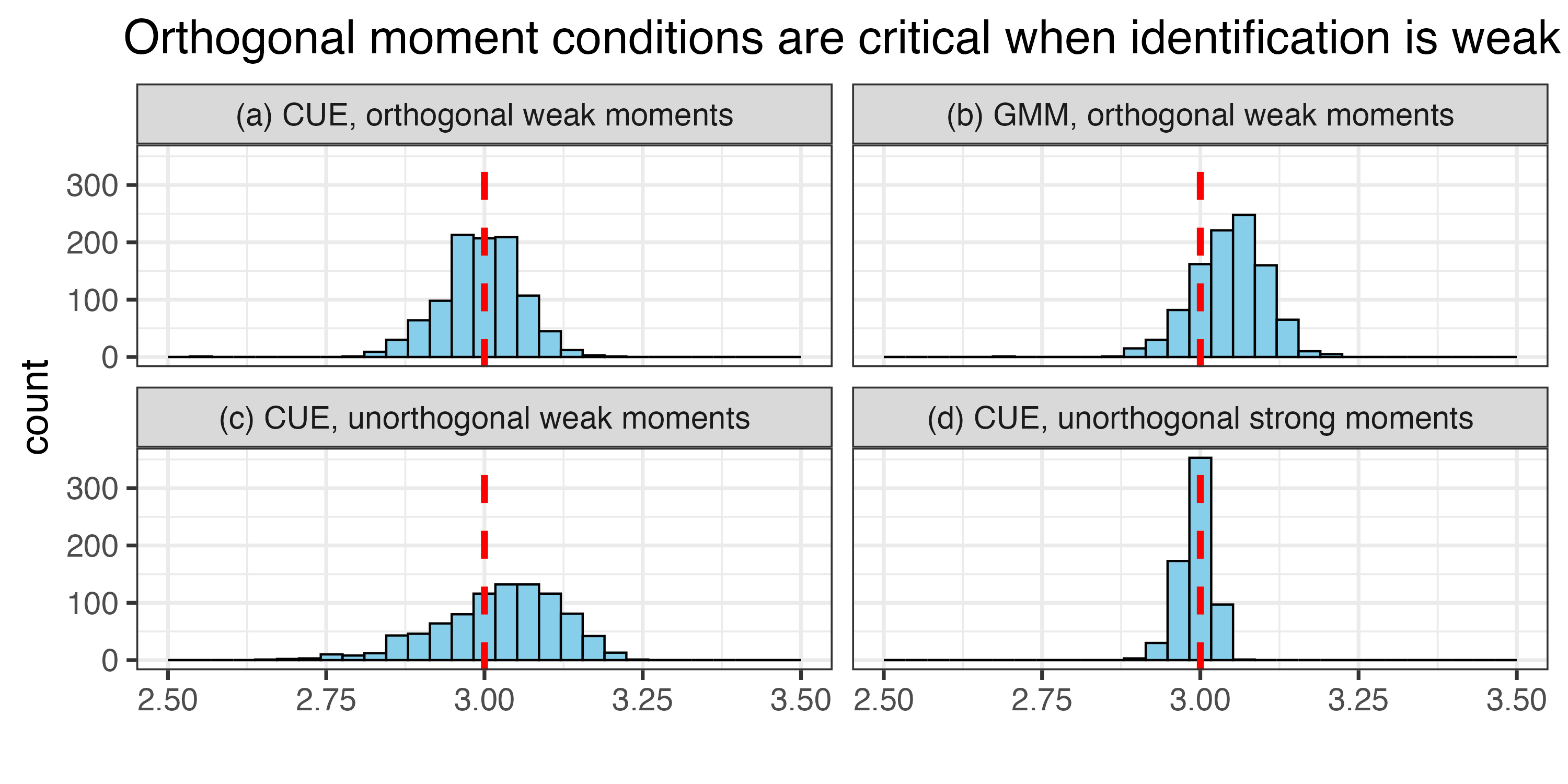} \vspace{-5mm}
    \caption{Neyman orthogonality plays a critical role under many weak moments. Data are simulated from an additive structural mean model, with the red dashed line indicating the true parameter value. Panels show sampling distributions across 1,000 simulations for: (a) CUE with orthogonal weak moments, (b) GMM with orthogonal weak moments, (c) CUE with non-orthogonal weak moments, and (d) CUE with non-orthogonal strong moments. Under weak identification, the orthogonalized CUE (panel a) remains unbiased, while the non-orthogonal CUE (panel c) exhibits substantial bias—despite being consistent under strong identification (panel d). Even with orthogonal moments, GMM (panel b) remains biased. Simulation details appear in Supplemental Material S13.2.}
    \label{fig: comparison of the estimators}
\end{figure}

{
We establish high-level conditions under which our proposed two-step CUE is consistent and asymptotically normal. This general approach enables future applications: researchers can apply our results to new problems by verifying the corresponding conditions. We demonstrate this versatility through three examples that verify these high-level regularity conditions. The first two examples involve instrumental variable settings: the additive structural mean model (ASMM) and the multiplicative structural mean model (MSMM). The third example addresses the setting of many weak proxies. To the best of our knowledge, our paper is the first to study weak identification of the target parameter in proximal causal inference. While \citet{bennett2023inferencestronglyidentifiedfunctionals} considered weakly identified nuisance parameters in proximal causal inference, they assumed the target parameter remained strongly identified.}

\section{Setup and two-step CUE estimator}
\vspace{-2mm}
\subsection{Framework: many weak moment conditions model}
We are interested in a finite-dimensional target parameter with true value $\beta_0$. The parameter space $\mathcal{B}$ is a compact subset of $\mathbb{R}^p$, and $\beta_0$ belongs to the interior of $\mathcal{B}$, satisfying the moment condition:
\begin{equation}
     \mathbbm{E}_{P_{0,N}}[g (O;\beta_0,\eta_{0,N})] = 0, \label{Equation: moment condition}
\end{equation}
where $P_{0,N}$ denotes the law of $O$, which may depend on sample size $N$, and $\mathbbm{E}_{P_{0,N}}[\cdot]$ denotes expectation with respect to $P_{0,N}$. The function $g(o;\beta,\eta) = (g^{(1)}(o;\beta,\eta),...,g^{(m)}(o;\beta,\eta))^T $ is an $m$-dimensional vector of known moment functions. The quantity $\eta_{0,N}\in \tilde{\mathcal{T}}_N$ is the true nuisance parameter value, which is a functional of the true data-generating distribution, where $\tilde{\mathcal{T}}_N$ denotes the nuisance parameter space. Finally, $(O_i)_{i=1}^N$ denotes observed data, assumed to be independent and identically distributed (i.i.d.) copies of $O$.

We consider the setting where the parameter $\beta_0$ is weakly identified. Weak identification is characterized by Assumption \ref{Assumption: weak moment condition}, which follows the formulation in \cite{newey2009generalized} in the setting without nuisance parameters. Define $G(o;\beta,\eta) = \frac{\partial}{\partial \beta}g(o;\beta,\eta)$, and let $\widebar{G}(\beta,\eta) = \mathbb{E}_{P_{0,N}} [{G}(O;\beta,\eta)]$, $\widebar{G}= \widebar{G}(\beta_0,\eta_{0,N})$, $\widebar{\Omega}(\beta,\eta) = \mathbb{E}_{P_{0,N}}[g(O;\beta,\eta)g^T(O;\beta,\eta)]$, $\widebar{\Omega} =\widebar{\Omega}(\beta_0,\eta_{0,N})$.

\begin{assumption}[Many weak moment conditions] \quad
    \begin{enumerate}[label=(\alph*)]
        \item There is a $p\times p$ matrix $S_N = \Tilde{S}_N \text{diag}(\mu_{1N},...,\mu_{pN})$ such that $\Tilde{S}_N$ is bounded, the smallest eigenvalue of $\Tilde{S}_N\Tilde{S}^T_N$ is bounded away from zero, for each $j$ either $\mu_{jN} = \sqrt{N}$ or $\mu_{jN}/\sqrt{N}\rightarrow 0$, $\mu_N = \min_{1\leq j\leq p}\mu_{jN}\rightarrow \infty$, $m/\mu_N^2 = O(1)$.
        \item $NS_N^{-1}\widebar{G}^T\widebar{\Omega}^{-1}\widebar{G} (S_N^{-1})^T\rightarrow H$, $H$ is nonsingular.
        \end{enumerate} \label{Assumption: weak moment condition}
\end{assumption}
Under the standard asymptotic regime, where $\mu_{jN}=\mu_{N}=\sqrt{N}$, $m$ is finite and $P_{0,N}$ does not depend on $N$, we may take $\Tilde{S}_N$ to be $I_p$ (the $p$-dimensional identity matrix). In this setting, Assumption \ref{Assumption: weak moment condition}(b) simplifies to requiring that $\widebar{G}^T\widebar{\Omega}^{-1}\widebar{G}$ is a fixed positive semidefinite matrix. This corresponds to the classical GMM asymptotics \citep{hansen1982GMM}, where  $\widebar{G}^T\widebar{\Omega}^{-1}\widebar{G}$ represents the asymptotic variance of the efficient GMM. When $\mu_N/\sqrt{N}\rightarrow 0$, the parameter $\beta_0$ becomes weakly identified. Later, in Section \ref{subsection: CUE}, we define the objective function of CUE (when the true value of nuisance is known) to be $\widehat{Q}(\beta,\eta_{0,N}) = \widehat{g}^T(\beta,\eta_{0,N})\widehat{\Omega}(\beta,\eta_{0,N})^{-1}\widehat{g}(\beta,\eta_{0,N})/2$, where $\widehat{g}(\beta,\eta) = \frac{1}{N}\sum_{i=1}^Ng(O_i;\beta,\eta)$ and $\widehat{\Omega}(\beta,\eta)=\frac{1}{N}\sum_{i=1}^N{g}(O_i;\beta,\eta){g}(O_i;\beta,\eta)^T$. Following \citet{newey2009generalized}, the Hessian of the objective function $\frac{\partial \widehat{Q}(\beta,\eta_{0,N})}{\partial \beta \partial \beta^T}$  converges to $\widebar{G}^T\widebar{\Omega}^{-1}\widebar{G}$ at $\beta=\beta_0$. Therefore, Assumption \ref{Assumption: weak moment condition}(b) implies that the objective function becomes nearly flat in certain directions around $\beta_0$, reflecting weak identification along those directions. 

Importantly,  Assumption \ref{Assumption: weak moment condition} allows for varying identification strength across different linear combinations of $\beta_0$, as characterized by $\Tilde{S}_N$. To see this, Assumption \ref{Assumption: weak moment condition}(b) can be rewritten as
\small
\begin{align*}
N \text{diag}(\mu_{1N}^{-1},...,\mu_{pN}^{-1})\mathbb{E}_{P_{0,N}}\left[\frac{\partial g(O;\beta_0,\eta_{0,N})}{\partial \tilde{S}^T_N\beta}\right]^T\widebar{\Omega}^{-1}\mathbb{E}_{P_{0,N}}\left[\frac{\partial g(O;\beta_0,\eta_{0,N})}{\partial \tilde{S}^T_N\beta}\right]\text{diag}(\mu^{-1}_{1N},...,\mu^{-1}_{pN})\rightarrow H.
\end{align*}
\normalsize
This formulation implies that the identification strength for each coordinate of $\tilde{S}_N^T\beta$ is measured by $\mu_{1N},...,\mu_{pN}$, and the overall identification strength is captured by $\mu_N = \min_{1\leq j\leq p}\mu_{jN}$. {In the special case where $\beta_0$ is a scalar, we can set $\Tilde{S}_N=1$, so the condition reduces to $N\mu_N^{-2}\widebar{G}^T\widebar{\Omega}^{-1}\widebar{G}\rightarrow H$, with $H$ a positive number; this is the setting considered by \cite{ye2024genius} and \cite{zhang2025debiased}, where there is no need to account for varying identification strength across different directions.}
The requirement $m/\mu_N^2 = O(1)$ imposes a lower bound on the strength of identification, ensuring that the number of moment conditions does not grow too quickly relative to the identification strength.


The following notations will be used throughout the paper.  For any vector $x$, we denote its $l_2$ norm by $\Vert x\Vert$. For a matrix $A$, $\Vert A\Vert$ denotes the spectral norm, and $\Vert A\Vert_F$ denotes the Frobenius norm. The $L^q(P)$ norm is written as $\Vert f \Vert_{P,q} = (\int |f|^q dP)^{1/q}$. For a vector $v$, $v^{(j)}$ denotes its $j$th element. For a matrix $A$, $A^{(j,k)}$ denotes the $(j,k)$th element, and $A^{(j)}$ denotes the $j$th column. For two matrices $A$ and $B$, we write $A \preceq B$ if $B-A$ is positive semidefinite. For a sequence of random variables $X_N$ and a sequence of positive numbers $a_N$, we write $X_N = o_p(a_N)$ if for all $\epsilon>0$, 
$\lim_{N\rightarrow \infty} P_{0,N}(\vert X_{N}\vert>\epsilon a_N) = 0$. We write $X_{N} = O_p(a_N)$ if for all $\epsilon>0$, there exists a constant $C>0$ such that $P_{0,N}(\vert X_N\vert>Ca_N)\leq 1-\epsilon$ for all $N$.

\subsection{Removing weak identification bias}
\label{subsection: CUE}
It is well-known that GMM estimators are typically biased under many weak asymptotics. \citet{han2006gmm} provided an intuitive explanation for this bias: the GMM objective function may not be asymptotically minimized at the true parameter value $\beta_0$. To illustrate this point, we assume for simplicity that the nuisance parameter $\eta_{0,N}$ is known. In this case, the GMM estimator is defined as $\widehat{\beta}_{\rm GMM} = \arg\min_{\beta \in \mathcal{B}}\widehat{g}(\beta,\eta_{0,N})^TW\widehat{g}(\beta,\eta_{0,N})/2$, where $W$ is a positive semidefinite weighting matrix, and $\widehat{g}(\beta,\eta) = \frac{1}{N}\sum_{i=1}^Ng(O_i;\beta,\eta)$ \citep{hansen1982GMM}. Let $\widebar{g}(\beta,\eta) = \mathbb{E}_{P_{0,N}}[g(O;\beta,\eta)]$. The expectation of the GMM objective function is
    \begin{align*}
        &\mathbb{E}_{P_{0,N}}\widehat{g}^T(\beta,\eta_{0,N})W\widehat{g}(\beta,\eta_{0,N})/2= \underbrace{(1-N^{-1})\widebar{g}(\beta,\eta_{0,N})^TW(\widebar{g}(\beta,\eta_{0,N}))/2}_{\text{Signal}}+ \underbrace{\text{tr}(W\widebar{\Omega}(\beta,\eta_{0,N}))/(2N)}_{\text{Noise}}.
    \end{align*}
     The first term reflects the signal and is minimized at $\beta_0$. However, the second term (noise) generally is not, and under many weak asymptotics, the signal may not dominate the noise, which leads to bias in the GMM estimator. 
     
     One way to remove the bias introduced by the noise term is to set $W=\widebar{\Omega}(\beta,\eta_{0,N})^{-1}$, so that the noise term does not depend on $\beta$. As a result, the expectation of the objective function is minimized at $\beta_0$. In practice, $\widebar{\Omega}(\beta,\eta_{0,N})$ can be replaced by its empirical version $\widehat{\Omega}(\beta,\eta)=\frac{1}{N}\sum_{i=1}^N{g}(O_i;\beta,\eta){g}(O_i;\beta,\eta)^T$, as long as $\sup_{\beta \in \mathcal{B}}\Vert \widehat \Omega(\beta,\eta_{0,N}) - \widebar{\Omega}(\beta,\eta_{0,N})\Vert$ converges to $0$ sufficiently fast. This leads to the objective function of the CUE: $\widehat{g}^T(\beta,\eta_{0,N})\widehat{\Omega}(\beta,\eta_{0,N})^{-1}\widehat{g}(\beta,\eta_{0,N})/2$.  In the absence of nuisance parameters, this is the estimator in \cite{newey2009generalized}, which removes the bias due to many weak moments.  When nuisance parameters are estimated in a first step, additional bias may be introduced, which we address in the next section.



\subsection{Removing bias from estimated nuisance parameters}
\label{subsection: Global Neyman orthogonality}
In general, nuisance parameter estimation can affect the asymptotic behavior of the target parameter estimator. To mitigate this impact, we focus on moment functions satisfying the \emph{global Neyman orthogonality} condition \citep{chernozhukov2018double,foster2023aos}. 

Let $f$ be a generic function (not necessarily a moment function). We define the order-$k$ Gateaux derivative map $D^{(k)}_{t,\beta,f}:\tilde{\mathcal{T}}_N \rightarrow \mathbb{R}^{m}$ in the direction of $\eta$, evaluated at $\beta \in \mathcal{B}$ and $t\in[0,1]$, as $
    D^{(k)}_{t,\beta,f}(\eta) = \frac{\partial^{k}}{\partial t^k}\mathbb{E}_{P_{0,N}}\left[f\left(O;\beta,(1-t)\eta_{0,N}+t\eta\right)\right]$.

\begin{definition}[Neyman Orthogonality]\label{Definition: Neyman Orthogonality}\quad
    (a) A function $f(o;\beta,\eta)$ is said to be {Neyman orthogonal} on a set $\mathcal{B}_1\subset \mathcal{B}$ if, for any $\beta \in \mathcal{B}_1$, $N$, and $\eta \in \tilde{\mathcal{T}}_N$,  the first-order Gâteaux derivative satisfies $D^{(1)}_{0,\beta,f}(\eta)=0$. If this condition holds for all $\beta\in \mathcal{B}$, we say that $f(o;\beta,\eta)$ is {global Neyman orthogonal}.\\
    (b) If $f(o;\beta,\eta)$ is a moment function and is Neyman orthogonal on $\mathcal{B}_1$, we say that the corresponding moment condition $\mathbb{E}_{P_{0,N}}[f(O;\beta_0,\eta_{0,N})]=0$ is a Neyman orthogonal moment condition on $\mathcal{B}_1$. When $\mathcal{B}_1=\mathcal{B}$, we refer to this as a global Neyman orthogonal moment condition.
\end{definition}
Neyman orthogonality implies that nuisance parameter estimation does not introduce first-order bias into $\mathbb{E}_{P_{0,N}}[f(O;\beta,\eta_{0,N})]$, regardless of the value of $\beta$ within $\mathcal{B}_1$. Later, in Section 3, we assume that the moment function $g(o;\beta,\eta)$ is globally Neyman orthogonal, which is a stronger condition than the version defined in \citet{chernozhukov2018double}, where orthogonality is imposed only at the single point $\mathcal{B}_1=\{\beta_0\}$. 

We require this stronger, global version of Neyman orthogonality for two main reasons. First, as discussed in Sections 2.4 and 3.1, global orthogonality is necessary to ensure estimator consistency under many weak moment asymptotics. In contrast, in conventional settings with a fixed number of strong moments, Neyman orthogonality is typically not required for consistency. The key intuition is that many weak moments can amplify first-order bias from nuisance parameter estimation when the moment condition is not orthogonal.

Second, the asymptotic expansion of the estimator under weak identification includes a second-order term. To achieve asymptotic normality, we must control the influence of nuisance parameter estimation on this second-order term. A sufficient condition for such control is Neyman orthogonality of the derivative function $G(o;\beta,\eta)=\frac{\partial}{\partial \beta}g(o;\beta,\eta)$ in a neighborhood of $\beta_0$. This condition is implied by the global Neyman orthogonality of $g(o;\beta,\eta)$, due to the following permanence property:
\begin{lemma}[Permanence properties of Neyman orthogonality]     \label{Lemma: permanence of global neyman orthogonality}
(a) Suppose  $g(o;\beta,\eta)$ satisfies global Neyman orthogonality,   and that the partial derivative   $\frac{\partial g(o;\beta,\eta)}{\partial \beta}$ exists for all $\beta \in \mathcal{B}_1\subset \mathcal{B}$. If there exists an integrable function $h(o)$ such that 
 $\vert \frac{\partial g(o;\beta,\eta)}{\partial \beta}\vert\leq h(o)$ almost surely,  then $\frac{\partial g(o;\beta,\eta)}{\partial \beta}$ is Neyman orthogonal on $\mathcal{B}_1$. (b) Suppose  $g_1(o;\beta,\eta)$ and $g_2(o;\beta,\eta)$ are both Neyman orthogonal on $\mathcal{B}_1$, and let $h_1(\beta)$ and $h_2(\beta)$ be functions of $\beta$.  Then the linear combination $g_1(o;\beta,\eta)h_1(\beta)+g_2(o;\beta,\eta)h_2(\beta)$ also satisfies Neyman orthogonality on $\mathcal{B}_1$.
\end{lemma}



In practice, we often start with an initial moment function $\Tilde{g}(o;\beta,\tilde{\eta})$, where $\tilde{\eta}$ is a nuisance parameter associated with this initial formulation and has a true value $\tilde{\eta}_{0,N}$. Typically, $\tilde{\eta}_{0,N}$ is a subset of the full nuisance parameter $\eta_{0,N}$, and $\eta_{0,N}$ can be expressed as a statistical functional of the data-generating distribution. That is, we can write $\eta_{0,N} $ as $ \eta(P_{0,N})$. 
A common approach to orthogonize a given moment function $\Tilde{g}(o;\beta,\tilde{\eta})$ is to add its first-order influence function. To define this, consider a parametric submodel  $\mathcal{P}_{t,N}=\{(1-t)P_{0,N}+tP_{1,N}, t\in [0,1]\}$, where $P_{1,N}$ is chosen so that $\eta(P)$ is well-defined for all $P\in \mathcal{P}_{t,N}$. We denote the nuisance parameter under distribution $P$ as $\eta(P)$. The first-order influence function, denoted by $\phi$, can be defined as follows \citep{ichimura2022influence}: $\frac{\partial }{\partial t}\mathbb{E}_{P_{0,N}}[\tilde{g}(O;\beta,\tilde{\eta}(P_{t,N}))]\Big\vert_{t=0} = \int\phi(o;\beta,\eta_{0,N}){dP_{0,N}(o)}$. 
This is essentially the Gateaux derivative characterization of the influence function. The orthogonalized moment function can then be constructed as $g(o;\beta,\eta) = \Tilde{g}(o;\beta,\tilde{\eta})+\phi(o;\beta,\eta)$. According to Theorem 1 in \citet{Chernozhukov2022locallyrobust} (also provided in the Supplemental Materials), this construction yields 
 a global Neyman orthogonal moment condition.
Alternatively, one can construct moment conditions that are Neyman orthogonal by inspection and then directly verify that they satisfy the definition of global Neyman orthogonality. Several examples are provided in Section \ref{sec:examples}.

To allow the use of flexible classes of nuisance parameter estimations, one can relax the requirement for the complexity of function class using cross-fitting  \citep{van2011targeted, chernozhukov2018double}. The idea is to use sample splitting to ensure that the data used to estimate the nuisance parameters is independent of the data used for target parameter estimation. We adopt a version of cross-fitting combined with the CUE approach, resulting in the following two-step estimator:
\begin{description}
    \item[Step I]~We partition the observation indices $\{1,...,N\}$ into $L$ mutually exclusive groups $\{I_l, l = 1,...,L\}$, each of equal size. For each fold $l$, estimate the nuisance parameter $\widehat{\eta}_l$ using only the data $\{O_i, i \in I_l^c\}$, that is,   all observations not in $I_l$. Define $l(i)$ as the index of the group containing observation $i$, i.e., $l(i) = l$ such that $i\in I_l$. 
\item[Step II]~Let
    \begin{align*}
        &\widehat{g}(\beta,\widehat{\eta}) =\frac{1}{N} \sum_{i =1}^N g(O_i;\beta,\widehat{\eta}_{l(i)}),\widehat{\Omega}(\beta,\widehat{\eta})  = \frac{1}{N}\sum_{i =1}^N g(O_i;\beta,\widehat{\eta}_{l(i)})g(O_i;\beta,\widehat{\eta}_{l(i)})^T, \\
        &\widehat{Q}(\beta,\widehat{\eta}) = \widehat{g}^T(\beta,\widehat{\eta})^T\widehat{\Omega}(\beta,\widehat{\eta})^{-1}\widehat{g}(\beta,\widehat{\eta})/2.
    \end{align*}
    The target parameter is then estimated by solving the optimization problem: $\widehat{\beta} = \arg\min_{\beta \in \mathcal{B}}\widehat{Q}(\beta,\widehat{\eta})$. 
    We call $\widehat{\beta}$ the two-step CUE.
\end{description}


\label{sec: setup}
\section{General theory for many weak moments with nuisance parameters}

\subsection{ Key theoretical insights}
\label{subsection: Two key features}

Before presenting our formal theory, we first highlight the key theoretical insights that distinguish our framework.

\begin{enumerate}
    \item In Section 3.2, we establish consistency of the proposed estimator. Notably, unlike in the standard double machine learning literature, our consistency result requires the Neyman orthogonality condition. To understand the reason, observe that a crucial step in proving consistency is controlling the rate of $\frac{\sqrt{N}}{\mu_N}\sup_{\beta \in \mathcal{B}} \Vert \widehat{g}(\beta,\widehat{\eta})-\widehat{g}(\beta,\eta_{0,N})\Vert$ and showing it is $o_p(1)$. This term is related to the bias from estimating $\widehat{\eta}$. 

Without Neyman orthogonality, the bias in $\sup_{\beta \in \mathcal{B}} \Vert \widehat{g}(\beta,\widehat{\eta})-\widehat{g}(\beta,\eta_{0,N})\Vert$ is approximately {of order} $\sqrt{m}\delta_N$, where $\delta_N$ is the convergence rate of the nuisance parameter estimator. {The factor  $\sqrt{m}$ reflects the accumulation of bias across $m$ moment conditions.} As a result, 
    $\frac{\sqrt{N}}{\mu_N}\sup_{\beta \in \mathcal{B}} \Vert \widehat{g}(\beta,\widehat{\eta})-\widehat{g}(\beta,\eta_{0,N})\Vert$  becomes $\frac{\sqrt{Nm}\delta_N}{\mu_N}$. Under standard asymptotics where $m$ is fixed and $\mu_N = \sqrt{N}$, it suffices to have $\delta_N = o(1)$; that is, consistency of the nuisance estimator guarantees consistency of the target estimator.  However, under weak identification—particularly when $\sqrt{m}/\mu_N=C$ for some constant $C>0$ — this requirement becomes  $\delta_N = o(1/\sqrt{N})$, which is unachievable even in parametric models.
    
In contrast, when the moment function $g(o;\beta,\eta)$ is global Neyman orthogonal, the first-order bias from estimating $\widehat{\eta}$ vanishes, and we show that $\frac{\sqrt{N}}{\mu_N}\sup_{\beta \in \mathcal{B}} \Vert \widehat{g}(\beta,\widehat{\eta})-\widehat{g}(\beta,\eta_{0,N})\Vert=o_p(1)$  under reasonable conditions on the nuisance parameter convergence rate. This ensures consistency with many weak moments.
    
    \item 
    The asymptotic expansion under many weak moments involves a higher-order term that requires additional theoretical development in the presence of nuisance parameters. Specifically,
    the asymptotic expansion of our two-step CUE estimator yields the standard GMM influence function term (up to a scaling matrix) plus a second-order U-statistic term. While the U-statistic term is negligible under standard asymptotics, it becomes non-negligible under many weak moment asymptotics. Crucially, we must show that the contribution of nuisance parameter estimation to this U-statistic term is asymptotically negligible, in addition to its contribution to the first-order term. This requirement introduces new technical challenges absent from the standard double machine learning literature and specific to the weak identification setting.
\end{enumerate}


\vspace{-4mm}
\subsection{Consistency} \vspace{-2mm}
\label{subsection: consistency}
In this section, we state the regularity conditions for consistency of the two-step CUE estimator and establish the consistency result.

\begin{assumption}[Identifiability]
 We assume the following conditions for the moment condition to ensure identifiability: (a)  For all $\beta \in \mathcal{B}$, let $\delta(\beta) = S^T_N (\beta-\beta_{0})/\mu_N$. There is a constant $C>0$ such that $\Vert \delta(\beta) \Vert \leq C\sqrt{N} \Vert \widebar{g}(\beta,\eta_{0,N}) \Vert/\mu_N$. (b) There is a constant $C>0$ and a random variable $\widehat{M} = O_p(1)$ such that for all $\beta \in \mathcal{B}$, $\Vert \delta(\beta) \Vert \leq C\sqrt{N}\Vert \widehat{g}(\beta,\eta_{0,N})\Vert/\mu_N +\widehat{M}$.
\label{Assumption: global identifiability}
\end{assumption}

\vspace{-6mm}

Note that $\sqrt{N} \| \widebar{g}(\beta, \eta_{0,N}) \| / \mu_N = \sqrt{N} \| \widebar{g}(\beta, \eta_{0,N}) - \widebar{g}(\beta_0, \eta_{0,N}) \| / \mu_N$, since $\widebar{g}(\beta_0, \eta_{0,N}) = 0$. Therefore, Assumption \ref{Assumption: global identifiability}(a) states that if $\widebar{g}(\beta, \eta_{0,N})$ is close to $\widebar{g}(\beta_0, \eta_{0,N})$, then $\beta$ must be close to $\beta_0$. This provides a population-level identifiability condition. Similarly, we have
$
\sqrt{N} \| \widehat{g}(\beta, \eta_{0,N}) \| / \mu_N = \sqrt{N} \| \widehat{g}(\beta, \eta_{0,N}) - \widehat{g}(\beta_0, \eta_{0,N}) \| / \mu_N + O_p(1),
$
since $\| \widehat{g}(\beta_0, \eta_{0,N}) \| = O_p(1)$, as shown in the Supplemental Materials. Thus, Assumption \ref{Assumption: global identifiability}(b) implies that if $\widehat{g}(\beta, \eta_{0,N})$ is close to $\widehat{g}(\beta_0, \eta_{0,N})$, then $\delta(\beta)$ is bounded in probability.
Together, these conditions ensure global identification of $\beta_0$ \citep{newey2009generalized}.

\begin{assumption}[Moment Function Regularity and Convergence]
The moment function  $g(o;\beta,\eta)$ is continuous in $\beta$ and  and the following conditions hold:
(a) $\sup_{\beta\in \mathcal{B}}\mathbb{E}_{P_{0,N}}$ $ [\{g(O;\beta,\eta_{0,N})^Tg(O;\beta,\eta_{0,N})\}^2]/N \rightarrow 0$; (b) There exists a constant $C>0$ such that for all $\beta \in \mathcal{B}$, $1/C \leq \xi_{min}(\widebar{\Omega}(\beta,\eta_{0,N}))\leq \xi_{max}(\widebar{\Omega}(\beta,\eta_{0,N}))\leq C$, where $\xi_{\min}$ and $\xi_{\max}$ denote the smallest and largest eigenvalues, respectively. (c) $\sup_{\beta \in \mathcal{B}}\Vert \widehat{\Omega}(\beta,\eta_{0,N})-\widebar{\Omega}(\beta,\eta_{0,N})\Vert_F =o_{p}(1)$; (d) For all $a,b\in \mathbb{R}^m$ and $\beta, \beta'\in \mathcal{B}$, there exists a constant $C>0$ such that 
$\vert a^T \{\widebar{\Omega}(\beta',\eta_{0,N})-\widebar{\Omega}(\beta,\eta_{{0,N}})\}b\vert \leq C\Vert a\Vert \Vert b \Vert \Vert \beta'-\beta\Vert$; (e) For every constant $C'>0$, there exists a constant $C>0$ and a random variable $\widehat{M}=O_p(1)$ such that for all $\beta$,$\beta'\in \mathcal{B}$ with $\Vert \delta(\beta)\Vert\leq C'$ and $\Vert \delta(\beta')\Vert\leq C'$, we have $\sqrt{N}\Vert \overline{g}(\beta',\eta_{0,N})-\overline{g}({\beta},\eta_{0,N}) \Vert /\mu_N \leq C\Vert \delta({\beta}')-\delta(\beta)\Vert$ and $\sqrt{N}\Vert \widehat{g}(\beta',\eta_{0,N})-\widehat{g}({\beta},\eta_{0,N}) \Vert /\mu_N \leq \widehat{M}\Vert \delta({\beta}')-\delta(\beta)\Vert$.
\label{Assumption: convergence of g-hat}
\end{assumption}
Assumption \ref{Assumption: convergence of g-hat} imposes regularity conditions on the moment function when $\eta = \eta_{0,N}$, and restricts how the number of moment conditions $m$ may grow with the sample size $N$. If $g(o; \beta, \eta_{0,N})$ is uniformly bounded, then Assumption \ref{Assumption: convergence of g-hat}(a) holds whenever $m^2 / N = o(1)$. Assumption \ref{Assumption: convergence of g-hat}(b) requires that $\widebar{\Omega}(\beta, \eta_{0,N})$ is positive definite with eigenvalues uniformly bounded above and below. Assumption \ref{Assumption: convergence of g-hat}(c) is about the convergence of the sample covariance matrix $\widehat{\Omega}$ to its population counterpart $\widebar{\Omega}$. For any fixed $\beta$, this condition is satisfied under $m^2 / N = o(1)$ if all components of $g(O; \beta, \eta_{0,N})$ are uniformly bounded. Moreover, when $g(o; \beta, \eta_{0,N})$ has a simple structure—such as separability as in equation \eqref{Equation: separable score function}—this condition is straightforward to verify.
Finally, Assumptions \ref{Assumption: convergence of g-hat}(d) and (e) are also easy to verify in the examples we provide.

{
The next two assumptions concern nuisance parameter estimation and represent our new contributions for handling many weak moments. Assumption \ref{Assumption: Global Neyman orthogonality} imposes global Neyman orthogonality on the entire parameter space $\mathcal{B}$ and a continuous second-order Gateaux derivative condition. Assumption \ref{Assumption: consistency, nuisance estimation} specifies the convergence rates required for the nuisance parameter estimators.}

\begin{assumption}[Global Neyman Orthogonality and Continuity of Second-Order Gateaux Derivative] The moment function $g(o;\beta,\eta)$ satisfies global Neyman orthogonality. In addition, for all $\beta\in \mathcal{B}$, $t\in[0,1]$, and $\eta\in \tilde{\mathcal{T}}_N$, the second order Gateaux derivative $D^{(2)}_{\beta,t,g}(\eta)$ exists  and is continuous in $t$.
 \label{Assumption: Global Neyman orthogonality}
\end{assumption}

To introduce the next assumption, we first define the following notation. Let $\mathcal{G}^{(j)}_{\eta}=\{g^{(j)}(o;\beta,\eta), \beta \in \mathcal{B}\}$, $\mathcal{G}^{(j,k)}_{\eta}=\{G^{(j,k)}(o;\beta,\eta), \beta \in \mathcal{B}\}$, $\mathcal{G}^{(j,k,r)}_{\eta}=\left\{\frac{\partial^2}{\partial \beta^{(k)}\partial \beta^{(r)}}g^{(j)}(o;\beta,\eta), \beta \in\mathcal{B}\right\}$, where $g^{(j)}(o;\beta,\eta)$ denotes the $j$th coordinate of $g(o;\beta,\eta)$, and $G^{(j,k)}(o;\beta,\eta)$ denotes the $(j,k)$th entry of $G(o;\beta,\eta)$. For any two function classes $\mathcal{G}_1$ and $\mathcal{G}_2$, define their product class as
$\mathcal{G}_1\cdot\mathcal{G}_2 = \{g_1\cdot g_2, g_1\in \mathcal{G}_1, g_2\in \mathcal{G}_2\}$. Let $\mathcal{E}_\eta^{(j)}$ denote the envelope function of $\mathcal{G}_\eta^{(j)}$ as defined in \citet{chernozhukov2014gaussian}. Similarly, define envelope functions $\mathcal{E}_\eta^{(j,k)}$ and $\mathcal{E}_\eta^{(j,k,r)}$ for the other classes.

\begin{assumption}[Nuisance Parameter Estimation]
Let $I\subset \{1,\dots, N\}$ be a random subset of size $N/L$. Suppose the nuisance parameter estimator $\widehat{\eta}$, constructed using only observations outside of $I$,  belongs to a realization set $\mathcal{T}_N\subset \tilde{\mathcal{T}}_N$ with probability at least $1-\Delta_N$, where $\Delta_N=o(1)$. Assume that $\mathcal{T}_N$ contains the true value $\eta_{0,N}$ and satisfies the following conditions:
    \begin{enumerate}[label=(\alph*)]
    \item { Let $\mathcal{G}=\{\mathcal{G}_\eta^{(j)},\mathcal{G}_\eta^{(i)}\times\mathcal{G}_{\eta'}^{(j)},i,j=1,...,m,\eta,\eta'\in\mathcal{T}_N\}$ be a set of function classes. For each $\mathcal{F}\in \mathcal{G}$, let $F$ be an envelope function. 
    We assume that for all $0<\epsilon\leq 1$, $\sup_{Q}\log N(\epsilon\Vert F\Vert_{Q,2},\mathcal{F},\Vert \cdot \Vert_{Q,2})\leq v\log(a/\epsilon)$, where $a$ and $v$ are constants which does not depend on $\mathcal{F}$, $N(\cdot)$ denotes the covering number for $\mathcal{F}$ and the supreme is taken over all finitely discrete probability measures $Q$ \citep{chernozhukov2014gaussian}. Additionally, assume $\Vert F\Vert_{P_{0,N},q}<\infty$ for all $q\geq 1$.}
    \item Define the following statistical rates for all $j,k \in \{1,...,m\}$:
    \begin{align*}
        &r^{(j)}_N: = \sup_{\eta\in \mathcal{T}_N,\beta\in \mathcal{B}}\left(\mathbb{E}_{P_{0,N}}\left\vert g^{(j)}(O;\beta,\eta)-g^{(j)}(O;\beta,\eta_{0,N})\right\vert^2\right)^{1/2}, \\
        &\lambda^{(j)}_N:=\sup_{\beta\in \mathcal{B},t\in(0,1),\eta\in \mathcal{T}_N}\left\vert \frac{\partial^2}{\partial t^2} \mathbb{E}_{P_{0,N}}g^{(j)}(O;\beta,(1-t)\eta_{0,N}+t\eta))\right\vert,\\
        & r^{(j,k)}_N:= \sup_{\beta \in \mathcal{B},\eta\in \mathcal{T}_N}\left(\mathbb{E}_{P_{0,N}}\left|g^{(j)}(\beta,\eta)g^{(k)}(\beta,\eta)-g^{(j)}(\beta,\eta_{0,N})g^{(k)}(\beta,\eta_{0,N})\right|^2\right)^{1/2}.
        \end{align*}
        We assume there exits $\epsilon>0$ such that:
        \begin{align*}
        & a_N \Bigg(\sqrt{\log{(1/a_N})}+N^{\epsilon}\Bigg) + N^{-1/2+\epsilon}\left(\log (1/a_N)+1\right)\leq  \delta_N, \quad \text{for all $a_N \in \{r_N^{(j)},r_N^{(j,k)}\}$},\\
        & \lambda^{(j)}_N \leq N^{-1/2}\delta_N,\quad \delta_N = o(\min(\sqrt{N}/{m},\mu_N/\sqrt{m})).
    \end{align*}
    \end{enumerate}
    \label{Assumption: consistency, nuisance estimation}
\end{assumption}
    Assumption \ref{Assumption: consistency, nuisance estimation}(a) requires all function classes in $\mathcal{G}$ to be VC-type classes \citep{van1996weak,chernozhukov2014gaussian}, which is a standard and mild condition that restricts the complexity of the function classes. For example, the set of linear combinations of a finite number of fixed functions $\{\sum_{i=1}^p{\beta}^{(i)}f_i:\beta \in \mathcal{B}\}$ is a VC class \citep{van2000asymptotic}. Moreover, when both $\mathcal{G}_\eta^{(i)}$ and $\mathcal{G}_{\eta'}^{(j)}$ are bounded VC-type classes, their product $\mathcal{G}_\eta^{(i)}\times \mathcal{G}_{\eta'}^{(j)}$ is also a VC-type class. 

Assumption \ref{Assumption: consistency, nuisance estimation}(b) is a mild condition on the convergence rates of nuisance estimators. Here, $r_N^{(j)}$ and $r_N^{(j,k)}$ are the first-order bias from estimating the nuisance parameters, which are typically driven by the slowest-converging nuisance estimator among the set. Because $\epsilon$ can be chosen arbitrarily small, $\delta_N$  is effectively of the order $a_N\sqrt{\log(1/a_N)}$, which slightly exceeds $a_N$ and can be interpreted as the slowest convergence rate among the nuisance parameter estimators. In the examples presented in Section 4, the second derivative term $\lambda^{(j)}_N$ typically involves products of convergence rates for two or more nuisance estimators; see the Supplemental Materials for more details.  Under standard asymptotic settings where $m$ is finite and $\mu_N=\sqrt{N}$, it suffices for $\delta_N=o(\sqrt{N})$ and $\lambda_N^{(j)}=o(1)$, which aligns with the usual conditions for achieving doubly robust consistency under cross-fitting, where the final estimator remains consistent as long as one set of nuisance models converges sufficiently fast even if the other set does not converge to the true nuisance parameter.

Unlike the standard GMM estimator, the weighting matrix in the two-step CUE objective function $\widehat{\Omega}(\beta,\eta)$ is estimated and varies with $\beta$ and $\eta$.  As a result, additional assumptions are required to ensure its convergence.

\begin{assumption}[Uniform convergence of the weighting matrix]
The following conditions hold: (a) $\sup_{\beta \in \mathcal{B},\eta \in \mathcal{T}_N} \Big\Vert \widebar{\Omega}(\beta,\eta)-\widebar{\Omega}(\beta,\eta_{0,N})\Big\Vert = o(1)$; (b) $\sup_{\beta \in \mathcal{B}} \Big\Vert \widehat{\Omega}(\beta,\eta_{0,N})-\widebar{\Omega}(\beta,\eta_{0,N}) \Big\Vert = o_{p}(1)$.
    \label{Assumption: consistency, matrix estimation}
\end{assumption}

Assumption \ref{Assumption: consistency, matrix estimation}(a) imposes an additional constraint on the convergence rate of the nuisance parameter estimators, which generally needs to be verified on a case-by-case basis; see the examples in Section 4. 
 Assumption \ref{Assumption: consistency, matrix estimation}(b) is a mild condition on the convergence of $\widehat{\Omega}(\beta,\eta_{0,N})$ to its population counterpart. For any fixed $\beta$, a sufficient condition for this to hold is $m^2/N = o(1)$. While uniform convergence over $\beta \in \mathcal{B}$ is a slightly stronger requirement, it remains reasonable in many practical settings.

\begin{theorem}[Consistency] Under Assumptions 1-\ref{Assumption: consistency, matrix estimation} and $m^2/N\rightarrow 0$, we have $\delta(\widehat{\beta})=o_p(1)$, where $\delta(\beta) = S^T_N (\beta-\beta_{0})/\mu_N$.
    \label{Theorem: Consistency}
\end{theorem}

    The consistency result differs from the standard results in the double machine learning literature in two important ways. First, the conclusion $S_N^T(\widehat{\beta}-\beta_0)/\mu_N = o_p(1) $ implies that different linear combinations of $\beta$ converge at different rates, and also implies that $\Vert \widehat{\beta}-\beta_0\Vert = o_p(1)$. This makes the result stronger than the usual notion of consistency.  Second, our consistency result depends on the global Neyman orthogonality assumption due to the need to account for many weak moments.  In contrast, conventional estimators typically achieve consistency under strong identification as long as the nuisance parameters are estimated consistently.

\subsection{Asymptotic normality}
\label{subsection: ASN}
Additional assumptions are required to establish asymptotic normality. The next assumption is about the differentiability of $g(o;\beta,\eta)$ with respect to $\beta$. 
\begin{assumption}
    The function $g(o;\beta,\eta)$ is twice continuously differentiable with respect to $\beta$ in a neighborhood of $\beta_0$, denoted as $\mathcal{B}'$. \label{Assumption: differentiability}
\end{assumption}

The next assumption specifies convergence rate conditions for the nuisance estimators and strengthens Assumption \ref{Assumption: consistency, nuisance estimation}.

\begin{assumption}[Nuisance Parameter Estimation]
Let $I\subset \{1,\dots, N\}$ be a random subset of size $N/L$. Suppose the nuisance parameter estimator $\widehat{\eta}$, constructed using only observations outside of $I$,  belongs to a realization set $\mathcal{T}_N\subset \tilde{\mathcal{T}}_N$ with probability {at least} $1-\Delta_N$, where $\Delta_N=o(1)$. Assume that $\mathcal{T}_N$ contains the true value $\eta_{0,N}$ and satisfies the following conditions:
    \begin{enumerate}[label=(\alph*)]
    \item Let $\mathcal{G}'=\{\mathcal{G}_\eta^{(i)},\mathcal{G}_\eta^{(i,k)},\mathcal{G}_\eta^{(i,k,r)},\mathcal{G}_\eta^{(i)}\times\mathcal{G}_{\eta'}^{(j)},\mathcal{G}_\eta^{(i,k)}\times \mathcal{G}_{\eta'}^{(j,r)},\mathcal{G}_\eta^{(i)}\times\mathcal{G}_{\eta'}^{(j,k)}, i,j=1,...,m,k,r=1,...,p,\eta,\eta'\in\mathcal{T}_N\}$ be a set of function classes.
     We assume that for all $0<\epsilon\leq 1$, $\sup_{Q}\log N(\epsilon\Vert F\Vert_{Q,2},\mathcal{F},\Vert \cdot \Vert_{Q,2})\leq v\log(a/\epsilon)$, where $a$ and $v$ are constants which does not depend on $\mathcal{F}$. Additionally, assume $\Vert F\Vert_{P_{0,N},q}<\infty$ for all $q\geq 1$.
    \item Define the following statistical rates for all $j,k,l,r$:
    \begin{align*}
        &r^{(j)}_N: = \sup_{\eta\in \mathcal{T}_N,\beta\in \mathcal{B}}\left(\mathbb{E}_{P_{0,N}}\left\vert g^{(j)}(O;\beta,\eta)-g^{(j)}(O;\beta,\eta_{0,N})\right\vert^2\right)^{1/2} \\
        &\lambda^{(j)}_N:=\sup_{\beta\in \mathcal{B},t\in(0,1),\eta\in \mathcal{T}_N}\left\vert \frac{\partial^2}{\partial t^2} \mathbb{E}_{P_{0,N}}g^{(j)}(O;\beta,(1-t)\eta_{0,N}+t\eta)\right\vert\\
        & r^{(j,l)}_N:=  \sup_{\beta \in \mathcal{B},\eta\in \mathcal{T}_N} \left(\mathbb{E}_{P_{0,N}}\left|g^{(j)}(\beta,\eta)g^{(k)}(\beta,\eta)-g^{(j)}(\beta,\eta_{0,N})g^{(k)}(\beta,\eta_{0,N})\right|^2\right)^{1/2} \\
        & r^{(j),(k)}_N:= \sup_{\beta \in \mathcal{B}',\eta\in \mathcal{T}_N} \left(\mathbb{E}_{P_{0,N}}\left|\frac{\partial g^{(j)}(\beta,\eta)}{\partial \beta^{(l)}}-\frac{\partial g^{(j)}(\beta,\eta_{0,N})}{\partial \beta^{(l)}}\right|^2\right)^{1/2} \\
         & r^{(j,k,l)}_N:= \sup_{\beta \in \mathcal{B'},\eta\in \mathcal{T}_N} \left(\mathbb{E}_{P_{0,N}}\left|\frac{\partial^2g^{(j)}(\beta,\eta)}{\partial \beta^{(l)}\partial \beta^{(r)}}-\frac{\partial^2g^{(j)}(\beta,\eta_{0,N})}{\partial \beta^{(l)}\partial \beta^{(r)}}\right|^2\right)^{1/2}  \\
        & r^{(j),(k,l)}_N:= \sup_{\beta \in \mathcal{B}',\eta\in \mathcal{T}_N} \left(\mathbb{E}_{P_{0,N}}\left|g^{(j)}(\beta,\eta)\frac{\partial g^{(k)}(\beta,\eta)}{\partial \beta^{l}}-g^{(j)}(\beta,\eta_{0,N})\frac{\partial g^{(k)}(\beta,\eta_{0,N})}{\partial \beta^{(l)}}\right|^2\right)^{1/2} \\
        & r^{(j,k,l,r)}_N:= \sup_{\beta\in \mathcal{B}',\eta\in \mathcal{T}_N} \left(\mathbb{E}_{P_{0,N}}\left|\frac{\partial g^{(j)}(\beta,\eta)}{\partial\beta^{(r)}}\frac{\partial g^{(k)}(\beta,\eta)}{\partial \beta^{(l)}}-\frac{\partial g^{(j)}(\beta,\eta_{0,N})}{\partial\beta^{(r)}}\frac{\partial g^{(k)}(\beta,\eta_{0,N})}{\partial \beta^{(l)}}\right|^2\right)^{1/2}.
        \end{align*}
        We assume there exits $\epsilon>0$ such that:
        \begin{align*}
        & a_N \Bigg(\sqrt{\log{(1/a_N})}+N^{\epsilon}\Bigg) + N^{-1/2+\epsilon}\left(\log (1/a_N)+1\right)\leq  \delta_N, \\
        &\text{ for $a_N \in \{ r^{(j)}_{N},r_N^{(j,l)},r_N^{(j),(k)},r_N^{(j,k,l)},r_N^{(j),(k,l)},r_N^{(j,k,l,r)}\}$}, \quad \lambda^{(j)}_N \leq N^{-1/2}\delta_N,\quad  \delta_N = o(1/\sqrt{m}).
    \end{align*}
    \end{enumerate}
     \label{Assumption: general score regularity}
\end{assumption}

 Under standard asymptotics where $m$ is finite and $\mu_N=\sqrt{N}$, this assumption requires $\delta_N=o(1)$ and $\lambda^{(j)}_N=o(1/\sqrt{N})$, which aligns with the typical rate conditions for double robustness \citep{chernozhukov2018double,andrea2021BKA}. Under the many weak moments asymptotic regime, however, stricter requirements are needed:  $\delta_N=o(1/\sqrt{m})$ and $\lambda^{(j)}_N\leq\delta_N/\sqrt{N}$. These conditions imply that each nuisance estimator must converge faster than $1/\sqrt{m}$, and that any product of two such convergence rates must be faster than $1/\sqrt{mN}$, reflecting the requirements needed to ensure valid inference in the presence of many weak moment conditions.

Define
\begin{align*}
    &\Omega^{(k)}(o;\beta,\eta) = g(o;\beta,\eta)\frac{\partial g(o;\beta,\eta)^T}{\partial \beta^{(k)}} ,\Omega^{(kl)}(o;\beta,\eta) = g(o;\beta,\eta)\frac{\partial g(o;\beta,\eta)^T}{\partial \beta^{(k)}\partial \beta^{(l)}}, \\
    &\Omega^{(k,l)}(o;\beta,\eta) = \frac{\partial g(o;\beta,\eta)^T}{\partial \beta^{(k)}}\frac{\partial g(o;\beta,\eta)^T}{\partial \beta^{(l)}}.
\end{align*}

\begin{assumption}
    The following conditions holds for $A(o;\beta,\eta) \in \{\Omega, \Omega^k, \Omega^{kl}, \Omega^{k,l}\}$: (a) $\sup_{\beta \in \mathcal{B}',\eta \in \mathcal{T}_N} \Big\Vert \widebar{A}(O;\beta,\eta)-\widebar{A}(O;\beta,\eta_{0,N}) \Big\Vert = o_{p}(1/\sqrt{m})$; (b) $\sup_{\beta \in \mathcal{B}'} \Big\Vert \widehat{A}(\beta,\eta_{0,N})-\widebar{A}(O;\beta,\eta_{0,N}) \Big\Vert = o_{p}(1/\sqrt{m})$.\label{Assumption: ASN matrix estimation} 
\end{assumption}
Compared to Assumption \ref{Assumption: consistency, matrix estimation}, Assumption \ref{Assumption: ASN matrix estimation} imposes stronger conditions on nuisance parameter convergence and the estimation accuracy of high-dimensional matrices.

There are several additional technical assumptions, including conditions on moments, convergence of the second-order derivatives of the objective function, and uniform convergence and smoothness conditions for $\Omega,\Omega^{k},\Omega^{kl},\Omega^{k,l}$. See Assumptions S1-S2 in the Supplemental Material for details.

{ Building on the general setting above, we also consider a special case where the moment function is \emph{separable}, under which the theoretical conditions simplify.} A moment function $g(o;\beta,\eta)$ is said to be separable if it can be written in the form 
\begin{align}
    g(o;\beta,\eta) = \sum_{b=1}^B g^{[b]}(o;\eta)h^{[b]}(\beta), \label{Equation: separable score function}
\end{align}
where each $g^{[b]}(o;\eta), $ for $b = 1,...,B$, is an $m\times q$ matrix that depends only on the nuisance parameter $\eta$, and each $h^{[b]}(\beta)$ is a $q\times 1$ vector-valued function that depends only on the parameter of interest $\beta$. The functions $h^{[b]}(\beta)$ are assumed to be twice continuously differentiable. We refer to functions of the form \eqref{Equation: separable score function} as separable moment functions because the nuisance parameters and the parameter of interest enter the model in a structurally separated way. { Simplified assumptions for separable moment functions, which are easier to verify, are provided in Supplemental Material S2.}

\begin{theorem}
    Suppose Assumptions \ref{Assumption: weak moment condition}-\ref{Assumption: ASN matrix estimation} and Assumptions S1-S2 in the Supplemental Material,  hold for a general moment function, or Assumptions \ref{Assumption: weak moment condition}-\ref{Assumption: convergence of g-hat}, \ref{Assumption: consistency, matrix estimation},\ref{Assumption: ASN matrix estimation}, and Assumptions S1-S3 in the Supplemental Material hold for a separable moment function. If  $m^3/N \rightarrow 0$, then we have
    \begin{equation*}
        S_N^{T}(\widehat{\beta}-\beta_0)\rightsquigarrow N(0,V),
    \end{equation*}
    where $V=H^{-1}+H^{-1}\Lambda H^{-1}$, $S_N^{-1}\mathbb{E}_{P_{0,N}}[(U^\perp)^{T}\Omega^{-1}U^\perp](S^{-1}_N)^{T}{\to} \Lambda$,  $H$ is defined in Assumption \ref{Assumption: weak moment condition}, and $\rightsquigarrow$ denotes weak convergence. The components $U^\perp$ are given by $U^{\perp,(k)} = G^{(k)}(O;\beta_0,\eta_{0,N})-\widebar{G}^{(k)}-\mathbb{E}_{P_{0,N}}[{G}^{(k)}(O;\beta_0,\eta_{0,N})g^T(O;\beta_0,\eta_{0,N})]\widebar{\Omega}g(O;\beta_0,\eta_{0,N})$, $U^\perp = [U^{\perp,(1)},...,U^{\perp,(p)}]$.\label{Theorem: consistency and ASN}
\end{theorem}
{The columns of $U^\perp$ are the population least squares residuals from  regressing $G(O;\beta_0,\eta_{0,N})-\widebar{G}$ on $g(O;\beta_0,\eta_{0,N})$.} 
Theorem \ref{Theorem: consistency and ASN} shows that the asymptotic distribution of the proposed two-step CUE estimator is the same as that of $\arg\min_{\beta}\widehat{Q}(\beta,\eta_{0,N})$ \citep{newey2009generalized}. In other words, the asymptotic distribution is unaffected by whether the nuisance parameter is known or estimated. 
While this property also appears in the double machine learning literature when using an estimating equation or GMM in the second step \citep{chernozhukov2018double,Chernozhukov2022locallyrobust}, 
our result extends it to the case where CUE is used as the second step under many weak moment asymptotics. Importantly, the required conditions and technical arguments differ significantly, since we must account for the impact of nuisance estimation on the U-statistic term and ensure it remains asymptotically negligible. In Supplemental Material S3, we provide a consistent variance estimator.

\subsection{Over-identification test}
We propose an over-identification test as a diagnostic tool, extending the traditional $J$-test for over-identification \citep{hansen1982GMM}. In practice, it is often important to test  whether all moment conditions are satisfied simultaneously—that is, whether $\mathbb{E}_{P_{0,N}}[g(O; \beta_0, \eta_{0,N})] = 0$. The null hypothesis is violated if at least one moment condition does not hold. Our version of the over-identification test is given below.
\begin{theorem}[Over-identification test] Assume all assumptions in Theorem \ref{Theorem: consistency and ASN} hold. Under the null hypothesis $H_0:\mathbb{E}_{P_{0,N}}[g(O;\beta_0,{\eta}_{0,N})]=0$, we have
    \begin{align*}
        P_{0,N}(2N\widehat{Q}(\widehat{\beta},\widehat{\eta})\geq \chi^2_{1-\alpha}(m-p))\rightarrow \alpha
    \end{align*}
    as $N\rightarrow \infty$, where $\chi^2_{1-\alpha}(m-p)$ denotes the $1-\alpha$ quantile of the chi-square distribution with  $m-p$ degree of freedom.
    \label{theorem: over-identification test}
\end{theorem}

The over-identification test can be viewed as a model specification test that assesses whether the moment condition model is correctly specified. For example, when moment conditions are based on instrumental variables, this test serves as a diagnostic for instrument validity. However, rejection may arise from other forms of misspecification, such as incorrect functional forms or invalid nuisance parameter estimation.




\section{Examples}\label{sec:examples}
In this section, we present three examples from causal inference where we construct globally Neyman orthogonal moment conditions and apply our general theory to each example.

\subsection{Additive structural mean model with many instrumental variables} 
{Let the full data be denoted by $\{\tilde{O}_i=((Y_i(a),a\in\mathcal{A}),A_i,Z_i,X_i,U_i), i = 1,...,N\}$, which are i.i.d. draws from the true data-generating distribution ${P}_{0,N}$}. Here, $Y$ is a continuous outcome, $A$ denotes treatment (which may be continuous or categorical) with support $\mathcal{A}$, $Z$ is an $m$-dimensional vector of instrumental variables, $X$ represents baseline covariates, and $U$ are unobserved confounders between $A$ and $Y$. The observed data consist of ${O_i = (Y_i, A_i, Z_i, X_i), i = 1, \dots, N}$. {In this example and the following two examples, we slightly abuse the notation $P_{0,N}$ by also using it to denote the observed data distribution implied by the true data-generating process.}

Let $Y(z,a)$ be the potential outcome \citep{rubin1972causal} under instrument value $z$ and treatment value $a$. To construct moment conditions, we assume (a) consistency: $Y=Y(z,a)$ if $A=a,Z=z$; (b) exclusion restriction: $Y(z,a) = Y(a)$, therefore $Y(z,a)$ can be written as $Y(a)$; (c) latent ignorability: $Y(a)\indep (A,Z)|X,U$; (d) IV independence:  $Z\indep U|X$; and (e) IV relevance: $Z$ is associated with $A$ given $X$. Under these assumptions, a commonly used model is the additive structural mean model (ASMM) \citep{robins1994snmm,stijn2014snmm}:
\begin{equation*}
    \mathbb{E}_{{P}_{0,N}}[Y(a)-Y(0)|A=a,Z=z,X=x,U=u]=\gamma(a,z,x;\beta_0),
\end{equation*}
where $\{\gamma(a,z,x;\beta),\beta \in \mathcal{B}\}$ is a user-specified function class with $\gamma(0,z,x;\beta)=0$ and $\gamma(a,z,x;0)=0$, and $\mathcal{B}$ is a compact subset of $\mathbb{R}^p$. This formulation implicitly assumes that the unobserved confounder $U$ does not modify the treatment effect. For ease of exposition, we consider the simple specification $\gamma(a, z, x; \beta)=\beta a$. In this case, the orthogonal moment function can be defined as ${g}(O;\beta,\eta) = (Y-\eta_Y(X)-\beta A +\beta \eta_A(X))(Z-\eta_Z(X))$,
where $\eta=(\eta_Z,\eta_Y,\eta_A)$. Here, $\eta_Y$ and $ \eta_A$ are functions of $X$ with true values $\eta_{Y,0,N}(X) =  \mathbb{E}_{P_{0,N}}[Y|X]$ and $\eta_{A,0,N}(X) = \mathbb{E}_{P_0,N}[A|X]$, $\eta_Z = (\eta_{Z^{(1)}},...,\eta_{Z^{(m)}})$ with true values $\eta_{Z^{(k)},0, N}(X) =  \mathbb{E}_{P_{0,N}}[Z^{(k)}|X]$, for $k = 1,...,m$. {
Under ASMM, $\widebar{G} = -\mathbb{E}_{P_{0,N}}[(A-\eta_{A,0,N}(X))(Z-\eta_{Z,0,N}(X))]$. Thus, our many weak moment asymptotic regime corresponds to weak conditional correlation between $A$ and $Z$ given $X$.
}

{
To satisfy the convergence rate conditions on nuisance parameter estimation in Assumption \ref{Assumption: general score regularity}, the following conditions are needed.} Let $I\subset \{1,\dots, N\}$ be a random subset of size $N/L$. Suppose the nuisance parameter estimator $\widehat{\eta}$, constructed using only observations outside of $I$, satisfies the following conditions: with $P_{0,N}$- probability  at least $1-\Delta_N$, for $j=1,...,m, l = 1,...,L$, $\Vert \widehat{\eta}_{A,l}-{\eta}_{A,0,N}\Vert_{P_{0,N},2}\leq \delta_N$, $\Vert \widehat{\eta}_{Y,l}-{\eta}_{Y,0,N}\Vert_{P_{0,N},2}\leq \delta_N$, $\Vert \widehat{\eta}_{Z^{(j)},l}-{\eta}_{Z^{(j)},0,N}\Vert_{P_{0,N},2}\leq \delta_N/m$, $\Vert \widehat{\eta}_{Y,l}-{\eta}_{Y,0,N}\Vert_{P_{0,N},\infty}\leq C$, $\Vert \widehat{\eta}_{A,l}-{\eta}_{A,0,N}\Vert_{P_{0,N},\infty}\leq C$, $\Vert \widehat{\eta}_{Z^{(j)},l}-{\eta}_{Z^{(j)},0,N}\Vert_{P_{0,N},\infty}\leq C$, $(\Vert \widehat{\eta}_{A,l}-{\eta}_{A,0,N}\Vert_{P_{0,N},2}+\Vert \widehat{\eta}_{Y,l}-{\eta}_{Y,0,N}\Vert_{P_{0,N},2})\times \Vert \widehat{\eta}_{Z^{(j)},l}-{\eta}_{Z^{(j)},0,N}\Vert_{P_{0,N},2}\leq N^{-1/2}\delta_N$, where $\delta_N=o(1/\sqrt{m})$, $\Delta_N=o(1)$,  and $C>0$ is a constant. Together with other regularity conditions specified in Supplemental Material S4.1, this condition provides sufficient conditions for the assumptions in Theorem \ref{Theorem: consistency and ASN}, thereby ensuring consistency and asymptotic normality of the two-step CUE under ASMM.

These rate conditions impose mild requirements on the nuisance estimator convergence rates.  Suppose $m$ is $O(N^{1/3-\epsilon})$ for some $0<\epsilon\leq 1/3$. Then the required convergence rate for each instrument nuisance estimator  $\widehat{\eta}_{Z^{(j)},l}$ is $o(N^{-1/2+3\epsilon/2})$. When $\epsilon$ is close to zero, an approximately parametric rate is needed due to the high-dimensionality of $Z$. In contrast, the convergence rate requirements for estimating  $\eta_{A,0,N}$ and $\eta_{Y,0,N}$ are much milder and can be as slow as $o(N^{{-1/6+\epsilon/2}})$.

\subsection{Multiplicative structural mean model with many instrumental variables}
Next, we give an example where the moment condition is nonlinear in the  parameter of interest. Building on the same notation and setup as the ASMM example, we now consider the case where $A \in \{0,1,...,K\}$ is an ordinal treatment.  We consider the following  multiplicative structural mean model (MSMM): 
\begin{equation*}
    \log \frac{\mathbb{E}_{P_{0,N}}[Y(a)|A=a,Z=z,U=u,X=x]}{\mathbb{E}_{P_{0,N}}[Y(0)|A=a,Z=z,U=u,X=x]} =\gamma(a,z,x;\beta_0).
\end{equation*}
Here, $\gamma(a,z,x;\beta)$ is again a user-specified function satisfying $\gamma(0,z,x;\beta)=0$. A common and widely used choice is $\gamma(a,z,x;\beta)=\beta a$, though effect modification by covariates can also be incorporated. The goal is to estimate $\beta_0$, which encodes the conditional average causal effect on the multiplicative scale. As in the ASMM setting, an initial moment condition can be defined by $\mathbb{E}_{P_{0,N}}[\tilde{g}(O;\beta_0,\eta_{Z,0,N})]=0$, where $\tilde{g}(O;\beta_0,\eta_{Z,0,N}) = Y\exp(-\beta_0 A)(Z-\eta_{Z,0,N}(X))$, and $\eta_{Z,0,N}$is defined as before. A global orthogonal form of the moment condition is
\footnotesize
\begin{equation*}    \mathbb{E}_{P_{0,N}}\left[\Big(Y\exp(-\beta A)-\sum_{a\in \{0,1,..,K\}}\mathbb{E}_{P_{0,N}}[Y|X,A=a]\exp(-\beta a)P_{0,N}(A=a|X)\Big)(Z-\mathbb{E}_{P_{0,N}}[Z|X])\right] = 0.
\end{equation*}
\normalsize
For ease of presentation, we focus on the case where $A$ is binary. In this setting, the orthogonal moment function is given by
\begin{equation*}
    g(o;\beta,\eta) = \Big(Y\exp(-\beta A)-\sum_{a\in \{0,1\}}\eta_{Y,(a)}(X)\exp(-\beta a)\eta_{A,(a)}(X)\Big)(Z-\eta_{Z}(X)),
\end{equation*}
where $\eta = (\eta_{Y,(0)},\eta_{Y,(1)},\eta_{A,(1)},\eta_Z)$, with true values $\eta_{Y,(a),0,N}(a,X) = \mathbb{E}_{P_{0,N}}[Y|A=a,X]$, $\eta_{A,(1),0,N}(X) = P_{0,N}(A=1|X)$, and $\eta_{Z,0,N}(X) = \mathbb{E}_{P_{0,N}}[Z|X]$.  The convergence rate conditions for the nuisance parameter estimators are similar to those in the ASMM example. Sufficient conditions for Theorem \ref{Theorem: consistency and ASN} to hold for the MSMM example are provided in Supplemental Material S4.2.

\subsection{Proximal causal inference}

In this section, we present a structural mean model approach for proximal causal inference with many weak treatment proxies, extending the method of \citet{ett2024proximal} to accommodate high-dimensional proxy variables.

Let $P_{0,N}$ denote the true data-generating distribution. Let $A$ be a continuous treatment with reference level $0$, $Y$ a continuous outcome, $X$ a set of observed confounders, $U$ a set of unmeasured confounders, $Z$ an $m$-dimensional vector of treatment proxies, and $W$ a one-dimensional outcome proxy. We assume the following proximal structural mean model (PSMM) holds for the potential outcomes: $\mathbb{E}_{P_{0,N}}[Y(a)-Y(0)|Z,A=a,U,X] = \beta_{a,0} a$. To construct moment conditions for identifying $\beta_{a,0}$, we impose the following assumptions:
\begin{assumption}[Identification of proximal structural mean model] 
    (a) (Latent ignorability for treatment and treatment proxies) $(Y(0),W)\indep (A,Z)|U,X$. (b) (Outcome proxy) There exists a scalar $\beta_{w,0}$ such that $\mathbb{E}_{P_{0,N}}[Y(0)-\beta_{w,0} W|U,X] =  \mathbb{E}_{P_{0,N}}[Y(0)-\beta_{w,0} W|X]$.
    \label{Assumption: proxy_identification}
\end{assumption}

Assumption \ref{Assumption: proxy_identification}(a) is a causal identification assumption, which is also assumed in \citet{ett2024proximal} and \citet{liu2024regression}. It states that, conditional on both the observed covariates and the unmeasured confounders, the joint distribution of treatment and treatment proxies $(A,Z)$ is independent of the outcome and outcome proxy $(Y,W)$.
Assumption \ref{Assumption: proxy_identification}(b) can be justified under the following structural equation models:
\begin{align*}
    & \mathbb{E}_{P_{0,N}}[Y|A,Z,U,X] = \beta_0+\beta_{a,0}A+\beta_{u,0}U+f_1(X), \\
    & \mathbb{E}_{P_{0,N}}[W|A,Z,U,X] = \alpha_0+\alpha_{u,0}U+f_2(X),
\end{align*}
where $f_1$ and $f_2$ are unknown functions of $X$. Under this model, the outcome proxy assumption is satisfied with $\beta_{w,0} = \beta_{u,0}/\alpha_{u,0}$. This setting generalizes the models in \citet{ett2024proximal} and \citet{liu2024regression}, which assume that $f_1$ and $f_2$ are linear in $X$. Based on these assumptions, we establish identification of both $\beta_{a,0}$ and $\beta_{w,0}$.  

  Under Assumption \ref{Assumption: proxy_identification}, the following moment condition holds: $\mathbb{E}_{P_{0,N}}[g(O;\beta_0,\eta_0)] = 0$. The global Neyman orthogonal moment function is defined as 
\begin{align*}
g(O;\beta,\eta)=((A-\eta_A(X)),(Z-\eta_Z(X)))^T(Y-\beta_{a} A-\beta_{w}W-\eta_{Y}(X)+\beta_{a}\eta_{A}(X)+\beta_{w}\eta_{W}(X)), \label{Equation: Moment function for PSMM}
\end{align*}
where $\beta = (\beta_a,\beta_w)$ with true value $\beta_0 = (\beta_{a,0},\beta_{w,0})$, and $\eta = (\eta_Y,\eta_W,\eta_A,\eta_{Z})$ with true values $\eta_{Y,0,N}(X) = \mathbb{E}_{P_{0,N}}[Y|X]$, $\eta_{A,0,N}(X) = \mathbb{E}_{P_{0,N}}[A|X]$, $\eta_{W,0,N}(X) = \mathbb{E}_{P_{0,N}}[W|X]$, $\eta_{Z,0,N}(X) = \mathbb{E}_{P_{0,N}}[Z|X]$. In this example, the matrix $\widebar{G}$ is given by
\begin{align*}
    \widebar{G} = 
\mathbb{E}_{P_{0,N}}\left[((A-\mathbb{E}_{P_{0,N}}[A|X]),(Z-\mathbb{E}_{P_{0,N}}[Z|X]))^T(-(A-\mathbb{E}_{P_{0,N}}[A|X]),-(W-\mathbb{E}_{P_{0,N}}[W|X]))\right].
\end{align*}
Thus, Assumption \ref{Assumption: weak moment condition} implies that the conditional correlations between $Z$ and $A$, as well as between $Z$ and $W$, are weak given $X$. This suggests that the conditional correlation between $Z$ and the unmeasured confounder $U$ is also weak, meaning that $Z$ serves as a weak proxy for $U$.

The convergence rate conditions for the nuisance parameter estimators are similar to those in the ASMM example. Sufficient conditions for Theorem \ref{Theorem: consistency and ASN} to hold for the PSMM example are provided in Supplemental Material S4.3.


\section{Simulation}
\subsection{Additive and multiplicative structural mean models}

The data-generating processes for the ASMM and MSMM simulations are described below:

\noindent \textbf{Sample sizes:} We consider $5$ different sample sizes: $N=1000$, $2000$, $5000$, $10000$, $20000$.

\noindent  \textbf{Baseline covariates:} We consider a three-dimensional covariate vector $\boldsymbol{X}=(X_1, X_2,X_3)$ from a multivariate normal distribution with mean zero and covariance matrix $I_3+0.2J_3$, truncated to lie within the interval [-4, 4] for each component. Here, $I_3$ is the $3\times 3$ identity matrix and $J_3$ is a $3\times 3$ matrix with all entries equal to $1$. 

\noindent \textbf{Instrumental variable:}
The dimension of the instrumental variable $m$ is set as $m=4N^{-1/4}$, yielding values $m=22, 26, 33, 40, 47$ for the respective sample sizes. Each component $Z_j$ is generated using the model $Z_j = X_1+X_2+X_3+\epsilon_j$, $j = 1,...,m$, where $\epsilon_j \sim \text{uniform}[-3,3]$ independently.

\noindent \textbf{Treatment:}
We first generate a latent continuous treatment variable $A^*$ from the model $A^*=\alpha\sum_{j=1}^m N^{-1/2-j/(3m)} Z_j+X_1+\text{sin}(X_2)+\text{expit}(X_3)+U+N(0,1)$, where $U\sim \text{Uniform}[-4,4]$ is an unmeasured confounder. For ASMM, we set $\alpha = 1$ and define the observed treatment as $A=A^*$.  For MSMM, we set $\alpha = 1.5$ and $A = I(A^*>0.6)$. $N(0,1)$ is a normal error term with mean 0 and variance 1. Each instance of $N(0,1)$ in the following equations should be interpreted as an independent standard normal error term.

\noindent \textbf{Outcome:} 
For ASMM, we generate outcomes using the following model $Y = 1+3A-X_1+\sin X_2-X_3^2+X_2X_3+U+N(0,1)$. For MSMM, the potential outcomes and the observed outcome are generated as $Y(0) = 1+0.5X_1+0.5X_2-0.5X_3+0.5I(U>0.5) + 0.5N(0,1)$,$Y(1) = Y(0)\exp(1)$, $ Y=AY(1)+(1-A)Y(0)$.

The data-generating distributions are designed to reflect weak identification. Specifically, the effect of each instrument $Z_j$ on the latent treatment $A^*$ is given by $\alpha N^{-1/2-j/(3m)}$, which decreases as the sample size $N$ increases. This setup ensures that the instruments become progressively weak in larger samples. Moreover, instrument strength varies across $j$, allowing for heterogeneity in instrument relevance. We also introduce nonlinear terms in the generation of both $A^*$ and $Y$ as the proposed two-step CUE method can deal with nonlinearity in the nuisance models.

For each simulated dataset, we construct the proposed estimator with all nuisance functions estimated using the \texttt{R} package \texttt{SuperLearner} \citep{van2007super}. We include the random forest and generalized linear models in the SuperLearner library. We compare our method with the standard two-step GMM estimator using the optimal weighting matrix, which can be obtained the following two steps. First, obtain an initial estimate $\Tilde{\beta}_{\rm GMM}$ by minimizing the objective function $\widehat{Q}_{\text{GMM}}(\beta,\widehat{\eta}) := \widehat{g}^T(\beta,\widehat{\eta})^T\widehat{g}(\beta,\widehat{\eta})/2$; Second, obtain the final GMM estimate $\widehat{\beta}_{\rm GMM}$ by minimizing the objective function $\widehat{Q}_{\text{GMM}}(\beta,\widehat{\eta}) := \widehat{g}^T(\beta,\widehat{\eta})^TW\widehat{g}(\beta,\widehat{\eta})/2$, where the weighting matrix is calculated using the optimal choice $W = \widehat{\Omega}^{-1}(\Tilde{\beta}_{\rm GMM},\widehat{\eta})$ \citep{hansen1982GMM,chernozhukov2018double}. The nuisance estimators are constructed in the same way as for the proposed two-step CUE. Under strong identification and a fixed number of moment conditions, GMM is known to be asymptotically unbiased and normally distributed \citep{Chernozhukov2022locallyrobust}.

Table \ref{tab: simulation results for ASMM and MSMM} shows the simulation results for ASMM and MSMM.  In all settings, the two-step CUE consistently exhibits smaller bias and more accurate confidence interval coverage. In contrast, the two-step GMM shows larger positive bias and confidence intervals with under-coverage. The standard errors for CUE are generally larger than those for GMM, which is expected because the GMM variance formula omits the U-statistics term.
The rejection rates for the over-identification test are close to the nominal type I error level.

\begin{table}[htbp]
  \centering
  \scriptsize
  \caption{Simulation results for the ASMM, MSMM and PSMM. The column {rejection rate} reports the empirical type I error of the over-identification test at the nominal 0.05 level. Estimate presents the average point estimates. Bias reports the average estimation bias. SE is the average of the estimated standard errors, while SD refers to the empirical standard deviation of the estimates across 1000 simulations. Coverage rate indicates the empirical coverage probability of the 95\% Wald confidence interval. CUE, GMM and 2SLS denote the proposed two-step continuously updated estimator, the two-step generalized method of moments estimator and the two-stage least square estimator proposed by \citet{ett2024proximal}, respectively. For PSMM, we show the simulation results for $\beta_a$ and $\beta_w$ separately.}

    \begin{tabular}{cccccccccc}
    \toprule
          & N     & m     & rejection rate & Method & Estimate & Bias  & SD    & SE  & Coverage rate \\
    \midrule
    \multirow{10}[10]{*}{ASMM} & \multirow{2}[2]{*}{1000} & \multirow{2}[2]{*}{22} & \multirow{2}[2]{*}{0.067} & CUE   & 2.992 & -0.008 & 0.087 & 0.087 & 0.952 \\
          &       &       &       & GMM   & 3.067 & 0.067 & 0.072 & 0.076 & 0.801 \\
\cmidrule{2-10}          & \multirow{2}[2]{*}{2000} & \multirow{2}[2]{*}{26} & \multirow{2}[2]{*}{0.063} & CUE   & 2.994 & -0.006 & 0.063 & 0.063 & 0.955 \\
          &       &       &       & GMM   & 3.046 & 0.046 & 0.056 & 0.057 & 0.847 \\
\cmidrule{2-10}          & \multirow{2}[2]{*}{5000} & \multirow{2}[2]{*}{33} & \multirow{2}[2]{*}{0.053} & CUE   & 3.001 & 0.001 & 0.042 & 0.043 & 0.952 \\
          &       &       &       & GMM   & 3.032 & 0.032 & 0.039 & 0.041 & 0.851 \\
\cmidrule{2-10}          & \multirow{2}[2]{*}{10000} & \multirow{2}[2]{*}{40} & \multirow{2}[2]{*}{0.048} & CUE   & 2.999 & -0.001 & 0.031 & 0.031 & 0.954 \\
          &       &       &       & GMM   & 3.021 & 0.021 & 0.029 & 0.03  & 0.868 \\
\cmidrule{2-10}          & \multirow{2}[2]{*}{20000} & \multirow{2}[2]{*}{47} & \multirow{2}[2]{*}{0.061} & CUE   & 2.999 & -0.001 & 0.023 & 0.023 & 0.949 \\
          &       &       &       & GMM   & 3.014 & 0.014 & 0.022 & 0.022 & 0.894 \\
    \midrule
    \multirow{10}[10]{*}{MSMM} & \multirow{2}[2]{*}{1000} & \multirow{2}[2]{*}{22} & \multirow{2}[2]{*}{0.037} & CUE   & 0.899 & -0.101 & 0.736 & 1.919 & 0.949 \\
          &       &       &       & GMM   & 1.44  & 0.44  & 15.282 & 0.256 & 0.248 \\
\cmidrule{2-10}          & \multirow{2}[2]{*}{2000} & \multirow{2}[2]{*}{26} & \multirow{2}[2]{*}{0.056} & CUE   & 0.947 & -0.053 & 0.281 & 0.793 & 0.965 \\
          &       &       &       & GMM   & 1.334 & 0.334 & 1.064 & 0.171 & 0.238 \\
\cmidrule{2-10}          & \multirow{2}[2]{*}{5000} & \multirow{2}[2]{*}{33} & \multirow{2}[2]{*}{0.039} & CUE   & 0.977 & -0.023 & 0.143 & 0.136 & 0.966 \\
          &       &       &       & GMM   & 1.206 & 0.206 & 0.088 & 0.108 & 0.396 \\
\cmidrule{2-10}          & \multirow{2}[2]{*}{10000} & \multirow{2}[2]{*}{40} & \multirow{2}[2]{*}{0.045} & CUE   & 0.991 & -0.009 & 0.101 & 0.1   & 0.953 \\
          &       &       &       & GMM   & 1.155 & 0.155 & 0.072 & 0.084 & 0.432 \\
\cmidrule{2-10}          & \multirow{2}[2]{*}{20000} & \multirow{2}[2]{*}{47} & \multirow{2}[2]{*}{0.031} & CUE   & 0.996 & -0.004 & 0.072 & 0.072 & 0.947 \\
          &       &       &       & GMM   & 1.109 & 0.109 & 0.057 & 0.063 & 0.509 \\
    \midrule
    \multicolumn{1}{c}{\multirow{10}[10]{*}{PSMM, \newline{}$\beta_a$}} & \multirow{2}[2]{*}{1000} & \multirow{2}[2]{*}{22} & \multirow{2}[2]{*}{0.047} & CUE   & 2.879 & -0.121 & 0.532 & 1.288 & 0.946 \\
          &       &       &       & 2SLS  & 3.186 & 0.186 & 0.201 & 0.202 & 0.826 \\
\cmidrule{2-10}          & \multirow{2}[2]{*}{2000} & \multirow{2}[2]{*}{26} & \multirow{2}[2]{*}{0.043} & CUE   & 2.984 & -0.016 & 0.228 & 0.450 & 0.943 \\
          &       &       &       & 2SLS  & 3.123 & 0.123 & 0.153 & 0.153 & 0.846 \\
\cmidrule{2-10}          & \multirow{2}[2]{*}{5000} & \multirow{2}[2]{*}{33} & \multirow{2}[2]{*}{0.051} & CUE   & 3.013 & 0.013 & 0.113 & 0.111 & 0.948 \\
          &       &       &       & 2SLS  & 3.072 & 0.072 & 0.102 & 0.100 & 0.885 \\
\cmidrule{2-10}          & \multirow{2}[2]{*}{10000} & \multirow{2}[2]{*}{40} & \multirow{2}[2]{*}{0.033} & CUE   & 3.011 & 0.011 & 0.075 & 0.077 & 0.927 \\
          &       &       &       & 2SLS  & 3.045 & 0.045 & 0.073 & 0.074 & 0.893 \\
\cmidrule{2-10}          & \multirow{2}[2]{*}{20000} & \multirow{2}[2]{*}{47} & \multirow{2}[2]{*}{0.048} & CUE   & 3.005 & 0.005 & 0.050 & 0.051 & 0.944 \\
          &       &       &       & 2SLS  & 3.023 & 0.023 & 0.052 & 0.052 & 0.923 \\
    \midrule
    \multicolumn{1}{c}{\multirow{10}[10]{*}{\newline{}PSMM, $\beta_w$}} & \multirow{2}[2]{*}{1000} & \multirow{2}[2]{*}{22} & \multirow{2}[2]{*}{0.047} & CUE   & 1.150 & 0.150 & 0.633 & 1.546 & 0.957 \\
          &       &       &       & 2SLS  & 0.780 & -0.220 & 0.236 & 0.237 & 0.816 \\
\cmidrule{2-10}          & \multirow{2}[2]{*}{2000} & \multirow{2}[2]{*}{26} & \multirow{2}[2]{*}{0.043} & CUE   & 1.020 & 0.020 & 0.270 & 0.537 & 0.950 \\
          &       &       &       & 2SLS  & 0.852 & -0.148 & 0.180 & 0.179 & 0.848 \\
\cmidrule{2-10}          & \multirow{2}[2]{*}{5000} & \multirow{2}[2]{*}{33} & \multirow{2}[2]{*}{0.051} & CUE   & 0.987 & -0.013 & 0.133 & 0.130 & 0.952 \\
          &       &       &       & 2SLS  & 0.914 & -0.086 & 0.120 & 0.116 & 0.889 \\
\cmidrule{2-10}          & \multirow{2}[2]{*}{10000} & \multirow{2}[2]{*}{40} & \multirow{2}[2]{*}{0.033} & CUE   & 0.989 & -0.011 & 0.088 & 0.090 & 0.935 \\
          &       &       &       & 2SLS  & 0.946 & -0.054 & 0.086 & 0.087 & 0.896 \\
\cmidrule{2-10}          & \multirow{2}[2]{*}{20000} & \multirow{2}[2]{*}{47} & \multirow{2}[2]{*}{0.048} & CUE   & 0.996 & -0.004 & 0.059 & 0.059 & 0.944 \\
          &       &       &       & 2SLS  & 0.973 & -0.027 & 0.062 & 0.061 & 0.921 \\
    \bottomrule
    \end{tabular}%

  \label{tab: simulation results for ASMM and MSMM}%
\end{table}%

\subsection{Proximal causal inference}
We generate datasets according to the following data-generating process:

\noindent \textbf{Sample sizes:} We use $5$ different sample sizes: $N=1000$, $2000$, $5000$, $10000$, $20000$.

\noindent  \textbf{Baseline covariates:} The covariate vector is three-dimensional, $\boldsymbol{X}=(X_1, X_2,X_3)$. $\boldsymbol{X}\sim N(0,I_3+0.2J_3)$, where $I_3$ is the $3\times 3$ identity matrix, $J_3$ is a $3\times 3$ matrix with all elements equal to $1$. The covariates are truncated to lie within the interval $[-4,4]$.

\noindent \textbf{Treatment proxies:} The dimension of the treatment proxy is set to $m=4N^{-1/4}$, resulting in values of 
$22, 26, 33, 40, 47$ for the corresponding sample sizes. Each proxy variable $Z^{(j)},$ $j = 1,...,m$  is generated as $Z^{(j)} = X_1+X_2+X_3+0.5 m^{-1+3(j-1)/(4(m-1))}U+\epsilon_j$,
where $\epsilon^{(j)} \sim \text{unif}[-3,3]$ and $U\sim \text{unif}[-4,4]$ is an unmeasured confounder.

\noindent \textbf{Outcome proxy:} The outcome proxy $W$ is generated from $W = X_1+X_2+\text{expit}(X_3)+U+N(0,1).$

\noindent \textbf{Treatment:} The continuous treatment $A$ is generated as $A = X_1+X_2-X_3+U+N(0,1).$

\noindent \textbf{Outcome:} The continuous outcome is generated as $Y = 3A+\sin X_1-X^2_2-X_3+X_2X_3+U+N(0,1).$

Again, the data-generating distribution is designed to reflect weak moment conditions. Each instance of $N(0,1)$ in the previous equations should be interpreted as an independent standard normal error term. The effect of $U$ on $Z^{(j)}$ is $0.5 m^{-1+3(j-1)/4(m-1)}$, which decreases as the sample size increases,  reflecting that the correlation between $U$ and $Z^{(j)}$ becomes weaker in larger samples. This also allows for varying strengths across different treatment proxies. Additionally, we introduce nonlinear terms in the generation of both $Y$ and $W$.

We compare the two-step CUE with the two-stage least squares (2SLS) approach proposed by \citet{ett2024proximal}. The 2SLS procedure consists of two stages.  In the first step, a regression model of $W$ on $A,Z,X$ is fitted, and the fitted values are denoted as $\widehat{W}$. In the second stage, a regression model of $Y$ on $A$, $\widehat{W}$, and $X$ is fitted. The final estimates of $\beta_{a,0}$ and $\beta_{w,0}$  given by the estimated regression coefficients for $A$ and $\widehat{W}$, respectively. The 2SLS method is implemented using the  \texttt{ivreg} function in \texttt{R}.

Table \ref{tab: simulation results for ASMM and MSMM} presents the simulation results for proximal causal inference. As in the IV setting, we construct the proposed estimator using the \texttt{R} package \texttt{SuperLearner}  \citep{van2007super}  to estimate all nuisance functions, 
 and we include the random forest and generalized linear models in the SuperLearner library. When the sample size is 1000, both CUE and 2SLS exhibit some bias, although CUE still achieves nominal coverage for the 95\% confidence intervals. When the sample size exceeds 2000, CUE demonstrates small bias and accurate coverage. For 2SLS, we observe an upward bias in the estimate of $\beta_a$ and a downward bias in the estimate of $\beta_w$. The standard errors of CUE estimates are larger than those of 2SLS, as expected. The over-identification test maintains nominal type I error control across all simulation settings.

\vspace{-4mm}
\section{Revisiting the return of education example} \vspace{-2mm}
In this section, we revisit the analysis of compulsory schooling laws' effect on earnings by \citet{Angrist1991QJE}. Their study estimated the impact of educational attainment (measured in years of schooling) on the logarithm of weekly wages using data from the 1970 and 1980 U.S. Censuses. A key challenge in this analysis is unmeasured confounding from variables such as innate ability, which may be positively associated with both education and earnings, making naive regression estimates potentially biased.

One of the key insights from \citet{Angrist1991QJE} is the use of quarter of birth as an instrumental variable for education. Individuals born early in the calendar year tend to be older when starting school and are more likely to reach the legal dropout age with fewer years of schooling. This leads to a correlation between quarter of birth and educational attainment. Importantly, there is no clear mechanism through which quarter of birth directly affects earnings, making it a plausible instrument.

In our analysis, we control for observed covariates including race, region dummies, year dummies, age, and age squared. Following \citet{Angrist1991QJE}, we construct instrumental variables by interacting year dummies with quarter-of-birth dummies, resulting in a total of 29 instruments used for the analysis.

We replicate the OLS and 2SLS estimates for the 1930–1939 birth cohort as reported in \citet{Angrist1991QJE}. Table \ref{tab: real data analysis} presents our replication results, obtained using the \texttt{lm} and \texttt{2SLS} functions in \texttt{R}. However, as noted by \citet{Bound1995JASA}, the 2SLS estimates may still suffer from weak instrument bias, even in very large samples.

To illustrate the performance of our proposed methodology, we estimate the effect of educational attainment on wages using both two-step GMM and our proposed two-step CUE across three models:
\begin{itemize}
    \item Model 1: An ASMM using the original continuous education variable and log weekly wage as the outcome.
    \item Model 2: An ASMM using a binary education variable (cutoff at 12 years) and log weekly wage.
    \item Model 3: A MSMM using the same binary education variable and the weekly wage on the original (non-log) scale.
\end{itemize}
Let $A^*$ denote the continuous treatment, $A$ the binary treatment, and $Y$ the weekly wage on the original scale. In Model 1, the coefficient on $A^*$ estimates the average treatment effect on the log weekly wage, given by $\mathbb{E}_{P_{0,N}}[\log(Y(a^* + 1)) - \log(Y(a^*))]$. In Model 2, the coefficient on $A$ corresponds to the average treatment effect on the log weekly wage, $\mathbb{E}_{P_{0,N}}[\log(Y(1)) - \log(Y(0))]$, which is constant across values of $a^*$. In Model 3, the coefficient on $A$ represents the logarithm of the average treatment effect on the original weekly wage { (on the multiplicative scale),} that is, $\log(\mathbb{E}_{P_{0,N}}[Y(1)] / \mathbb{E}_{P_{0,N}}[Y(0)])$. For both GMM and the proposed two-step CUE, we use linear probability models for the instruments and estimate the nuisance functions for log weekly wage and education attainment using Super Learner with generalized linear models and random forest as candidate algorithms. The results are presented in Table \ref{tab: real data analysis}.

Comparing the original results from \citet{Angrist1991QJE} with those from Model 1, we find that the point estimate from GMM (0.0599) is almost identical to that from 2SLS (0.0556). The CUE estimate (0.0481) is smaller than both 2SLS and GMM. The standard error of CUE (0.0770) is much larger than those of 2SLS and GMM, which are 0.0354 and 0.0289, respectively.  A similar pattern holds in Models 2 and 3, where the CUE point estimates are consistently smaller than those from GMM, and the standard errors are larger. This pattern aligns with our simulation results, which also showed upward bias in GMM and larger standard errors for CUE. The over-identification tests for Models 1, 2, and 3 do not reject the null hypothesis, suggesting no empirical evidence of model misspecification. {Models 1, 2 and 3 are compatible under the strong null hypothesis that $Y(a^*)$ is constant for each individual regardless the value of $a^*$.}

\begin{table}[htbp]
  \centering
  \footnotesize
  \caption{Effects of education attainment on weekly wage using different methods. Original results refer to those reported in Table V of \citet{Angrist1991QJE}. Model 1: ASMM uses a continuous measure of education attainment as the treatment and log weekly wage as the outcome. Model 2: ASMM uses a binary treatment (education attainment above or below 12 years) with log weekly wage as the outcome. Model 3: MSMM uses the same binary treatment but models weekly wage on the original (non-log) scale.}

    \begin{tabular}{clccc}
    \toprule
          &       & Estimate   & SE  & 95\% CI \\
    \midrule
    \multirow{2}[2]{*}{Original results} & OLS & 0.0632 & 0.0003 & [0.0625,0.0390] \\
          & 2SLS  & 0.0599 & 0.0289 & [0.003,0.116] \\
    \midrule
    \multirow{2}[2]{*}{Model 1: ASMM} & GMM   & 0.0556 & 0.0354 & [-0.013,0.125] \\
          & CUE   & 0.0481 & 0.0770 & [-0.103,0.199] \\
    \midrule
    \multirow{2}[2]{*}{Model 2: ASMM } & GMM   & 0.315 & 0.194 & [-0.065,0.695] \\
          & CUE   & 0.305 & 0.294 & [-0.271,0.881] \\
    \midrule
    \multirow{2}[2]{*}{Model 3: MSMM} & GMM   & 0.381 & 0.179 & [0.030,0.731] \\
          & CUE   & 0.232 & 0.238 & [-0.234,0.698] \\
    \bottomrule
    \end{tabular}%

  \label{tab: real data analysis}%
\end{table}%

\vspace{-6mm}
\section{Discussion}

{In this paper, we develop a general framework for estimation and inference under many weak moment conditions with function-valued nuisance parameters. Our framework reveals several distinctive features that set it apart from existing literature. First, global Neyman orthogonality plays a critical role in achieving consistency, not merely asymptotic normality, under many weak moments. This stands in contrast to standard semiparametric inference, where orthogonality primarily facilitates asymptotic normality but is not required for consistency. Second, our asymptotic theory requires faster convergence rates for nuisance parameter estimators than existing strongly-identified settings. Third, our estimator converges at a rate slower than $\sqrt{N}$ and its asymptotic expansion includes both the standard GMM influence function term and a non-negligible U-statistic term, necessitating careful control over how nuisance parameter estimation affects both components.

Building on these insights, we propose a two-step CUE combined with cross-fitting that accommodates flexible, possibly nonparametric, first-stage estimation. We establish high-level conditions for consistency and asymptotic normality, enabling future researchers to apply our results to new problems by verifying these conditions. We demonstrate this versatility through three applications: the additive structural mean model (ASMM) and multiplicative structural mean model (MSMM) with many weak instruments, and proximal causal inference with many weak treatment proxies. To the best of our knowledge, our work is the first to study weak identification of the target parameter in proximal causal inference. We also develop an over-identification test for assessing model specification.

Several directions for future research remain open. First, in standard double machine learning, estimators can achieve semiparametric efficiency when the moment function corresponds to the efficient score. Investigating whether our estimator achieves optimality in an analogous sense would require developing new notions of semiparametric efficiency under many weak moment asymptotics. Second, extending our identification and estimation strategies for proximal causal inference to longitudinal settings and nonlinear models would broaden the framework's applicability. Third, developing methods for estimation and inference of weakly identified function-valued parameters, such as dose-response curves and conditional treatment effects, represents an important avenue for future work.}

\onehalfspacing

\spacingset{.85}
\bibliographystyle{apalike}
\bibliography{paper-ref}

 \newpage
 \spacingset{1.0}
 \begin{center}
    \Large Supplementary materials for “GMM with Many Weak Moment Conditions and Nuisance Parameters: General Theory and Applications to Causal Inference” 
\end{center}
\renewcommand{\thesection}{S\arabic{section}}
\setcounter{page}{1}
\setcounter{section}{0}

\renewcommand{\thefigure}{S.\arabic{figure}}

\setcounter{figure}{0}

\renewcommand{\thetable}{S\arabic{table}}

\setcounter{table}{0}

\renewcommand{\thetheorem}{S\arabic{theorem}}

\renewcommand{\theproposition}{S\arabic{proposition}}

\setcounter{theorem}{0}

\renewcommand{\thelemma}{S\arabic{lemma}}

\setcounter{lemma}{0}

\renewcommand{\theremark}{S\arabic{remark}}

\setcounter{remark}{0}

\renewcommand{\theassumption}{S\arabic{assumption}}

\setcounter{assumption}{0}
Let $P$ be the true data-generating distribution (later in our proofs, $P$ will be taken as $P_{0,N}$). We will use the following empirical process notations:
\begin{align*}
    &Pf := \int f dP,\\
    &\mathbb{P}_nf := \int f d\mathbb{P}_n = \frac{1}{n}\sum_{i=1}^n f(O_i),\\
    & \mathbb{G}_n f : = \sqrt{n}(\mathbb{P}_n-P)f.
\end{align*}
We also use $\mathbb{P}_{n,l}$ to denote the empirical distribution in the $l$-th fold, that is, $\mathbb{P}_{n,l}f = \frac{1}{N_{l}}\sum_{l(i)= l}f(O_i)$, where $N_l = \sum_{i=1}^N I\{l(i)=l\}$ is the sample size in the $l$-th fold and $I\{\}$ is indicator function. $\mathbb{G}_{n,l}$ can be defined in a similar way, that is, $\mathbb{G}_{n,l} = \sqrt{N_l}(\mathbb{P}_{n,l}-P)f$. The notation $a_N \lesssim b_N$ means $a_N$ is less than $b_N$ up to a constant which does not depend  on  $N$.

We will use the following notations repeatedly throughout the supplementary material:
\begin{align*}
& g(\beta,\eta):o\mapsto g(o;\beta,\eta),\\
& \Omega(\beta,\eta):o\mapsto g(o;\beta,\eta)g^T(o;\beta,\eta) ,\mkern9mu G(\beta,\eta):o\mapsto \frac{\partial g(o;\beta,\eta)}{\partial \beta} \\
&\text{For $h$ = $g,G$ or $\Omega$, }\mkern9mu h_i(\beta,\eta) = h(O_i;\beta,\eta),\mkern9mu \widehat{h}_l(\beta,\eta) = \mathbbm{P}_{n,l} h(\beta,\eta)\\
&h(\beta,\eta),\mkern9mu  h_i = h_i(\beta_0,\eta_{0,N}),\mkern9mu  \widebar{h}(\beta,\eta) = P_{0,N}h(\beta,\eta), \mkern9mu\widebar{h} = \widebar{h}(\beta_0,\eta_{0,N}),\\
    & Q(\beta,\eta) = \widebar{g}(\beta,\eta)^T\widebar{\Omega}(\beta,\eta)^{-1}\widebar{g}(\beta,\eta)/2+m/(2N)\\
    & \widetilde{Q}(\beta,\eta) = \widehat{g}(\beta,{\eta})^T \widebar{\Omega}(\beta,\eta_0)^{-1}\widehat{g}(\beta,{\eta})/2 \\
    & \widehat{Q}(\beta,{\eta}) = \widehat{g}(\beta,{\eta})^T\widehat{\Omega}(\beta,{\eta})^{-1}\widehat{g}(\beta,{\eta})/2.
\end{align*}

\section{Additional regularity conditions for asymptotic normality}
In this section, we present some additional regularity conditions for asymptotic normality.
\begin{assumption}
    (a) $(\mathbb{E}_{P_{0,N}}[\Vert g(O;\beta_0,\eta_{0,N}) \Vert^4]+\mathbb{E}_{P_{0,N}}[\Vert G(O;\beta,\eta_{0,N}) \Vert_{F}^4])m/N\rightarrow 0$. (b) For a constant $C$ and $j = 1,...,p$,
    \begin{align*}
      &\left\Vert \mathbb{E}_{P_{0,N}}[G(O;\beta,\eta_{0,N})G^T(O;\beta,\eta_{0,N})]\right\Vert\leq C, \quad \left\Vert \mathbb{E}_{P_{0,N}}\left[\frac{\partial G(O;\beta_0,\eta_{0,N})}{\partial \beta^{(j)}}\frac{\partial G^T(O;\beta_0,\eta_{0,N})}{\partial \beta^{(j)}}\right]\right\Vert\leq C,\\
      &\sqrt{N} \left\Vert \mathbb{E}_{P_{0,N}}\left[\frac{\partial G(O;\beta_0,\eta_{0,N})}{\partial\beta^{(j)}}\right](S_N^{-1})^T\right\Vert_F\leq C.
    \end{align*}
    (c) For $\Bar{\beta}\xrightarrow[]{p} \beta_0$, then for $k=1,...,p$,
    \begin{align*}
       \left \Vert \sqrt{N}\left[\widehat{G}(\Bar{\beta},\eta_{0,N})-\widehat{G}({\beta}_0,\eta_{0,N})\right](S_N^{-1})^T \right\Vert = o_{p}(1),\quad\Bigg\Vert \sqrt{N}\Bigg[\frac{\partial \widehat{G}(\Bar{\beta},\eta_{0,N})}{\partial \beta^{(k)}}-\frac{\partial \widehat{G}({\beta}_0,\eta_{0,N})}{\partial \beta^{(k)}}\Bigg] (S_N^{-1})^T\Bigg\Vert= o_{p}(1).
    \end{align*}
\label{Assumption: further restriction for moments.}
\end{assumption}

A sufficient condition for \ref{Assumption: further restriction for moments.}(a) to hold is that every element of $g$ and $G$ is uniformly bounded and $m^3/N=o(1)$. The second and third inequalities of Assumption \ref{Assumption: further restriction for moments.} (b) and Assumption \ref{Assumption: further restriction for moments.} (c) hold trivially if $g(o;\beta,\eta)$ is linear in $\beta$ (and hence $G(o;\beta,\eta)$ does not depend on $\beta$). Assumption  \ref{Assumption: further restriction for moments.} (c) is used for   obtaining the asymptotic distribution of $\hat{\beta}-\beta$ by showing the convergence of the second order derivative of objective evaluated at $\widebar{\beta}$ converges to its counterpart evaluated at $\widebar{\beta}$ with $\eta$ fixed at $\eta_{0,N}$.

\begin{assumption}
    For all $\beta,\Tilde{\beta}$ in a neighborhood $\mathcal{B}'$ of $\beta_0$ and $A(o;\beta,\eta)$ equal to $\Omega^k(o;\beta,\eta)$, $\Omega^{kl}(o;\beta,\eta)$ or $\Omega^{k,l}(o;\beta,\eta)$, we have (a) $\sup_{\beta \in \mathcal{B}'}\Vert \widehat{A}(\beta,\eta_{0,N})-\widebar{A}(\beta,\eta_{0,N})\Vert \xrightarrow[]{p}0$, (b) $|a'[A(\Tilde{\beta},\eta_{0,N})-A(\beta,\eta_{0,N}))]b|\leq C\Vert a\Vert \Vert b\Vert \Vert \Tilde{\beta}-\beta \Vert$, where $\widehat{A}(\beta,\eta) = \frac{1}{N}\sum_{i=1}^NA(O_i;\beta,\eta)$, $\widebar{A}(\beta,\eta) = \mathbb{E}_{P_{0,N}}[A(O;\beta,\eta)]$.
    \label{Assumption: Further restrictions for uniform convergence}
\end{assumption}

Assumption \ref{Assumption: Further restrictions for uniform convergence} is also made in \citet{newey2009generalized}, which imposes further assumptions on Uniform convergence and smoothness conditions for $\Omega,\Omega^{k},\Omega^{kl},\Omega^{k,l}$.

\section{Additional regularity condition for separable moment functions}

\begin{assumption}[Separable score regularity and nuisance parameters estimation]\quad

    \begin{enumerate}[label=(\alph*)]
    \item $g^{[b]}(o;\eta)$ satisfies global Neyman orthogonality for any $b\in \{1,2,...,B\}$ on $\mathcal{B}$. The second-order Gateux derivative $D^{(2)}_{\beta,t,g^{[b]}}(\eta)$ is continuous.
    \item Let $I\subset \{1,\dots, N\}$ be a random subset of size $N/L$. Suppose the nuisance parameter estimator $\widehat{\eta}$, constructed using only observations outside of $I$,  belongs to a realization set $\mathcal{T}_N\subset \tilde{\mathcal{T}}_N$ with probability {at least} $1-\Delta_N$, where $\Delta_N=o(1)$. Assume that $\mathcal{T}_N$ contains the true value $\eta_{0,N}$ and satisfies the following conditions: 

    Define the following statistical rates for all $j,k \in \{1,...,m\},l,r\in\{1,...,q\}$, $b_1,b_2\in\{1,...,B\}$:
    \begin{align*}
        &r^{(b_1),(j,q)}_N: = \sup_{\eta\in \mathcal{T}_N}\left(\mathbb{E}_{P_{0,N}}\left\vert g^{[b_1],(j,q)}(O;\eta)-g^{[b_1],(j,q)}(O;\eta_{0,N})\right\vert^2\right)^{1/2} \\
        &\lambda^{(j)}_N:=\sup_{\beta\in \mathcal{B},t\in(0,1),\eta\in \mathcal{T}_N}\left\vert \frac{\partial^2}{\partial t^2} \mathbb{E}_{P_{0,N}}g^{(j)}(O;\beta,(1-t)\eta_{0,N}+t\eta)\right\vert\\
        & r^{(b_1,b_2),(j,k,l,r)}_N:= \sup_{\eta \in \mathcal{T}_N} \left(\mathbb{E}_{P_{0,N}}\left|g^{[b_1],(j,l)}(O;\eta)g^{[b_2],(k,r)}(O;\eta)-g^{[b_1],(j,l)}(O;\eta_{0,N})g^{[b_2],(k,r)}(O;\eta_{0,N})\right|^2\right)^{1/2},
        \end{align*}
      where $g^{[b],(j,l)}(o;\eta)$ denotes the $(j,l)$th element of the matrix $g^{[b]}(o;\eta)$. We assume there exists $\epsilon>0$ such that:
        \begin{align*}
        & a_N\leq  \delta_N, \quad {\lambda^{(j)}_N} \leq N^{-1/2}\delta_N,\quad  \delta_N = o(1/\sqrt{m}),
    \end{align*}
    for $a_N \in \{r^{(b_1),(j,q)}_N, r^{(b_1,b_2),(j,k,l,r)}_N\}$.
    \end{enumerate}
    \label{Assumption: separable score}
\end{assumption}
Similar to Assumption 8, Assumption \ref{Assumption: separable score} imposes conditions on the convergence rates of nuisance estimators, but specifically for separable moment functions. A key simplification is that the definitions of  $r^{(b_1),(j,q)}_N$ and $r^{(b_1,b_2),(j,k,l,r)}_N$ do not involve taking the supremum over $\beta$, which makes these conditions more straightforward to verify.  Assumption 9 is also easier to check for separable moment functions because the supremum over $\beta$ can be bounded by a constant due to the compactness of $\mathcal{B}$. This separation enables the use of concentration inequalities for norms of high-dimensional random matrices \citep{tropp2016expected}. 

\section{A consistent variance estimator}

We now provide a consistent estimator for the asymptotic variance. Let $\widehat{H} = \frac{\partial^2\widehat{Q}(\widehat{\beta},\widehat{\eta})}{\partial\beta \partial \beta^T}$, $ \widehat{D}^{(j)}(\beta) = {\frac{1}{N}\sum_{i=1}^N\frac{\partial g(O_i;\beta,\widehat{\eta}_{l(i)})}{\partial\beta_j}}-\frac{1}{N}\sum_{i=1}^N\frac{\partial g(O_i;\beta,\widehat{\eta}_{l(i)})}{\partial \beta_j}g(O_i;\beta,\widehat{\eta}_{l(i)})^T\widehat{\Omega}(\beta,\widehat{\eta})^{-1}\widehat{g}(\beta,\widehat{\eta})$, \quad $\widehat{D} = (\widehat{D}^{(1)}(\widehat{\beta}),...,\widehat{D}^{(p)}(\widehat{\beta}))$, $\widehat{\Omega} = \widehat{\Omega}(\widehat{\beta},\widehat{\eta})$. Then, the variance estimator is given by  $\widehat{V}/N$, where $\widehat{V} = \widehat{H}^{-1}\widehat{D}^T\widehat{\Omega}^{-1}\widehat{D}\widehat{H}^{-1}$. The following theorem establishes that $\widehat{V}$ is a consistent variance estimator.

\begin{theorem}[Consistency of variance estimator] Suppose all assumptions in Theorem \ref{Theorem: consistency and ASN} hold. Then,
    \begin{align*} S_N^T\widehat{V}S_N/N\xrightarrow[]{p}V.
    \end{align*}
    \label{Theorem: consistency of variance estimator}
    Furthermore, for any vector $c$, if there exists a sequence $r_N$ and a nonzero vector $c^*\neq 0$ such that $r_NS_N^{-1}c\rightarrow c^*$, then
    \begin{align*}
        \frac{c^T(\widehat{\beta}-\beta_0)}{\sqrt{c^T\widehat{V}c/N}}\rightsquigarrow  N(0,1).
    \end{align*}
    \label{theorem: variance estimator}
\end{theorem}

In the standard double machine learning literature, the asymptotic variance is usually estimated by the empirical variance of the influence function with plug-in nuisance estimates  \citep{chernozhukov2018double}. However, this conventional variance estimator ignores the contribution of the U-statistics term to the asymptotic variance and therefore tends to underestimate the variance under many weak moment conditions.

\section{More details and formal asymptotic results for the three examples}

\subsection{ASMM}
We considered a commonly used specification for the ASMM, which is $\gamma(a, z, x; \beta) = \beta a$ in the main body of our paper. Another widely used form is $\gamma(a, z, x; \beta) = a(\beta^{(1)} + x^T\beta^{(2)})$, which allows the treatment effect to vary linearly with baseline covariates $X$.

The construction of moment condtions can be generalized to general SNMM. Under the ASMM, we have
\begin{align*}
    &\mathbb{E}_{{P}_{0,N}}[Y-\gamma(A,Z,X;\beta_0)|Z,X]=\mathbb{E}_{P_0,N}[\mathbb{E}_{{P}_{0,N}}[Y-\gamma(A,Z,X;\beta_0)|Z,X,A,U]|Z,X] \\
    =& \mathbb{E}_{{P}_{0,N}}[\mathbb{E}_{{P}_{0,N}}[Y(0)|Z,X,A,U]|Z,X]=\mathbb{E}_{\tilde{P}_{0,N}}[Y(0)|X].
\end{align*}
Similarly, we can show that $\mathbb{E}_{P_{0,N}}[Y-\gamma(A,Z,X;\beta_0)|X] =\mathbb{E}_{{P}_{0,N}}[Y(0)|X]$, which implies $\mathbb{E}_{P_{0,N}}[Y-\gamma(A,Z,X;\beta_0)|X]=\mathbb{E}_{P_{0,N}}[Y-\gamma(A,Z,X;\beta_0)|Z,X]$. Therefore, define the moment function $\tilde{g}(O; \beta, \eta_Z) = (Y - \gamma(A, Z, X; \beta))(Z - \eta_Z(X))$, where $\eta_Z$ is a function of $X$, and its true value is given by $\eta_{Z,0,N}(X) = \mathbb{E}_{P_{0,N}}[Z \mid X]$. Then, $\mathbb{E}_{P_{0,N}}[\tilde{g}(O; \beta_0, \eta_{Z,0,N})] = 0$, which defines a valid moment condition for identifying and estimating $\beta_0$. 


For consistency and asymptotic normality of our proposed two-step CUE in the presence of high-dimensional weak instrumental variables, we need the following regularity assumptions.
\begin{assumption}[Regularity assumptions for ASMM] We assume the following conditions hold: (a)  The variables $(X,Z)$ are uniformly bounded. There exists a constant $C$ such that $\mathbb{E}_{P_{0,N}}[Y^{4}|X]<C$, $\mathbb{E}_{P_{0,N}}[A^{4}|X]<C$, $\mathbb{E}_{P_{0,N}}[A^{2}|Z,X]<C$. (b) There exists a positive constant $H$ such that $\mu_N^{-2} N\widebar{G}^T\widebar{\Omega}^{-1}\widebar{G}\rightarrow H$, where $\mu_N\rightarrow\infty$, $\mu_N=O(\sqrt{N})$, $m/\mu_N^2$ is bounded for all $N$, and $m^3/N\rightarrow 0$. (c) There exists a constant $C$ such that $\frac1CI_m\preceq\mathbb{E}_{P_{0,N}}[(Z-\mathbb{E}_{P_{0,N}}[Z|X])(Z-\mathbb{E}_{P_{0,N}}[Z|X])^T|X] \preceq CI_m$. (d) Let $I\subset \{1,\dots, N\}$ be a random subset of size $N/L$. Suppose the nuisance parameter estimator $\widehat{\eta}$, constructed using only observations outside of $I$, satisfies the following conditions: with $P_{0,N}$- probability  at least $1-\Delta_N$, for $j=1,...,m$, $\Vert \widehat{\eta}_{A,l}-{\eta}_{A,0,N}\Vert_{P_{0,N},2}\leq \delta_N$, $\Vert \widehat{\eta}_{Y,l}-{\eta}_{Y,0,N}\Vert_{P_{0,N},2}\leq \delta_N$, $\Vert \widehat{\eta}_{Z^{(j)},l}-{\eta}_{Z^{(j)},0,N}\Vert_{P_{0,N},2}\leq \delta_N/m$, $\Vert \widehat{\eta}_{Y,l}-{\eta}_{Y,0,N}\Vert_{P_{0,N},\infty}\leq C$, $\Vert \widehat{\eta}_{A,l}-{\eta}_{A,0,N}\Vert_{P_{0,N},\infty}\leq C$, $\Vert \widehat{\eta}_{Z^{(j)},l}-{\eta}_{Z^{(j)},0,N}\Vert_{P_{0,N},\infty}\leq C$, $\Vert \widehat{\eta}_{A,l}-{\eta}_{A,0,N}\Vert_{P_{0,N},2}\times \Vert \widehat{\eta}_{Z^{(j)},l}-{\eta}_{Z^{(j)},0,N}\Vert_{P_{0,N},2}\leq N^{-1/2}\delta_N$, $\Vert \widehat{\eta}_{Y,l}-{\eta}_{Y,0,N}\Vert_{P_{0,N},2}\times \Vert \widehat{\eta}_{Z^{(j)},l}-{\eta}_{Z^{(j)},0,N}\Vert_{P_{0,N},2}\leq N^{-1/2}\delta_N$, where $\delta_N=o(1/\sqrt{m})$, $\Delta_N=o(1)$,  and $C>0$ is a constant.
    \label{Assumption: Additive SMM}
\end{assumption}

    Assumption \ref{Assumption: Additive SMM}(d) also requires that all $m+2$ nuisance estimates simultaneously belong to the realization set. In the standard double machine learning literature \citep{chernozhukov2018double}, where $m$ is fixed, this is not a concern because the number of nuisance parameters remains constant. However, under many weak asymptotics, the number of nuisance parameters
$\eta_{Z^{(j)},0,N}$ increases with $m$, making it more challenging to control their convergence rate uniformly. If $\widehat{\eta}_{Z^{(j)},l}$ is estimated using nonlinear least squares, then under appropriate regularity conditions, its $L_2$ error concentrates within an interval of width $O(\sqrt{\log N/N})$ with probability at least $1-C/N$ (see Theorem 5.2 of \citet{koltchinskii2011oracle}). Therefore, the joint probability that all $\widehat{\eta}_{Z^{(j)},l}, $ for $j=1,\dots, m$, lie within such intervals is at least $1-Cm/N$. Since $m/N=o(1)$, it is possible to control the convergence rates of all  $(\widehat{\eta}_{Z^{(j)},l})$ simultaneously, even as $m$ goes to infinity. 

The consistency and asymptotic normality of the two-step CUE estimator in the presence of high-dimensional weak instrumental variables $Z$ is established in the following theorem.
\begin{theorem}
   Suppose Assumption \ref{Assumption: Additive SMM} holds. Then the CUE estimator $\widehat{\beta}$ satisfies all regularity conditions stated in Theorem 2 for separable moment functions. Therefore, $\widehat{\beta}$ is consistent and asymptotically normal, with its asymptotic distribution given by Theorem 2.
    \label{Theorem: ASMM}
\end{theorem}

\subsection{MSMM}

\begin{assumption}[Regularity assumptions for MSMM] We assume the following conditions hold: (a) The variables $(X,Z)$ are uniformly bounded. There exists a constant $C$ such that $\mathbb{E}_{P_{0,N}}[Y^{4}|A,X]<C$. (b) There exists a positive constant $H$ such that $\mu_N^{-2} N\widebar{G}^T\widebar{\Omega}^{-1}\widebar{G}\rightarrow H$, where $\mu_N \rightarrow\infty$, $\mu_N=O(\sqrt{N})$, $m/\mu_N^2$ is bounded for all $N$, and $m^3/N\rightarrow 0$. (c) There exists a constant $C$ such that $\frac1CI_m\preceq \mathbb{E}_{P_{0,N}}[(Z-\mathbb{E}_{P_{0,N}}[Z|X])(Z-\mathbb{E}_{P_{0,N}}[Z|X])^T|X]  \preceq CI_m$. (d) Let $I\subset \{1,\dots, N\}$ be a random subset of size $N/L$. Suppose the nuisance parameter estimator $\widehat{\eta}$, constructed using only observations outside of $I$, satisfies the following conditions: with $P_{0,N}$- probability  at least $1-\Delta_N$, $\Vert \widehat{\eta}_{A,(1),l}-{\eta}_{A,(1),0,N}\Vert_{P_{0,N},2}\leq \delta_N$, $\Vert \widehat{\eta}_{Y,(a),l}-{\eta}_{Y,(a),0,N}\Vert_{P_{0,N},2}\leq \delta_N$, $\Vert \widehat{\eta}_{Z^{(j)},l}-{\eta}_{Z^{(j)},0,N}\Vert_{P_{0,N},2}\leq \delta_N/m$, $\Vert \widehat{\eta}_{A,(1),l}-{\eta}_{A,(1),0,N}\Vert_{P_{0,N},\infty}\leq C$, $\Vert \widehat{\eta}_{Y,(a),l}-{\eta}_{Y,(a),0,N}\Vert_{P_{0,N},\infty}\leq C$, $\Vert \widehat{\eta}_{Z^{(j)},l}-{\eta}_{Z^{(j)},0,N}\Vert_{P_{0,N},\infty}\leq C$, $\Vert \widehat{\eta}_{A,(1),l}-{\eta}_{A,(1),0,N}\Vert_{P_{0,N},2}\times \Vert \widehat{\eta}_{Z^{(j)},l}-{\eta}_{Z^{(j)},0,N}\Vert_{P_{0,N},2}\leq N^{-1/2}\delta_N$, $\Vert \widehat{\eta}_{Y,(a),l}-{\eta}_{Y,(a),0,N}\Vert_{P_{0,N},2}\times \Vert \widehat{\eta}_{Z^{(j)},l}-{\eta}_{Z^{(j)},0,N}\Vert_{P_{0,N},2}\leq N^{-1/2}\delta_N$, where $\delta_N=o(1/\sqrt{m})$, $\Delta_N=o(1)$, and $C>0$ is a constant.
\label{Assumption: MSMM}
\end{assumption}

In this example,
\begin{align*}
    \widebar{G}=-\exp(-\beta_0)\mathbb{E}_{P_{0,N}}\left[\left(YI\{A=1\}-\mathbb{E}_{P_{0,N}}[YI\{A=1\}|X]\right)(Z-\mathbb{E}_{P_{0,N}}[Z|X])\right],
\end{align*}
so Assumption \ref{Assumption: MSMM}(b) implies that the conditional correlation between $YI\{A=a\}$ and $Z$ given $X$ is weak. Since $Z$ affects $Y$ only through $A$, this weak correlation can still be interpreted as weak identification from a small conditional correlation between $A$ and $Z$.

The consistency and asymptotic normality of the two-step CUE estimator under high-dimensional weak instrumental variables $Z$ is established in the following theorem:
\begin{theorem}
    Suppose Assumption \ref{Assumption: MSMM} holds.  Then the CUE estimator $\widehat{\beta}$ satisfies all regularity conditions stated in Theorem 2 for separable moment functions. Therefore, $\widehat{\beta}$  is consistent and asymptotically normal, with its asymptotic distribution given by Theorem 2.
    \label{Theorem: MSMM}
\end{theorem}

\subsection{PSMM}

For consistency and asymptotic normality of the CUE estimator in the presence of many weak treatment proxies, we require the following regularity conditions:
\begin{assumption}[Regularity assumptions for proximal causal inference example] We assume the following conditions hold: (a) The variables $(A,W,Z)$ are uniformly bounded. There exists a constant $C>0$ such that $\mathbb{E}_{P_{0,N}}[Y^4|A,X,Z]<C$ and $\mathbb{E}_{P_{0,N}}[W^4|A,X,Z]<C$. (b) Assumption 1 holds, and $m^3/N\rightarrow 0$. (c) $\frac1C I_{m+1}\preceq \mathbb{E}_{P_{0,N}}[(\tilde{Z}-\mathbb{E}_{P_{0,N}}[\tilde{Z}|X])(Z-\mathbb{E}_{P_{0,N}}[\tilde{Z}|X])^T|X]  \preceq   CI_{m+1}$, where $\tilde{Z} = (A,Z^T)^T$. (d) Let $I\subset \{1,\dots, N\}$ be a random subset of size $N/L$. Suppose the nuisance parameter estimator $\widehat{\eta}$, constructed using only observations outside of $I$, satisfies the following conditions: with $P_{0,N}$-probability no less than $1-\Delta_N$, $\Vert \widehat{\eta}_{A,l}-{\eta}_{A,0,N}\Vert_{P_{0,N},2}\leq \delta_N$, $\Vert \widehat{\eta}_{Y,l}-{\eta}_{Y,0,N}\Vert_{P_{0,N},2}\leq \delta_N$,$\Vert \widehat{\eta}_{W,l}-{\eta}_{W,0,N}\Vert_{P_{0,N},2}\leq \delta_N$,$\Vert \widehat{\eta}_{Z^{(j)},l}-{\eta}_{Z^{(j)},0,N}\Vert_{P_{0,N},2}\leq \delta_N/m$,$\Vert \widehat{\eta}_{A,l}-{\eta}_{A,0,N}\Vert_{P_{0,N},\infty}\leq C$, $\Vert \widehat{\eta}_{Y,l}-{\eta}_{Y,0,N}\Vert_{P_{0,N},\infty}\leq C$,$\Vert \widehat{\eta}_{W,l}-{\eta}_{W,0,N}\Vert_{P,\infty}\leq C$,$\Vert \widehat{\eta}_{Z^{(j)},l}-{\eta}_{Z^{(j)},0,N}\Vert_{P_{0,N},\infty}\leq C$,$\Vert \widehat{\eta}_{Z^{(j)},l}-\eta_{Z^{(j)},0,N}\Vert_{P_{0,N},2}(\Vert \widehat{\eta}_{Y,l}-\eta_{Y,0,N}\Vert_{P_{0,N},2}+\Vert \widehat{\eta}_{A,l}-\eta_{A,0,N}\Vert_{P_{0,N},2}+\Vert \widehat{\eta}_{W,0}-\eta_{W,0,N}\Vert_{P_{0,N},2})\leq N^{-1/2}\delta_N$, $\Vert \widehat{\eta}_{A,l}-\eta_{A,0,N}\Vert_{P_{0,N},2}(\Vert \widehat{\eta}_{Y,l}-\eta_{Y,0,N}\Vert_{P_{0,N},2}+\Vert \widehat{\eta}_{A,l}-\eta_{A,0,N}\Vert_{P_{0,N},2}+\Vert \widehat{\eta}_{W,l}-\eta_{W,0,N}\Vert_{P_{0,N},2})\leq N^{-1/2}\delta_N$, where $\delta_N=o(1/\sqrt{m})$, $\Delta_N=o(1)$, and $C>0$ is a constant.
\label{Assumption: PSMM}
\end{assumption}

In this example, the matrix $\widebar{G}$ is given by
\begin{align*}
    \widebar{G} = 
    \mathbb{E}_{P_{0,N}}\left[\begin{pmatrix*}[r]
  A-\mathbb{E}_{P_{0,N}}[A|X] \\
  Z-\mathbb{E}_{P_{0,N}}[Z|X]
\end{pmatrix*}(-(A-\mathbb{E}_{P_{0,N}}[A|X]),-(W-\mathbb{E}_{P_{0,N}}[W|X]))\right].
\end{align*}
Thus, Assumption \ref{Assumption: PSMM}(b) implies that the conditional correlations between $Z$ and $A$, as well as between $Z$ and $W$, are weak given $X$. This, in turn, suggests that the conditional correlation between $Z$ and the unmeasured confounder $U$ is also weak, meaning that $Z$ serves as a weak proxy for $U$. Assumption \ref{Assumption: PSMM}(c) is reasonable under this setting, as weak dependence between $Z$ and $U$ given $X$ implies that the components $Z^{(i)}$ and $Z^{(j)}$ are only weakly correlated with each other given $X$ when $i\neq j$.

The following theorem establishes the consistency and asymptotic normality of the two-step CUE estimator.
\begin{theorem}
   Suppose Assumption \ref{Assumption: PSMM} holds. Then the CUE estimator $\widehat{\beta}$ satisfies all the regularity conditions stated in Theorem 2 for separable moment functions. Therefore, $\widehat{\beta}$ is consistent and asymptotically normal, with the asymptotic distribution as stated in Theorem 2.
    \label{thm: PSMM}
\end{theorem}

\section{Bounding supremum of empirical processes}
In this section, we will briefly review useful tools for bounding the supremum of empirical processes. Let $P$ be the underlying data generating distribution, $O_1,O_2,...,O_n$ are iid sampled from $P$, the support of $O_i$ is denoted as $\mathcal{O}$. Let $\mathbb{P}_n$ be the empirical measure as defined before. Let $\mathcal{F}$ be a class of measurable functions mapping from $\mathcal{O}$ to $\mathbb{R}$ with envelope function $F$.

It is often of interest to study the convergence rates suprema of empirical processes. Let 
\begin{align*}
    \Vert \mathbb{G}_n \Vert_{\mathcal{F}} = \sup_{f \in  \mathcal{F}}\vert\mathbb{G}_nf\vert.
\end{align*}
Our goal is to give a high-probability tail bound for $\Vert \mathbb{G}_n \Vert_{\mathcal{F}}$.   Chapter 2.14 of \citet{van1996weak} described such techniques. The technique we used here is based on Dudley's entropy integral \citep{van1996weak}. The uniform entropy is defined as
\begin{equation*}
    J(\delta,\mathcal{F}|F,L_2) = \sup_{Q}\int_0^\delta \sqrt{1+\log N(\epsilon \Vert F\Vert_{Q,2},\mathcal{F},L_2(Q))}d\epsilon
\end{equation*}
where the supremum is taken over all discreet probability measures $Q$ with $\Vert F\Vert_{Q,2}>0$.


Suppose $\mathcal{F}$ is a VC-type class, which means
\begin{equation*}
    \log \sup_{Q} N(\epsilon \Vert F\Vert_{Q,2},\mathcal{F},\Vert\cdot \Vert_{Q,2})\leq v\log(a/\epsilon)
\end{equation*}
for some constants $a\geq e$ and $v\geq 1$, where $N(\epsilon \Vert F\Vert_{Q,2},\mathcal{F},\Vert\cdot \Vert_{Q,2})$ is the covering number, then we have the following result

\begin{lemma}
\begin{equation*}
    \mathbb{E}_P[\Vert \mathbb{G}_n \Vert_{\mathcal{F}}]\leq K \Bigg(\sqrt{v\sigma^2\log\Big(\frac{a\Vert F\Vert_{P,2}}{\sigma}\Big)}+\frac{v\Vert M\Vert_{P,2}}{\sqrt{n}}\log \Big(\frac{a\Vert F\Vert_{P,2}}{\sigma}\Big)\Bigg)
\end{equation*}
    where $K$ is a absolute constant. Furthermore, for $t\geq 1$, with probability bigger than $1-t^{-q/2}$,
\begin{align*}
    \Vert \mathbb{G}_n\Vert_{\mathcal{F}}\leq (1+\alpha)E_P[\Vert \mathbb{G} \Vert_{\mathcal{F}}] +K(q)\Bigg((\sigma+n^{-1/2}\Vert M \Vert_{P,q})\sqrt{t}+\alpha^{-1}n^{-1/2}\Vert M\Vert_{P,2}t\Bigg)
\end{align*}
$\forall \alpha >0$, where $c$ is a constant, $\Vert M\Vert_{P,q}\leq n^{1/q}\Vert F\Vert_{P,q}$ and $K(q)>0$ is a constant only depending on $q$.

\label{lemma: concentration of supremum ep}
\end{lemma}
 This result can be found in \citet{chernozhukov2018double}. We will mainly use a special case of this result in our paper. Let $t = n^{\frac{2}{q}}$, then we have with probability bigger than $1-\frac{1}{n}$, 
 \begin{align*}
     &\Vert \mathbb{G}_n\Vert_{\mathcal{F}}\\
     \leq & (1+\alpha)K\Bigg(\sigma\sqrt{v\log\Big(\frac{a\Vert F\Vert_{P,2}}{\sigma}\Big)}+\frac{v\Vert M\Vert_{P,2}}{\sqrt{n}}\log \Big(\frac{a\Vert F\Vert_{P,2}}{\sigma}\Big) \Bigg) +\\
     &K(q)\Bigg(\Big(\sigma+\frac{\Vert M \Vert_{P,q}}{\sqrt{n}}\Big)n^{\frac{1}{q}}+\frac{1}{\alpha \sqrt{n}}\Vert M \Vert_{P,2}n^{\frac{2}{q}}\Bigg)
 \end{align*}
 where I used the fact that $\Vert M \Vert_{P,2}\leq \Vert M \Vert_{P,q}$ for $q\geq 2$, we further have $\Vert M \Vert_{P,q}\leq n^{1/q}\Vert F\Vert_{P,q}$.

 The reason why we choose $t = n^{\frac{2}{q}}$ is that, in this case, $m^2t^{-q/2} = m^2/(n^{1/2})^2$, which is $o(1)$ if we assume $m = o(n^{1/2})$.
 \section{Other useful results}

In our paper, we need to bound terms like $\Vert \frac{1}{N}\sum_{i=1}^Ng_i(\beta_0,\eta_{0,N})g_i^T(\beta_0,\eta_{0,N})-\mathbb{E}_{P_{0,N}}g_i(\beta_0,\eta_{0,N})g_i^T(\beta_0,\eta_{0,N})\Vert$ and $\Vert \frac{1}{N}\sum_{i=1}^NG_i(\beta_0,\eta_{0,N})g_i^T(\beta_0,\eta_{0,N})-\mathbb{E}_{P_{0,N}}G_i(\beta_0,\eta_{0,N})g_i^T(\beta_0,\eta_{0,N})\Vert$. Therefore, we need some results for bounding the norm of random matrices. We will mainly use the following result, which is Theorem 1 of \citet{tropp2016expected}.

\begin{lemma}
    \textbf{(The expected norm of an independent sum of random matrices).} Let $X_1,...,X_N$ be a sequence of i.i.d. 
$d_1\times d_2$ random matrices with $\mathbb{E}[X_i]=0$ for all $i$. Let
$S := \sum_{i=1}^NX_i$, define the matrix variance parameter to be
\begin{equation*}
    \nu(S) = \max\Big\{\Big\Vert \sum_{i=1}^N\mathbb{E}[X_iX_i^T]\Big\Vert, \Big\Vert \sum_{i=1}^N\mathbb{E}[X_i^TX_i]\Big\Vert\Big\},
\end{equation*}
and the large deviation parameter to be
\begin{equation*}
    L:=(\mathbb{E}\max_i\Vert X_i\Vert^2)^{1/2}.
\end{equation*}

Define the dimensional constant
\begin{align*}
    C(d_1,d_2):=4(1+2\lceil \log(d_1+d_2)\rceil).
\end{align*}
Then
\begin{align*}
    \sqrt{\nu(S)/4}+L/4\leq (\mathbb{E}\Vert S\Vert)^{1/2}\leq \sqrt{C(d_2,d_2)\nu(S)}+C(d_2,d_2)L.
\end{align*}\label{lemma: Tropp's inequality}
\end{lemma}
\ref{lemma: Tropp's inequality} is useful for establishing that the empirical mean of high-dimensional random matrices concentrates around its expectation.

\begin{lemma}
    Let $X$ be a $m\times m$ positive semidefinite random matrix, $Y$ be a random variable with $|Y|\leq K $ a.s., then
    \begin{align*}
        \Vert \mathbb{E}XY\Vert \leq 2\Vert \mathbb{E}XK\Vert.
    \end{align*}
\end{lemma}

\begin{proof}
    \begin{align*}
        \Vert \mathbb{E}XY\Vert= \Vert \mathbb{E}X(Y^+-Y^-)\Vert\leq  \Vert \mathbb{E}XY^+\Vert+ \Vert \mathbb{E}XY^-\Vert
        \end{align*}
        For $\Vert \mathbb{E}XY^+\Vert$, by the definition of $\Vert \cdot \Vert$, we have
        \begin{align*}
            &\Vert \mathbb{E}XY^+\Vert = \sup_{\Vert v\Vert=1} v^T\mathbb{E}XY^+v = \sup_{\Vert v\Vert=1} \mathbb{E}v^TXY^+v\\
            \leq & \sup_{\Vert v\Vert=1} \mathbb{E}v^TXKv=\sup_{\Vert v\Vert=1} v^T\mathbb{E}XKv = \Vert \mathbb{E}XK\Vert
        \end{align*}
        Similarly, we have$ \Vert EXY^-\Vert\leq \Vert EXK\Vert$, therefore,
        \begin{align*}
        \Vert EXY\Vert \leq 2\Vert EXK\Vert.
    \end{align*}
\end{proof}

\begin{lemma}
    Let $X,Z$ be $m\times m$ positive semidefinite random matrices with $X\leq Z$. Let $Y$ be a positive random variable. Then $\Vert \mathbb{E}XY\Vert \leq \Vert \mathbb{E}ZY\Vert$.
\end{lemma}

\begin{proof}
    \begin{align*}
        &\Vert \mathbb{E}XY\Vert = \lambda_{\max}(\mathbb{E}XY) = \lambda_{\max}(\mathbb{E}((X-Z)Y+ZY)\leq \lambda_{\max}(\mathbb{E}(X-Z)Y)+\lambda_{\max}(\mathbb{E}ZY)\\
        \leq &\lambda_{\max}(\mathbb{E}ZY) = \Vert \mathbb{E}ZY\Vert,
    \end{align*}
    where the first inequality is due to Weyl's inequality, the second inequality is because $\mathbb{E}(X-Z)Y$ is negative-semidefinite.
\end{proof}

The following lemma is Lemma A0 of \citet{newey2009generalized}.
\begin{lemma}
    If $A$ be  positive semidefinite matrix, if $\Vert \widehat{A}-A\Vert_F\xrightarrow[]{p}0$, $\lambda_{\min}(A)\geq 1/C$, $\lambda_{\max}(A)\leq C$, then with probability approaching 1 $\xi_{\min}(\widehat{A})\geq 1/(2C)$ and $\xi_{\max}\leq 2C$.
\end{lemma}

The following lemma is Corollary 7.1 of \citet{zhang2021concentration}.
\begin{lemma}
   Let $\{X_i\}_{i=1}^N$ be identically distributed but not necessarily independent and assume $E(|X_1|^p)<\infty, (p\geq 1)$. Then $\mathbb{E}\max_{1\leq i\leq n}|X_i| = o(n^{1/p})$.
\end{lemma}
 \section{Technical lemmas}
In this section, we will present some technical lemmas that will be used for subsequent proof.

\begin{proposition} \citep{Chernozhukov2022locallyrobust}. Suppose that $\eta(P_{1,N}) = \eta_{1,N} \in \mathcal{T}_{N}$, $\eta(P_{0,N}) = \eta_{0,N}$. let $\mathcal{P}_N = \{P_{t,N}, t \in [0,1]\}$ be a collection of distributions such that $P_{t,N} = (1-t)P_{0,N}+tP_{1,N}$ and $\eta(P_{t,N}) = (1-t)\eta_{0,N}+t\eta_{1,N}$. Suppose that For all $\beta \in \mathcal{B}_1$ and $\eta_{1,N}\in \mathcal{T}_N$,(a) there exists a function $\phi(O;\beta,\eta)$ satisfying $\frac{d}{dt}\mathbb{E}_{P_{0,N}}[\Tilde{g}(O;\beta,\eta(P_t))] = \int \phi(o;\beta,\eta_{0,N})dP_1(o),\mathbb{E}_{P_{0,N}}[\phi(O;\beta,\eta_{0,N})] = 0$, $\mathbb{E}_{P_{0,N}}[\phi(O;\beta,\eta_{0,N})^2]<\infty$, (b) $\int \phi(o;\beta,\eta(P_{t,N}))dP_{t,N}(o)$ for $t\in[0,1]$ and  (c) $\int \phi(o;\beta,\eta(P_t))dP_{0,N}(o)$ and $\int \phi(o;\beta,\eta(P_t))dP_{1,N}(o)$ are continuous at $t = 0$, then $g(\beta,\eta) = \Tilde{g}(\beta,\eta)+\phi(\beta,\eta)$ is Neyman orthogonal on $\mathcal{B}_1$.
\end{proposition}
\begin{lemma}[Permanence properties of Neyman orthogonality]     \label{Lemma: permanence of global neyman orthogonality}
(a) Suppose  $g(o;\beta,\eta)$ satisfies global Neyman orthogonality,   and that the partial derivative   $\frac{\partial g(o;\beta,\eta)}{\partial \beta}$ exists for all $\beta \in \mathcal{B}_1\subset \mathcal{B}$. If there exists an integrable function $h(o)$ such that 
 $\vert \frac{\partial g(o;\beta,\eta)}{\partial \beta}\vert\leq h(o)$ almost surely,  then $\frac{\partial g(o;\beta,\eta)}{\partial \beta}$ is Neyman orthogonal on $\mathcal{B}_1$. (b) Suppose  $g_1(o;\beta,\eta)$ and $g_2(o;\beta,\eta)$ are both Neyman orthogonal on $\mathcal{B}_1$, and let $h_1(\beta)$ and $h_2(\beta)$ be functions of $\beta$.  Then the linear combination $g_1(o;\beta,\eta)h_1(\beta)+g_2(o;\beta,\eta)h_2(\beta)$ also satisfies Neyman orthogonality on $\mathcal{B}_1$.
\end{lemma}

\begin{proof}
    For (a), the existence of integrable $h$ such that $\vert \frac{\partial g(o;\beta,\eta)}{\partial \beta}\vert\leq h(o)$ implies derivatives and expectations can be exchanged. Therefore,
    \begin{align*}
        &\frac{\partial}{\partial t}\mathbb{E}_{P_{0,N}}\Bigg[\frac{\partial}{\partial \beta}g(O;\beta,(1-t)\eta_{0,N}+t\eta_{1,N})\Bigg]\Bigg\vert_{t=0}\\
        =&\frac{\partial}{\partial \beta}\Bigg\{\frac{\partial}{\partial t}\mathbb{E}_{P_{0,N}}\Bigg[g(O;\beta,(1-t)\eta_{0,N}+t\eta_{1,N})\Bigg]\Bigg\vert_{t=0}\Bigg\}\\
        =&0.
    \end{align*}
    (b) is obvious.
\end{proof}

\begin{lemma}
    Suppose $g(o;\beta,\eta)$ is a separable moment function such that
    \begin{equation*}
        g(o;\beta,\eta) = \sum_{b=1}^Bg^{[b]}(o;\eta)h^{[b]}(\beta).
    \end{equation*}
    Furthermore, suppose $h^{[b]}(\beta), b = 1,...,B$ are continuously differentiable functions of $\beta$. Then under Assumption \ref{Assumption: separable score},
    \begin{align*}
    &\sup_{\beta \in \mathcal{B}}\Big\Vert (\mathbb{P}_{n,l}-P_{0,N}) (g(\beta,\widehat{\eta}_l)-g(\beta,{\eta}_{0,N})) \Big\Vert = O_{P_{0,N}}\Big(\sqrt{\frac{m}{N}}\delta_N\Big),\\
    &\sup_{\beta \in \mathcal{B}}\Vert (\mathbb{P}_{n,l}-P_{0,N})(G(\beta,\widehat{\eta}_{l})-G(\beta,\eta_{0,N})) \Vert = O_{P_{0,N}}\Bigg(\sqrt{\frac{m}{N}}\delta_N\Bigg).
\end{align*}
\label{Lemma: empirical process term of g, separable case}
\end{lemma}
\begin{proof}
Direct calculation yields
    \begin{align*}
        & \sup_{\beta \in \mathcal{B}}\Big\Vert (\mathbb{P}_{n,l}-P_{0,N}) (g(\beta,\widehat{\eta}_l)-g(\beta,{\eta}_{0,N})) \Big\Vert\\
         = &\sup_{\beta \in \mathcal{B}}\left\Vert (\mathbb{P}_{n,l}-P_{0,N}) \left(\sum_{i=1}^B \left(g^{[b]}(\widehat{\eta}_l)-g^{[b]}({\eta}_{0,N})\right)h^{[b]}(\beta))\right) \right\Vert\\
         \leq & \sum_{b=1}^B \left\Vert (\mathbb{P}_{n,l}-P_{0,N})  \left(g^{[b]}(\widehat{\eta}_l)-g^{[b]}({\eta}_{0,N})\right) \right\Vert \sup_{\beta \in B}\left\Vert h^{[b]}(\beta))\right\Vert\\
         \lesssim & \sum_{b=1}^B \left\Vert (\mathbb{P}_{n,l}-P_{0,N})  \left(g^{[b]}(\widehat{\eta}_l)-g^{[b]}({\eta}_{0,N})\right) \right\Vert,
    \end{align*}
    where the third inequality is due to triangle inequality, the fourth inequality is due to the boundedness of $h^{[b]}(\beta)$. 

    For any $b\in \{1,...,B\}$, we consider
    \begin{align*}
       &\frac{1}{{N}}\mathbb{E}_{P_{0,N}}\left[ \left\Vert \mathbb{G}_{n,l}  \left(g^{[b]}(\widehat{\eta}_l)-g^{[b]}({\eta}_{0,N})\right) \right\Vert^2\Bigg\vert O_i, i\in I^c_l\right]\\
       \leq &\frac{1}{{N}}\mathbb{E}_{P_{0,N}}\left[ \left\Vert g^{[b]}(\widehat{\eta}_l)-g^{[b]}({\eta}_{0,N})\right\Vert^2_F\Bigg\vert O_i, i\in I^c_l\right]\\
       \leq &\sup_{\eta\in\mathcal{T}_N}\frac{1}{{N}}\mathbb{E}_{P_{0,N}}\left[ \left\Vert  g^{[b]}(\eta)-g^{[b]}({\eta}_{0,N})\right\Vert^2_F\Bigg\vert O_i, i\in I^c_l\right]\\
      \leq &  \sup_{\eta\in\mathcal{T}_N}\frac{1}{\sqrt{N}}\mathbb{E}_{P_{0,N}}\left\Vert  g^{[b]}(\eta)-g^{[b]}({\eta}_{0,N}) \right\Vert^2_F\\
       = & \sup_{\eta\in\mathcal{T}_N}\frac{1}{{N}}\sum_{j=1}^m\sum_{r=1}^q\mathbb{E}_{P_{0,N}}\left\vert  g^{[b](j,r)}(\eta)-g^{[b](j,r)}({\eta}_{0,N}) \right\vert^2\\
       \lesssim & \frac{m\delta^2_N}{N}.
    \end{align*}
    where the last line is due to Assumption \ref{Assumption: separable score}. Therefore,
    \begin{equation*}
        \sup_{\beta \in \mathcal{B}}\Big\Vert (\mathbb{P}_{n,l}-P_{0,N}) (g(\beta,\widehat{\eta}_l)-g(\beta,{\eta}_{0,N})) \Big\Vert = O_p\left(\sqrt{\frac{m}{N}}\delta_N\right).
    \end{equation*}
    Similarly, we can prove that
    \begin{equation*}
        \sup_{\beta \in \mathcal{B}}\Vert (\mathbb{P}_{n,l}-P_{0,N})(G(\beta,\widehat{\eta}_{l})-G(\beta,\eta_{0,N})) \Vert = O_{p}\Bigg(\sqrt{\frac{m}{N}}\delta_N\Bigg).
    \end{equation*}
\end{proof}

\begin{lemma} 
Suppose Assumption 5 holds, we have
    \begin{align*}
    \sup_{\beta \in \mathcal{B}}\Big\Vert (\mathbb{P}_{n,l}-P_{0,N}) (g(\beta,\widehat{\eta}_l)-g(\beta,{\eta}_{0,N})) \Big\Vert = O_{p}\Big(\sqrt{\frac{m}{N}}\delta_N\Big).
\end{align*}
Suppose Assumption 8 holds, we further have
\begin{equation*}
    \sup_{\beta \in \mathcal{B}}\Vert (\mathbb{P}_{n,l}-P_{0,N})(G(\beta,\widehat{\eta}_{l})-G(\beta,\eta_{0,N})) \Vert = O_{p}\Bigg(\sqrt{\frac{m}{N}}\delta_N\Bigg).
\end{equation*}
\label{lemma: empirical process term of g}
\end{lemma}

\begin{proof}
First we notice that
    \begin{align*}
        &\sup_{\beta \in \mathcal{B}}\Big\Vert \mathbb{G}_{n,l} (g(\beta,\widehat{\eta}_l)-g(\beta,{\eta}_{0,N})) \Big\Vert\\
        =&\sup_{\beta \in \mathcal{B}}\sqrt{\sum_{j=1}^m \Big(\mathbb{G}_{n,l} (g^{(j)}(\beta,\widehat{\eta}_l)-g^{(j)}(\beta,{\eta}_{0,N}))\Big)^2}\\
        \leq & \sqrt{\sum_{j=1}^m \Big(\sup_{\beta \in \mathcal{B}}\Big\vert\mathbb{G}_{n,l} (g^{(j)}(\beta,\widehat{\eta}_l)-g^{(j)}(\beta,{\eta}_{0,N}))\Big\vert\Big)^2} \label{lemma sqrt}
    \end{align*}

    Since we have already assumed that, for any $1\leq j\leq m$, and $\eta\in \mathcal{T}_N$ the function class
    \begin{align*}
        \mathcal{G}^{(j)}_\eta = \{(g^{(j)}(\beta,\eta)-g^{(j)}(\beta,{\eta}_{0,N})), \beta \in B\}
    \end{align*}
    is a VC-type class such that
    \begin{align*}
        \log \sup_{Q} N(\epsilon \Vert F\Vert_{Q,2},\mathcal{G}_\eta^{(j)}, \Vert \cdot \Vert_{Q,2})\leq v\log(a/\epsilon), \quad 0< \epsilon <1.
    \end{align*}
    
    We apply Lemma \ref{lemma: concentration of supremum ep} conditional on $I^c_l$, so $\widehat{\eta}_l$ can be treated as fixed. then we have that with probability $1-mN^{-1}$,
    \begin{align*}
        &\sup_{g\in \mathcal{G}_{\eta}^{(j)}}|\mathbb{G}_{n,l} g|\\
        \leq & (1+\alpha)K\Bigg(\sigma\sqrt{v\log\Big(\frac{a\Vert F\Vert_{P,2}}{\sigma}\Big)}+\frac{v\Vert M\Vert_{P,2}}{\sqrt{N}}\log \Big(\frac{a\Vert F\Vert_{P,2}}{\sigma}\Big) \Bigg) +\\
     &K(q)\Bigg(\Big(\sigma+\frac{\Vert M \Vert_{P,q}}{\sqrt{N}}\Big)N^{\frac{1}{q}}+\frac{1}{\alpha \sqrt{N}}\Vert M \Vert_{P,2}N^{\frac{2}{q}}\Bigg)
    \end{align*}
    holds for all $1\leq j\leq m$ simultaneously and for any $\alpha>0$, $\sup_{g\in \mathcal{G}^{(j)}_\eta} \Vert g\Vert_{P,2}^2\leq \sigma^2\leq \Vert F\Vert_{P,2}^2$. By the assumption $\Vert F\Vert_{P,q}<\infty$ for all $q$, taking $\sigma$ as $r^{(j)}_N$, we have with probability larger than $1-mN^{-1}$,

    \small\begin{equation*}
        \begin{split}
            &\sup_{g\in \mathcal{G}_{\eta}^{(j)}}|\mathbb{G}_n g|\\
        \leq & (1+\alpha)K\Bigg(r^{(j)}_N\sqrt{v\log\Big(\frac{a\Vert F\Vert_{P_{0,N},2}}{r^{(j)}_N}\Big)}+\frac{vN^{1/q}\Vert F\Vert_{P_{0,N},q}}{\sqrt{N}}\log \Big(\frac{a\Vert F\Vert_{P_{0,N},2}}{r^{(j)}_N}\Big) \Bigg) \\
        &+K(q)\Bigg(\Big(r^{(j)}_N+\frac{N^{1/q}\Vert F \Vert_{P_{0,N},q}}{\sqrt{N}}\Big)N^{\frac{1}{q}}+\frac{N^{1/q}\Vert F \Vert_{P_{0,N},q}}{\alpha \sqrt{N}}N^{\frac{2}{q}}\Bigg)\\
        \lesssim & r^{(j)}_N\sqrt{\log \frac{1}{r^{(j)}_N}}+N^{-1/2+\epsilon}\log \frac{1}{r^{(j)}_N}+r^{(j)}_NN^{\epsilon}+N^{-1/2+\epsilon}\\
        \lesssim & r^{(j)}_N \Bigg(\sqrt{\log\frac{1}{r^{(j)}_N}}+N^{\epsilon}\Bigg) + N^{-1/2+\epsilon}\left(\frac{1}{\log r^{(j)}_N}+1\right)\\
        \leq & \delta_N
        \end{split}
    \end{equation*}\normalsize
for all $1\leq j\leq m$. For the second inequality, we choose $q$ such that $q\geq 3/\epsilon$.

To sum up,
    \begin{equation*}
        \sup_{\beta \in \mathcal{B}}\Big\Vert \mathbb{G}_{n,l} (g(\beta,\widehat{\eta}_l)-g(\beta,{\eta}_{0,N})) \Big\Vert = O_{p}(\sqrt{m}\delta_N).
    \end{equation*}

    Similarly, we can prove that
    \begin{equation*}
        \sup_{\beta \in \mathcal{B}}\Big\Vert \mathbb{G}_{n,l} (G(\beta,\widehat{\eta}_l)-G(\beta,{\eta}_{0,N})) \Big\Vert = O_{p}(\sqrt{m}\delta_N)
    \end{equation*}
    under Assumption 8.
\end{proof}

\begin{lemma}
Suppose $g$ is global Neyman orthogonal and either Assumption \ref{Assumption: consistency, nuisance estimation} or Assumption \ref{Assumption: general score regularity} holds, then for $k\in I_l$
    \begin{equation*}
        \sup_{\beta \in \mathcal{B}} \Big\Vert \mathbb{E}_{P_{0,N}}[g(O_k;\beta,\widehat{\eta}_l)|(O_i)_{i\in I_l^c}]-\mathbb{E}_{P_{0,N}}[g(O_k;\beta,{\eta}_{0,N})] \Big\Vert = O_p(\sqrt{m}\lambda_N),
    \end{equation*}
    where $\lambda_N^{(j)}\leq \lambda_N$ for all $j$.
    \label{Lemma: rate for the first order bias}
\end{lemma}

\begin{proof}
    With probability $1-o(1)$, $\widehat{\eta}_l\in \mathcal{T}_N$. Notice that
    \begin{align*}
        &\sup_{\beta \in \mathcal{B}} \Big\Vert \mathbb{E}_{P_{0,N}}[g(O_k;\beta,\widehat{\eta}_l)|(O_i)_{i\in I_l^c}]-\mathbb{E}_{P_{0,N}}[g(O_k;\beta,{\eta}_{0,N})] \Big\Vert\\
        =&\sup_{\beta \in \mathcal{B}}\sqrt{\sum_{j=1}^m(\mathbb{E}_{P_{0,N}}[g^{(j)}(O_k;\beta,\widehat{\eta}_l)|(O_i)_{i\in I_l^c}]-\mathbb{E}_{P_{0,N}}[g^{(j)}(O_k;\beta,{\eta}_{0,N})])^2}\\
        \leq & \sqrt{\sum_{j=1}^m(\sup_{\beta \in \mathcal{B}}\vert \mathbb{E}_{P_{0,N}}[g^{(j)}(O_k;\beta,\widehat{\eta}_l)|(O_i)_{i\in I_l^c}]-\mathbb{E}_{P_{0,N}}[g^{(j)}(O_k;\beta,{\eta}_{0,N})]\vert)^2}
    \end{align*}
    Therefore, we only need to bound
    \begin{equation*}
        \sup_{\beta \in \mathcal{B}}\vert \mathbb{E}_{P_{0,N}}[g^{(j)}(O_k;\beta,\widehat{\eta}_l)|(O_i)_{i\in I_l^c}]-\mathbb{E}_{P_{0,N}}[g^{(j)}(O_k;\beta,{\eta}_{0,N})]\vert
    \end{equation*}
    for every $j$. Let $f_l(t;\beta,j) = \mathbb{E}_{P_{0,N}}[g^{(j)}(O_k;\beta,\eta_{0,N}+t(\widehat{\eta}_l-\eta_{0,N}))|(O_i)_{i\in I_l^c}]-\mathbb{E}_{P_{0,N}}[g^{(j)}(O_k;\beta,{\eta}_{0,N})]$, $t\in [0,1]$. Then by Taylor expansion, we have for some $\tilde{t}\in [0,t]$,
    \begin{equation*}
        f_l(t;\beta,j) = f_l(0;\beta,j)+f'_l(0;\beta,j)+f''_l(\tilde{t};\beta,j).
    \end{equation*}
    Notice that
    \begin{align*}
        &\vert f_l(0;\beta,j)\vert =\vert \mathbb{E}_{P_{0,N}}[g^{(j)}(O_k;\beta,\eta_{0,N})|(O_i)_{i\in I_l^c}]-E[g^{(j)}(O_k;\beta,{\eta}_{0,N})]\vert = 0, \\
        &\vert f'_l(0;\beta,j) \vert\leq \sup_{\eta\in \mathcal{T}_N} \left\vert\frac{\partial }{\partial t}\mathbb{E}_{P_{0,N}}g^{(j)}(O;\beta,\eta_{0,N}+t(\eta-\eta_{0,N}))\right\vert =0, \\
        & \sup_{\beta \in \mathcal{B}}\vert f''_l(\tilde{t};\beta,j) \vert \leq \sup_{\beta\in \mathcal{B},t\in(0,1),\eta\in \mathcal{T}_N}\left\vert \frac{\partial^2 }{\partial t^2}\mathbb{E}_{P_{0,N}}g^{(j)}(O;\beta,\eta_0+t(\eta-\eta_{0,N}))\right\vert\leq \lambda_N,
    \end{align*}
    where the first equation is due to $\mathbb{E}_{P_{0,N}}[g^{(j)}(O_k;\beta,\eta_{0,N})|(O_i)_{i\in I_l^c}]=\mathbb{E}[g^{(j)}(O_k;\beta,\eta_{0,N})]$, the second line is due to Neyman orthogonality, the third line is due to Assumption \ref{Assumption: consistency, nuisance estimation} or Assumption \ref{Assumption: general score regularity}.  Here we $\lambda_N$ is a upper bound for $\lambda^{(j)}_N$, which can be taken as $N^{-1/2}\delta_N$ under Assumptions  \ref{Assumption: consistency, nuisance estimation} or \ref{Assumption: general score regularity}. Therefore, 
    \begin{equation*}
        \sup_{\beta \in \mathcal{B}} \Big\Vert \mathbb{E}_{P_{0,N}}[g(O_k;\beta,\widehat{\eta}_l)|(O_i)_{i\in I_l^c}]-\mathbb{E}_{P_{0,N}}[g(O_k;\beta,{\eta}_0)] \Big\Vert\leq \sqrt{m\lambda^{2}_N} = \sqrt{m}\lambda_N
    \end{equation*}
    holds with probability $1-o(1)$, hence
    \begin{equation*}
        \sup_{\beta \in \mathcal{B}} \Big\Vert \mathbb{E}_{P_{0,N}}[g(O_k;\beta,\widehat{\eta}_l)|(O_i)_{i\in I_l^c}]-\mathbb{E}_{P_{0,N}}[g(O_k;\beta,{\eta}_{0,N})] \Big\Vert = O_p(\sqrt{m}\lambda_N)
    \end{equation*}.
\end{proof}

\begin{lemma}
Suppose $g$ is global Neyman orthogonal, and either Assumption \ref{Assumption: consistency, nuisance estimation} or Assumption \ref{Assumption: general score regularity} holds, we have
\begin{equation*}
    \sup_{\beta\in \mathcal{B}}\Vert \widehat{g}(\beta,\widehat{\eta})-\widehat{g}(\beta,{\eta}_{0,N}) \Vert = O_p\Big(\sqrt{\frac{m}{N}}\delta_N+\sqrt{m}\lambda_N\Big) 
\end{equation*} \label{lemma: ghat rate}
\end{lemma}

\begin{proof}
    We only need to prove that 
    \begin{align*}
        \sup_{\beta\in \mathcal{B}}\Vert \widehat{g}_l(\beta,\widehat{\eta}_l)-\widehat{g}_l(\beta,{\eta}_{0,N}) \Vert = O_p\Big(\sqrt{\frac{m}{N}}\delta_N+\sqrt{m}\lambda_N\Big),
    \end{align*}
    where $\widehat{g}_l(\beta,\eta) = \mathbb{P}_{n,l}g(\beta,\eta)$. The result follows by noticing that
    \begin{align*}
        &\sup_{\beta\in \mathcal{B}}\Vert \widehat{g}_l(\beta,\widehat{\eta}_l)-\widehat{g}_l(\beta,{\eta}_{0,N}) \Vert\\
        \leq & \frac{1}{\sqrt{N}}\sup_{\beta \in \mathcal{B}}\Big\Vert \mathbb{G}_{n,l} (g(\beta,\widehat{\eta}_l)-g(\beta,{\eta}_{0,N})) \Big\Vert + \sup_{\beta \in \mathcal{B}} \Big\Vert \mathbb{E}_{P_{0,N}}[g(O_k;\beta,\widehat{\eta}_l)|(O_i)_{i\in I_l^c}]-\mathbb{E}_{P_{0,N}}[g(O_k;\beta,{\eta}_{0,N})] \Big\Vert
    \end{align*}
\end{proof}

\begin{remark}
    Under Assumption \ref{Assumption: general score regularity}, we further have
    \begin{equation*}
        \sup_{\beta\in \mathcal{B}}\Vert \widehat{g}(\beta,\widehat{\eta})-\widehat{g}(\beta,{\eta}_{0,N}) \Vert = o_P(1/\sqrt{N}).
    \end{equation*}
\end{remark}

\begin{lemma}
Under Assumption \ref{Assumption: consistency, nuisance estimation}, we have
    \begin{align*}
        \Vert \widehat{\Omega}(\beta_0,\widehat{\eta})-\overline{\Omega}(\beta_0,\eta_{0,N})\Vert = O_p\Big(\frac{m}{\sqrt{N}}\delta_N\Big). 
    \end{align*} \label{lemma: Omega uniform convergence}
\end{lemma}

\begin{proof}
We only need to derive the convergence rate for  $\Vert \widehat{\Omega}_l(\beta_0,\widehat{\eta}_l)-\overline{\Omega}(\beta_0,\eta_{0,N})\Vert$ for some $l$, where $\widehat{\Omega}_l(\beta,{\eta}) = \mathbb{P}_{n,l}\Omega(\beta,\eta)$. Note that
\begin{align*}
    &\Vert \widehat{\Omega}_l(\beta_0,\widehat{\eta}_l)-\overline{\Omega}(\beta_0,\eta_{0,N})\Vert\\
    =& \Vert (\mathbb{P}_{n,l}-P_{0,N})(\Omega(\beta_0,\widehat{\eta}_l)-\Omega(\beta_0,{\eta}_{0,N})) \Vert + \Vert (\mathbb{P}_{n,l}-P_{0,N})\Omega(\beta_0,{\eta}_{0,N}) \Vert  + \Vert P_{0,N}[\Omega(\beta_0,\widehat{\eta}_l)-\overline{\Omega}(\beta_0,\eta_{0,N})] \Vert
\end{align*}

To derive the convergence rate for each term, we will vectorize the matrice $\widehat{\Omega}_l(\beta_0,\widehat{\eta}_l)$ and $\overline{\Omega}(\beta_0,\eta_{0,N})$ as $m^2$ dimensional vectors. Then we can apply the same techniques we used in Lemma \ref{lemma: empirical process term of g} and \ref{Lemma: rate for the first order bias}. For the first term, 
    \begin{align*}
        &\mathbb{E}_{P_{0,N}}\Vert \sqrt{n}(\mathbb{P}_{n,l}-P_{0,N})(\Omega(\beta_0,\widehat{\eta}_l)-\Omega(\beta_0,{\eta}_{0,N})) \Vert^2\\
        \leq & \mathbb{E}_{P_{0,N}}\Vert \widehat{\Omega}_l(\beta_0,\widehat{\eta}_l)-\overline{\Omega}(\beta_0,\eta_{0,N})\Vert_F^2\\
        \leq & \sup_{\eta\in \mathcal{T}_n}\mathbb{E}_{P_{0,N}}[\Vert \text{vec}(\widehat{\Omega}_l(\beta_0,{\eta}))-\text{vec}(\overline{\Omega}(\beta_0,\eta_{0,N}))\Vert^2] \\
        \leq & \sum_{j=1}^m \sum_{k=1}^m r^{(j,k)}_n
        \leq  m^2\delta^2_N.
    \end{align*}
    The second line is because the spectral norm is always upper bounded by the Frobenius Norm. The third line is due to cross-fitting, so that the randomness of $\widehat{\eta}_l$ can be ignored. Therefore,
    \begin{align*}
        \Vert (\mathbb{P}_{n,l}-P_{0,N})(\Omega(\beta_0,\widehat{\eta}_l)-\Omega(\beta_0,{\eta}_{0,N})) \Vert = O_p\Big(\frac{m}{\sqrt{N}}\delta_N\Big).
    \end{align*}
\end{proof}

\begin{remark}
    Under the assumption $\delta_N = o(1/\sqrt{m})$, we further have
    \begin{equation*}
        \sup_{\beta\in \mathcal{B}}\Vert (\mathbb{P}_{n,l}-P_{0,N})(\Omega(\beta,\widehat{\eta}_l)-\Omega(\beta,{\eta}_{0,N})) \Vert=o_p(1/\sqrt{m}).
    \end{equation*}
\end{remark}

\begin{lemma}
Under Assumption \ref{Assumption: consistency, nuisance estimation}, we have
    \begin{align*}
        \sup_{\beta \in \mathcal{B}}\Vert (\mathbb{P}_{n}-P_{0,N})(\Omega(\beta,\widehat{\eta})-\Omega(\beta,{\eta}_{0,N})) \Vert = O_p\Big(\frac{m}{\sqrt{N}}\delta_N\Big).  
    \end{align*} 
    If we further assume Assumption 
    \ref{Assumption: general score regularity} or Assumption \ref{Assumption: separable score} holds, then
    \begin{align*}
        \sup_{\beta \in \mathcal{B}}\Vert (\mathbb{P}_{n}-P_{0,N})(\Omega(\beta,\widehat{\eta})-\Omega(\beta,{\eta}_{0,N})) \Vert = o_p\Big(\sqrt{\frac{m}{N}}\Big).  
    \end{align*}
    \label{Lemma: Empirical process term for matrix}
\end{lemma}

\begin{proof}
 We will vectorize the matrices ${\Omega}(\beta,\widehat{\eta}_l)$ and ${\Omega}(\beta,\eta_{0,N})$ as $m^2$ dimensional vectors. Then we can apply the same techniques we used in Lemma \ref{lemma: empirical process term of g} and \ref{Lemma: rate for the first order bias}. For the first term, we have
\begin{align*}
    &\sup_{\beta\in \mathcal{B}}\Vert (\mathbb{P}_{n,l}-P_{0,N})(\Omega(\beta,\widehat{\eta}_l)-\Omega(\beta,{\eta}_{0,N}))\Vert\\
    \leq & \sup_{\beta\in \mathcal{B}}\Vert (\mathbb{P}_{n,l}-P_{0,N})(\Omega(\beta,\widehat{\eta}_l)-\Omega(\beta,{\eta}_{0,N}))\Vert_F\\
    = &\sup_{\beta\in \mathcal{B}}\Vert (\mathbb{P}_{n,l}-P_{0,N})(\text{vec}(\Omega(\beta,\widehat{\eta}_l))-\text{vec}(\Omega(\beta,{\eta}_{0,N})))\Vert
\end{align*}
Note that $\text{vec}(\Omega(\beta,\eta))$ is a $m^2$-dimensional vector with each component of the form $g^{(j)}(\beta,\eta)g^{(k)}(\beta,\eta)$, with $i,k \in\{1,...,m\}$. Following the similar derivations as in \ref{lemma: empirical process term of g}, we have
\begin{align*}
    \sup_{\beta \in \mathcal{B}}\Vert (P_{n,l}-P)(\Omega(\beta,\widehat{\eta}_l)-\Omega(\beta,{\eta}_{0,N}))\Vert = O_p\Big(\sqrt{\frac{m^2}{N}}\delta_N\Big),
\end{align*}
The only difference here is that $\text{vec}(\Omega(\beta,\eta))$ is a $m^2$-dimensional vector, so here the size is $O_p\Big(\sqrt{\frac{m^2}{N}}\delta_N\Big)$ instead of $O_p\Big(\sqrt{\frac{m}{N}}\delta_N\Big)$. If we further assume Assumption \ref{Assumption: general score regularity} holds, then  
\begin{align*}
        \sup_{\beta \in \mathcal{B}}\Vert (\mathbb{P}_{n,l}-P_{0,N})(\Omega(\beta,\widehat{\eta}_l)-\Omega(\beta,{\eta}_{0,N})) \Vert = o_p\Big(\sqrt{\frac{m}{N}}\Big).  
    \end{align*}
     since $\delta_N = o(1/\sqrt{m})$. The proof under Assumption \ref{Assumption: separable score} holds is very similar to that of Lemma \ref{Lemma: empirical process term of g, separable case} and is therefore omitted.
\end{proof}

\begin{remark}
    Combing Lemma \ref{Lemma: Empirical process term for matrix} with Assumption \ref{Assumption: consistency, matrix estimation} or Assumption \ref{Assumption: ASN matrix estimation} allows us to drive convergence rate of 
    \begin{equation*}
        \Vert \widehat{\Omega}(\beta,\widehat{\eta})-\overline{\Omega}(\beta,\eta_{0,N})\Vert
    \end{equation*}
    by noticing that 
    \begin{align*}
    &\Vert \widehat{\Omega}_l(\beta,\widehat{\eta}_l)-\overline{\Omega}(\beta,\eta_{0,N})\Vert\\
    =& \Vert (\mathbb{P}_{n,l}-P_{0,N})(\Omega(\beta,\widehat{\eta}_l)-\Omega(\beta,{\eta}_{0,N})) \Vert + \Vert (\mathbb{P}_{n,l}-P_{0,N})\Omega(\beta,{\eta}_{0,N}) \Vert  + \Vert P_{0,N}[\Omega(\beta,\widehat{\eta}_l)-{\Omega}(\beta,\eta_{0,N})] \Vert. 
\end{align*}\label{Remark: Decompositions of matrix terms}
\end{remark}

\begin{lemma}
Under Assumption \ref{Assumption: convergence of g-hat}, 
\begin{equation*}
    \Big\Vert \widehat{g}(\beta_0,\eta_{0,N}) \Big\Vert = O_p\Big(\sqrt{\frac{m}{N}}\Big).
\end{equation*}
    \label{lemma: g hat rate 2}
\end{lemma}

\begin{proof}
We only need to bound $\mathbb{E}_{P_{0,N}}\Big\Vert \widehat{g}(\beta_0,\eta_{0,N}) \Big\Vert^2$:
    \begin{align*}
         & \mathbb{E}_{P_{0,N}}\left\Vert\widehat{g}(\beta_0,\eta_{0,N}) \right\Vert^2 = \mathbb{E}_{P_{0,N}}\widehat{g}^T(\beta_0,\eta_{0,N})\widehat{g}(\beta_0,\eta_{0,N})=\frac{1}{N}\mathbb{E}_{P_{0,N}}\text{tr}[g^T(\beta_0,\eta_{0,N})g(\beta_0,\eta_{0,N})]\\
         =& \frac{1}{N}\mathbb{E}_{P_{0,N}}\text{tr}[g(O;\beta_0,\eta_{0,N})g^T(\beta_0,\eta_{0,N})] = \frac{1}{N}\text{tr}\widebar{\Omega}(\beta_0,\eta_{0,N})\lesssim \frac{m}{N}
    \end{align*}
    Therefore,
    \begin{equation*}
        \Big\Vert \widehat{g}(\beta_0,\eta_{0,N}) \Big\Vert = O_p\Big(\sqrt{\frac{m}{N}}\Big).
    \end{equation*}
\end{proof}

\begin{lemma}
Suppose $m^2/N\rightarrow 0$, $m^2/\mu_N=O(1)$ and Assumptions \ref{Assumption: weak moment condition},\ref{Assumption: global identifiability},\ref{Assumption: convergence of g-hat},\ref{Assumption: Global Neyman orthogonality},\ref{Assumption: consistency, nuisance estimation},\ref{Assumption: consistency, matrix estimation} hold, then
    \begin{align*}
        \Vert \delta(\widehat{\beta})\Vert = O_p(1).
    \end{align*}
\end{lemma}
\begin{proof}
Direct calculation yields
\begin{align*}
    &\Vert \widehat{g}(\widehat{\beta},\widehat{\eta})\Vert^2 = \widehat{g}(\widehat{\beta},\widehat{\eta})^T\widehat{g}(\widehat{\beta},\widehat{\eta})\\
    \lesssim & \widehat{g}(\widehat{\beta},\widehat{\eta})^T\widehat{\Omega}(\widehat{\beta},\widehat{\eta})^{-1}\widehat{g}(\widehat{\beta},\widehat{\eta})=\widehat{Q}(\widehat{\beta},\widehat{\eta})\leq \widehat{Q}({\beta}_0,\widehat{\eta})\\
    \leq & \vert\widehat{Q}({\beta}_0,\widehat{\eta})-\widehat{Q}({\beta}_0,{\eta}_0)\vert+\widehat{Q}({\beta}_0,{\eta}_{0,N})\\
    \leq & \vert (\widehat{g}(\beta_0,\widehat{\eta})-\widehat{g}(\beta_0,\eta_{0,N})) \widehat{\Omega}^{-1}(\beta_0,\widehat{\eta})\widehat{g}(\beta_0,\widehat{\eta})\vert\\
    &+\vert \widehat{g}(\beta_0,\eta_{0,N})\widehat{\Omega}(\beta_0,\widehat{\eta})^{-1}(\widehat{\Omega}(\beta_0,\widehat{\eta})-\widehat{\Omega}(\beta_0,{\eta}_{0,N}))\widehat{\Omega}(\beta_0,{\eta}_{0,N})\vert\\
    &+\vert \widehat{g}(\beta_0,\eta_{0,N}) \widehat{\Omega}^{-1}(\beta_0,\widehat{\eta})(\widehat{g}(\beta_0,\widehat{\eta})-\widehat{g}(\beta_0,{\eta}_{0,N}))\vert +\widehat{Q}(\beta_0,\eta_{0,N}).
\end{align*}
Since
\begin{align*}
    &\Vert \widehat{g}(\beta_0,\widehat{\eta})-\widehat{g}(\beta_0,\eta_{0,N})\Vert = O_p\Big(\sqrt{\frac{m}{N}}\delta_N\Big),\\
    &\Vert \widehat{\Omega}(\beta_0,\widehat{\eta})-\widehat{\Omega}(\beta_0,{\eta}_{0,N})\Vert \leq \Vert \widehat{\Omega}(\beta_0,\widehat{\eta})-\widebar{\Omega}(\beta_0,{\eta}_{0,N})\Vert+\Vert\widebar{\Omega}(\beta_0,{\eta}_{0,N})-\widehat{\Omega}(\beta_0,{\eta}_{0,N})\Vert =o_p(1),\\
    &\Vert \widehat{g}(\beta_0,\eta_{0,N})\Vert = O_p\Big(\sqrt{\frac{m}{N}}\Big), \\
    &\widehat{Q}(\beta_0,\eta_{0,N})\lesssim \Vert \widehat{g}(\beta_0,\eta_{0,N})\Vert^2 = O_p\Big(\frac{m}{N}\Big),
\end{align*}
we have $\Vert \widehat{g}(\widehat{\beta},\widehat{\eta})\Vert = O_p(\sqrt{m/N})$. Therefore,
\begin{align*}
    \Vert \delta(\widehat{\beta})\Vert \leq& C\sqrt{N}\Vert \widehat{g}(\widehat{\beta},\eta_{0,N})\Vert/\mu_N+\widehat{M}\\
    \leq &C\sqrt{N}\Vert \widehat{g}(\widehat{\beta},\eta_{0,N})- \widehat{g}(\widehat{\beta},\widehat{\eta})\Vert/\mu_N+C\sqrt{N}\Vert \widehat{g}(\widehat{\beta},\widehat{\eta})\Vert/\mu_N+\widehat{M}\\
    = O_p(1).
\end{align*}
\end{proof}

\begin{lemma}
Assume conditions in Theorem \ref{Theorem: Consistency} hold, then
    \begin{equation*}
        \sup_{\beta \in \mathcal{B},\delta(\beta)\leq C}\Big\Vert \widehat{g}(\beta,\eta_{0,N}) \Big\Vert = O_{p}\Big(\frac{\mu_N}{\sqrt{N}}\Big) 
    \end{equation*}
    \label{lemma: g hat rate 3}
\end{lemma}

\begin{proof}
    \begin{align*}
        & \sup_{\beta \in \mathcal{B},\delta(\beta)\leq C}\Big\Vert \widehat{g}(\beta,\eta_{0,N}) \Big\Vert
        \leq  \sup_{\beta \in \mathcal{B},\delta(\beta)\leq C}\Big\Vert \widehat{g}(\beta,\eta_{0,N})-\widehat{g}(\beta_0,\eta_{0,N}) \Big\Vert + \Big\Vert \widehat{g}(\beta_0,\eta_{0,N})\Big\Vert\\
        \leq&  \frac{\mu_N}{\sqrt{N}}\widehat{M}\sup_{\beta \in \mathcal{B},\delta(\beta)\leq C}\Vert \delta(\beta)\Vert +O_p\Big(\sqrt{\frac{m}{{N}}}\Big)
        \leq O_p\Big({\frac{\mu_N}{\sqrt{N}}}\Big),
    \end{align*}
    where the second last inequality is due to Assumption \ref{Assumption: convergence of g-hat}, the last inequality is due to boundedness of $m/\mu^2_N$.
\end{proof}

\begin{lemma}
Assume conditions in Theorem \ref{Theorem: Consistency} hold, then
    \begin{equation*}
        \sup_{\beta \in \mathcal{B},\delta(\beta)\leq C}\Big\Vert \widehat{g}(\beta,\widehat{\eta}) \Big\Vert = O_p\Big(\frac{\mu_N}{\sqrt{N}}\Big).
    \end{equation*}
    \label{lemma: g rate 4}
\end{lemma}
\begin{proof}
\begin{align*}
    \sup_{\beta \in \mathcal{B},\delta(\beta)\leq C}\Big\Vert \widehat{g}(\beta,\widehat{\eta}) \Big\Vert\leq &\sup_{\beta \in \mathcal{B},\delta(\beta)\leq C}\Big\Vert \widehat{g}(\beta,\widehat{\eta})-\widehat{g}(\beta,{\eta}_{0,N}) \Big\Vert+\sup_{\beta \in \mathcal{B},\delta(\beta)\leq C}\Big\Vert \widehat{g}(\beta,{\eta}_{0,N})-\widehat{g}(\beta_0,{\eta}_{0,N}) \Big\Vert+\Big\Vert\widehat{g}(\beta_0,{\eta}_{0,N}) \Big\Vert\\
=&O_p\Big(\sqrt{\frac{m}{N}}\delta_N+\sqrt{m}\lambda_N\Big)+O_p\Big(\sqrt{\frac{m}{N}}\Big)+O_P\Big(\frac{\mu_N}{\sqrt{N}}\Big)=O_P\Big(\frac{\mu_N}{\sqrt{N}}\Big).
\end{align*}
\end{proof}

\begin{lemma}
Suppose $g$ is global Neyman orthogonal and continuously differentiable in a neighborhood of $\beta_0$, denoted as $\mathcal{B}'$.  Suppose in addition Assumption \ref{Assumption: general score regularity} holds, we have
\begin{equation*}
    \sup_{\beta\in \mathcal{B}'}\Vert \widehat{G}(\beta,\widehat{\eta})-\widehat{G}(\beta,{\eta}_{0,N}) \Vert = O_p\Big(\frac{1}{\sqrt{N}}\Big) 
\end{equation*} \label{lemma: Ghat rate}
\end{lemma}

\begin{proof}
    since $g$ is global Neyman orthogonal  and $g$ is continuously differentiable in a neighborhood of $\mathcal{B}'$, we have $G$ is also  Neyman orthogonal in $\mathcal{B}'$. The rest of proof is similar to the proof of \ref{lemma: ghat rate}, thus omitted.
\end{proof}

\begin{lemma}
Assume conditions for Theorem \ref{Theorem: consistency and ASN} hold, then
    \begin{equation*}
N\mu_N^{-1}\frac{\partial \widehat{Q}(\beta,\widehat{\eta})}{\partial \beta}\Bigg|_{\beta = \beta_0}=N\mu_N^{-1}\frac{\partial \tilde{Q}(\beta,\widehat{\eta})}{\partial \beta}\Bigg|_{\beta = \beta_0} +o_p(1).
    \end{equation*}
    \label{Lemma: equivalence of FO derivative}
\end{lemma}

\begin{proof}
    It suffices to show that for $k-$th coordinate of $\beta$, denoted as $\beta^{(k)}$ ($k\in \{1,...,p\}$), we have
    \begin{equation*}
        N\mu_N^{-1}\frac{\partial \widehat{Q}(\beta,\widehat{\eta})}{\partial \beta^{(k)}}\Bigg|_{\beta = \beta_0}=N\mu_N^{-1}\frac{\partial \tilde{Q}(\beta,\widehat{\eta})}{\partial \beta^{(k)}}\Bigg|_{\beta = \beta_0} +o_p(1).
    \end{equation*}
    Following the same derivation as in \citet{ye2024genius}, we have
    \begin{align*}
        &\frac{\partial \widehat{Q}(\beta,\widehat{\eta})}{\partial \beta^{(k)}}\Bigg|_{\beta = \beta_0}= \widebar{G}^{(k),T}\widehat{\Omega}(\beta_0,\widehat{\eta})^{-1}\widehat{g}(\beta_0,\widehat{\eta})+(N^{-1}\sum_{i = 1}^N \check{U}_i^{(k)})^T\widehat{\Omega}(\beta_0,\widehat{\eta})^{-1}\widehat{g}(\beta,\widehat{\eta}) \\
        &\frac{\partial \tilde{Q}(\beta,\widehat{\eta})}{\partial \beta^{(k)}}\Bigg|_{\beta = \beta_0}= \widebar{G}^{(k),T}\Bar{\Omega}(\beta_0,{\eta}_{0,N})^{-1}\widehat{g}(\beta_0,\widehat{\eta})+(N^{-1}\sum_{i = 1}^N \tilde{U}_i^{(k)})^T\widebar{\Omega}(\beta_0,{\eta}_{0,N})^{-1}\widehat{g}(\beta,\widehat{\eta}) 
    \end{align*}
    where
    \begin{align*}
        &\tilde{U}_i^{(k)} = G_i^{(k)}(\beta_0,\widehat{\eta}_{l(i)})-\overline{G}^{(k)}-\mathbb{E}_{P_{0,N}}(G_i^{(k)}g_i^T)\Bar{\Omega}^{-1}g_i(\beta_0,\widehat{\eta}_{l(i)}),\\
        &\check{U}_i^{(k)} = G_i^{(k)}(\beta_0,\widehat{\eta}_{l(i)})-\overline{G}^{(k)}-\Big(N^{-1}\sum_{i=1}^N G_i^{(k)}(\beta_0,\widehat{\eta}_{l(i)})g_i(\beta_0,\widehat{\eta}_{l(i)})\Big)\widehat{\Omega}(\beta_0,\widehat{\eta}_{l(i)})^{-1}g_i(\beta_0,\widehat{\eta}_{l(i)}).
    \end{align*}

    By contrasting the explicit formula for $N\mu_N^{-1}\frac{\partial \widehat{Q}(\beta,\widehat{\eta})}{\partial \beta^{(k)}}\Bigg|_{\beta = \beta_0}$ and $N\mu_N^{-1}\frac{\partial \tilde{Q}(\beta,\widehat{\eta})}{\partial \beta^{(k)}}\Bigg|_{\beta = \beta_0}$, we need to show the following three conditions:
    \begin{enumerate}
        \item $N\mu_N^{-1}(N^{-1}\sum_{i=1}^N (\check{U}_i^{(k)}-\tilde{U}_i^{(k)}))^T\widehat{\Omega}(\beta,\widehat{\eta})^{-1}\widehat{g}(\beta_0,\widehat{\eta}) = o_p(1)$
        \item $N\mu_N^{-1}(N^{-1}\sum_{i = 1}^N \tilde{U}_i^{(k)})^T\{\widehat{\Omega}(\beta_0,\widehat{\eta})^{-1}-\widebar{\Omega}^{-1}(\beta_0,\eta_{0,N})\}\widehat{g}(\beta_0,\widehat{\eta})=o_p(1)$
        \item $N\mu_N^{-1}\overline{G}^{(k)T}\{\widehat{\Omega}(\beta_0,\widehat{\eta})^{-1}-\widebar{\Omega}^{-1}(\beta_0,\eta_{0,N})\}\widehat{g}(\beta_0,\widehat{\eta})=o_p(1)$.
    \end{enumerate}

    We firstly show 1., by derivations in \citet{ye2024genius},
    \begin{align*}
        &|N\mu_N^{-1}(N^{-1}\sum_{i=1}^N (\check{U}_i^{(k)}-\tilde{U}_i^{(k)})\widehat{\Omega}(\beta,\widehat{\eta})^{-1}\widehat{g}(\beta_0,\widehat{\eta})|\\
        \leq & C\mu_N^{-1}N \Vert \widehat{g}(\beta_0,\widehat{\eta})\Vert^2\Big\{\Vert \widehat{\Omega}(\beta_0,\widehat{\eta})-\widebar{\Omega})\Vert+\Vert N^{-1}\sum_{i=1}^N G_i^{(k)}(\beta_0,\widehat{\eta}_{l(i)})g_i(\beta_0,\widehat{\eta}_{l(i)})- E(G_i^{(k)}g_i^T)\Vert\Big\}\\
        =& o_p(1)
    \end{align*}
    Since $\Vert \widehat{g}(\beta_0,\widehat{\eta})\Vert\leq \Vert \widehat{g}(\beta_0,\widehat{\eta})-\widehat{g}(\beta_0,{\eta}_{0,N})\Vert + \Vert \widehat{g}(\beta_0,\eta_{0,N})\Vert= O_p(\sqrt{m/N})$, $\Vert \widehat{\Omega}(\beta_0,\widehat{\eta})-\Omega\Vert = o_p(1/\sqrt{m})$, $\Vert N^{-1}\sum_{i=1}^N G_i^{(k)}(\beta_0,\widehat{\eta}_{l(i)})g_i(\beta_0,\widehat{\eta}_{l(i)})- \mathbb{E}_{P_{0,N}}(G_i^{(k)}g_i^T)\Vert=o_p(1/\sqrt{m})$ and $\sqrt{m}/\mu_n=O(1)$. Here we will show that   $\Vert N^{-1}\sum_{i=1}^N G_i^{(k)}(\beta_0,\widehat{\eta}_{l(i)})g_i(\beta_0,\widehat{\eta}_{l(i)})- \mathbb{E}_{P_{0,N}}(G_i^{(k)}g_i^T)\Vert=o_p(1/\sqrt{m})$. 
    It suffices to show that
    \begin{align*}
        \Vert\mathbb{P}_{n,l} G^{(k)}(\beta_0,\widehat{\eta})g^T(\beta_0,\widehat{\eta})-P_{0,N}G^{(k)}(\beta_0,{\eta}_{0,N})g^T(\beta,{\eta}_{0,N})\Vert = o_p(1/\sqrt{m}).
    \end{align*}
Following the similar derivations in Lemma \ref{Lemma: Empirical process term for matrix} and Remark \ref{Remark: Decompositions of matrix terms}, we have
\begin{align*}
    &\Vert\mathbb{P}_{n,l} G^{(k)}(\beta_0,\widehat{\eta})g^T(\beta_0,\widehat{\eta})-P_{0,N}G^{(k)}(\beta_0,{\eta}_{0,N})g^T(\beta,{\eta}_{0,N})\Vert_F\\
    \leq & \Vert (\mathbb{P}_{n,l}-P_{0,N})(\text{vec}(G^{(k)}(\beta_0,\widehat{\eta})g^T(\beta_0,\widehat{\eta}))-\text{vec}(G^{(k)}(\beta_0,{\eta}_{0,N})g^T(\beta_0,{\eta}_{0,N}))) \Vert \\
    &+ \Vert (\mathbb{P}_{n,l}-P_{0,N})\text{vec}(G^{(k)}(\beta_0,{\eta}_{0,N})g^T(\beta_0,{\eta}_{0,N})) \Vert \\
    &+ \Vert P_{0,N}(\text{vec}(G^{(k)}(\beta_0,\widehat{\eta})g^T(\beta_0,\widehat{\eta}))-\text{vec}(G^{(k)}(\beta_0,{\eta}_{0,N})g^T(\beta_0,{\eta}_{0,N}))) \Vert \\
    \leq & o_p\left(\sqrt{\frac{m}{N}}\right)+o_p\left(\frac{1}{\sqrt{m}}\right) +o\left(\frac{1}{\sqrt{m}}\right)\\
    =&o_p\left(\frac{1}{\sqrt{m}}\right).
\end{align*}
For the last inequality, we used the similar derivations as in \ref{Lemma: Empirical process term for matrix} and Assumption \ref{Assumption: ASN matrix estimation}.

For 2., we first show that 
\begin{align*}
    \left\Vert \frac{1}{\sqrt{mN}}\sum_{i=1}^N\tilde{U}^{(k)}_i \right\Vert=O_{p}(1).
\end{align*}

Recall
\begin{align*}
    U^{(k),\perp}_i=G^{(k)}_i-\widebar{G}^{(k)}-\mathbb{E}_{P_{0,N}}[G^{(k)}_ig^T_i]\widebar{\Omega}^{-1}g_i.
\end{align*}
To simplify notation, we write $U^{(k)}_i$ instead of $U^{(k),\perp}_i$. Therefore, $U^{(k)}_i$ is the residual for projecting $G_i^{(k)}-\widebar{G}^{(k)}$ onto the space spanned by $g_i$. Therefore, $\mathbb{E}_{P_{0,N}}(\Vert U_i^{(k)}\Vert^2)\leq \mathbb{E}_{P_{0,N}}(\Vert G_i^{(k)}\Vert^2)\leq Cm^2$. Using Markov inequality, we have 
\begin{align*}
    \frac{1}{\sqrt{mN}}\left\Vert \sum_{i=1}^NU^{(k)}_i \right\Vert=O_{p}(1).
\end{align*}
Next,
\begin{align*}
    &\left\Vert \frac{1}{\sqrt{N}}\sum_{i=1}^N(\tilde{U}^{(k)}_i-{U}^{(k)}_i) \right\Vert\\
    =&\left\Vert \frac{1}{\sqrt{N}}\sum_{i=1}^N \left(G_i^{(k)}(\beta_0,\widehat{\eta}_{l(i)})-G_i^{(k)}(\beta_0,\eta_{0,N})-\mathbb{E}_{P_{0,N}}[G^{(k)}_ig_i^T]\widebar{\Omega}^{-1}(g_i(\beta_0,\widehat{\eta}_{l(i)})-g_i)\right) \right\Vert\\
    \lesssim & \sqrt{N}\left\Vert \widehat{G}^{(k)}(\beta_0,\widehat{\eta})-\widehat{G}^{(k)}(\beta_0,{\eta}_{0,N})  \right\Vert+\sqrt{N}\left\Vert \widehat{g}(\beta_0,\widehat{\eta})-\widehat{g}(\beta_0,{\eta}_{0,N})  \right\Vert=o_{p}(1).
\end{align*}
For the last line, we used the fact that
\begin{align*}
   & \left\Vert \widehat{G}^{(k)}(\beta_0,\widehat{\eta})-\widehat{G}^{(k)}(\beta_0,{\eta}_{0,N})  \right\Vert = o_p\left(\frac{1}{\sqrt{N}}\right),\\
   &\left\Vert \widehat{g}(\beta_0,\widehat{\eta})-\widehat{g}(\beta_0,{\eta}_{0,N}) \right\Vert=o_p\left(\frac{1}{\sqrt{N}}\right),
\end{align*}
which are results from Lemma \ref{lemma: ghat rate} and Lemma \ref{lemma: Ghat rate}. We also use the fact that
\begin{align*}
    \mathbb{E}_{P_{0,N}}[G_i^{(k)}g_i^T]\leq C.
\end{align*}
This inequality can be derived using matrix Cauchy-Schwartz inequality, the assumption that all eignvalues of $\widebar{\Omega}$ are bounded above and below by some constants, and the Assumption that $\Vert\mathbb{E}_{P_{0,N}}G_iG_i^T\Vert \leq C$. 

Therefore, we conclude that
\begin{align*}
    \left\Vert \frac{1}{\sqrt{mN}}\sum_{i=1}^N\tilde{U}^{(k)}_i \right\Vert=O_{p}(1).
\end{align*}

    Following the derivations of \citet{ye2024genius}, we have
    \begin{align*}
        &N\mu_N^{-1}(N^{-1}\sum_{i = 1}^N \tilde{U}_i^{(k)})^T\{\widehat{\Omega}(\beta_0,\widehat{\eta})^{-1}-\Omega^{-1}\}\widehat{g}(\beta_0,\widehat{\eta})\\
        \lesssim & \Big\Vert \frac{1}{\sqrt{mN}}\sum_{i=1}^N \tilde{U}^{(k)}_i\Big\Vert \Vert \sqrt{m}\{\widehat{\Omega}(\beta_0,\widehat{\eta})-\widebar{\Omega}(\beta_0,\eta_{0,N})\}\Vert\Vert \mu_N^{-1}\sqrt{N}\widehat{g}(\beta_0,\widehat{\eta})\Vert\\
        =& o_p(1).
    \end{align*}
    The proof of 3. is similar to the proof of 2., thus omitted.
\end{proof}
\begin{lemma}
    \begin{align*}
       \left\Vert \frac{\partial \widehat{g}(\beta_0,\eta_{0,N})}{\partial \beta^{(k)}}  \right\Vert = O_p\left(\sqrt{\frac{m}{N}}\right),
        \left\Vert \frac{\partial^2 \widehat{g}(\beta_0,\eta_{0,N})}{\partial \beta^{(k)}\partial\beta^{(r)}}  \right\Vert = O_p\left(\sqrt{\frac{m}{N}}\right).
    \end{align*}
\end{lemma}

\begin{proof}
    Assumption \ref{Assumption: further restriction for moments.} implies that
    \begin{align*}
        \left\Vert\mathbb{E}_{P_{0,N}}\Bigg[\frac{\partial \widehat{g}(\beta_0,\eta_{0,N})}{\partial \beta^{(k)}} \frac{\partial \widehat{g}(\beta_0,\eta_{0,N})^T}{\partial \beta^{(k)}} \Bigg] \right\Vert\leq C.
    \end{align*}
    Therefore,
    \begin{align*}
        \mathbb{E}_{P_{0,N}}\left\Vert \frac{\partial \widehat{g}(\beta_0,\eta_{0,N})}{\partial \beta^{(k)}}  \right\Vert^2=\mathbb{E}_{P_{0,N}}\Bigg[\frac{\partial \widehat{g}(\beta_0,\eta_{0,N})^T}{\partial \beta^{(k)}} \frac{\partial \widehat{g}(\beta_0,\eta_{0,N})}{\partial \beta^{(k)}} \Bigg] =\text{tr}\frac{1}{N}\mathbb{E}_{P_{0,N}}\Bigg[\frac{\partial {g}_i(\beta_0,\eta_{0,N})}{\partial \beta^{(k)}} \frac{\partial {g}_i(\beta_0,\eta_{0,N})^T}{\partial \beta^{(k)}}\Bigg]\leq \frac{m}{N}.
    \end{align*}
    Therefore,
    \begin{align*}
         \left\Vert \frac{\partial \widehat{g}(\beta_0,\eta_{0,N})}{\partial \beta^{(k)}}  \right\Vert = O_p\left(\sqrt{\frac{m}{N}}\right).
    \end{align*}
    Similarly, we can prove
    \begin{align*}
        \left\Vert \frac{\partial^2 \widehat{g}(\beta_0,\eta_{0,N})}{\partial \beta^{(k)}\partial\beta^{(r)}}  \right\Vert = O_p\left(\sqrt{\frac{m}{N}}\right).
    \end{align*}
\end{proof}

\begin{remark}
    A direct consequence for this lemma is that
    \begin{align*}
        \left\Vert \widehat{G}(\beta_0,\eta_{0,N})  \right\Vert = O_p\left(\sqrt{\frac{m}{N}}\right), \left\Vert \frac{\partial \widehat{G}(\beta_0,\eta_{0,N})}{\partial\beta^{(r)}}  \right\Vert = O_p\left(\sqrt{\frac{m}{N}}\right).
    \end{align*}
\end{remark}

    \begin{lemma}
    For any $\widebar{\beta}\xrightarrow[]{p}\beta_0$,
        \begin{align*}
            NS_N^{-1}\Bigg( \frac{\partial^2 \widehat{Q}(\widebar{\beta},\widehat{\eta})}{\partial \beta \partial \beta^T} -\frac{\partial^2 \widehat{Q}(\widebar{\beta},{\eta}_{0,N})}{\partial \beta \partial \beta^T}\Bigg)S_{N}^{-1T} = o_p(1).
        \end{align*}
        \label{Lemma: convergence of second order derivative of Q}
    \end{lemma}
\begin{proof}
    We firstly calculate the explicit form of $\frac{\partial^2\widehat{Q}(\beta,\eta)}{\partial\beta^{(k)}\partial\beta^{(r)}}$.
Direct calculation yields,
\begin{align*}
    \frac{\partial\widehat{\Omega}(\beta,\eta)}{\partial \beta^{(k)}} =& \frac{\partial}{\partial \beta}\left[\frac{1}{N}\sum_{i=1}^Ng_i(\beta,\eta)g_i(\beta,\eta)^T\right]=\frac{1}{N}\sum_{i=1}^N \left[\frac{\partial g_i(\beta,\eta)}{\partial \beta^{(k)}}g_i(\beta,\eta)^T+g_i(\beta,\eta)\frac{\partial g_i(\beta,\eta)^T}{\partial \beta^{(k)}}\right],\\
    \frac{\partial^2\widehat{\Omega}(\beta,\eta)}{\partial \beta^{(k)}\partial \beta^{(r)}}=&\frac{2}{N}\sum_{i=1}^N \left[ \frac{g_i(\beta,\eta)}{\partial\beta^{(k)}\partial\beta^{(r)}}g_i(\beta,\eta)^T+\frac{\partial g_i(\beta,\eta)}{\partial \beta^{(k)}}\frac{\partial g_i(\beta,\eta)^T}{\partial \beta^{(r)}}\right],\\
    \frac{\partial\widehat{Q}(\beta,\eta)}{\partial\beta^{(k)}}=&\frac{\partial \widehat{g}(\beta,\eta)^T}{\partial \beta^{(k)}}\widehat{\Omega}(\beta,\eta)^{-1}\widehat{g}(\beta,\eta)-\frac{1}{2}\widehat{g}(\beta,\eta)\widehat{\Omega}^{-1}\frac{\partial \widehat{\Omega}(\beta,\eta)}{\partial \beta^{(k)}}\widehat{\Omega}^{-1}\widehat{g}(\beta,\eta),\\
    \frac{\partial^2\widehat{Q}(\beta,\eta)}{\partial\beta^{(k)}\partial\beta^{(r)}}=&\frac{\partial^2\widehat{g}(\beta,\eta)^T}{\partial \beta^{(k)}\partial \beta^{(r)}}\widehat{\Omega}(\beta,\eta)^{-1}\widehat{g}(\beta,\eta)\\
    &-\frac{\partial \widehat{g}(\beta,\eta)^T}{\partial \beta^{(k)}}\widehat{\Omega}^{-1}(\beta,\eta)\frac{\partial \widehat{g}(\beta,\eta)^T}{\partial \beta^{(r)}}\widehat{\Omega}^{-1}(\beta,\eta)\widehat{g}(\beta,\eta)\\
    &-\frac{\partial \widehat{g}(\beta,\eta)^T}{\partial \beta^{(r)}}\widehat{\Omega}^{-1}(\beta,\eta)\frac{\partial \widehat{g}(\beta,\eta)^T}{\partial \beta^{(k)}}\widehat{\Omega}^{-1}(\beta,\eta)\widehat{g}(\beta,\eta)\\
    &+\frac{\partial \widehat{g}(\beta,\eta)^T}{\partial \beta^{(k)}}\widehat{\Omega}(\beta,\eta)^{-1}\frac{\partial \widehat{g}(\beta,\eta)^T}{\partial \beta^{(r)}}\\
    &+\widehat{g}(\beta,\eta)^T\widehat{\Omega}(\beta,\eta)^{-1}\frac{\partial \Omega(\beta,\eta)}{\partial\beta^{(r)}}\widehat{\Omega}(\beta,\eta)^{-1}\frac{\partial \widehat{\Omega}(\beta,\eta)}{\partial\beta^{(k)}}\widehat{\Omega
    }(\beta,\eta)^{-1}\widehat{g}(\beta,\eta)\\
    &-\frac{1}{2}\widehat{g}(\beta,\eta)\widehat{\Omega}(\beta,\eta)^{-1}\frac{\partial^2\widehat{\Omega}(\beta,\eta)}{\partial \beta^{(r)}\partial \beta^{(k)}}\widehat{\Omega}(\beta,\eta)^{-1}\widehat{g}(\beta,\eta)\\
    =:&J^{(k,r)}_1(\beta,\eta)-J^{(k,r)}_2(\beta,\eta)-J^{(k,r)}_3(\beta,\eta)+J^{(k,r)}_4(\beta,\eta)+J^{(k,r)}_5(\beta,\eta)-\frac{1}{2}J^{(k,r)}_6(\beta,\eta).
\end{align*}

Define $J_i(\beta,\eta)$ to be the matrix with $(k,j)$-th element equal to $J^{(k,j)}_i(\beta,\eta)$ for $i = 1,...,6$. It remains to show that for all $i\in\{1,2,3,4,5,6\}$, we have
\begin{align*}
    NS_N^{-1}(J_i(\widebar{\beta},\widehat{\eta})-J_i(\widebar{\beta},{\eta}_{0,N}))S_{N}^{-1T} = o_p(1).
\end{align*}

We will firstly show that
\begin{align*}
     NS_N^{-1}(J_1(\widebar{\beta},\widehat{\eta})-J_1(\widebar{\beta},{\eta}_{0,N}))S_{N}^{-1T} = o_p(1),
\end{align*}
the proofs for other terms are similar thus omitted.

Firstly notice that
\begin{align*}
    J_1(\beta,\eta) = \left(\frac{\partial\widehat{G}(\beta,\eta)}{\partial \beta^{(1)}},...,\frac{\partial\widehat{G}(\beta,\eta)}{\partial \beta^{(p)}}\right)^T\widehat{\Omega}(\beta,\eta)^{-1}\widehat{g}(\beta,\eta).
\end{align*}
\begin{align*}
    &\Vert NS_N^{-1}(J_1(\widebar{\beta},\widehat{\eta})-J_1(\widebar{\beta},{\eta}_{0,N}))S_{N}^{-1T}\Vert \\
    \leq&\left\Vert NS_N^{-1}\left[\left(\frac{\partial\widehat{G}(\widebar{\beta},\widehat{\eta})}{\partial \beta^{(1)}},...,\frac{\partial\widehat{G}(\widebar{\beta},\widehat{\eta})}{\partial \beta^{(p)}}\right)-\left(\frac{\partial\widehat{G}(\widebar{\beta},\eta_{0,N})}{\partial \beta^{(1)}},...,\frac{\partial\widehat{G}(\widebar{\beta},\eta_{0,N})}{\partial \beta^{(p)}}\right)\right]^T\widehat{\Omega}(\widebar{\beta},\widehat{\eta})^{-1}\widehat{g}(\widebar{\beta},\widehat{\eta})S_{N}^{-1T}\right\Vert\\
    &+\left\Vert NS_N^{-1}\left(\frac{\partial\widehat{G}(\widebar{\beta},\eta_{0,N})}{\partial \beta^{(1)}},...,\frac{\partial\widehat{G}(\widebar{\beta},\eta_{0,N})}{\partial \beta^{(p)}}\right)^T\widehat{\Omega}(\widebar{\beta},\widehat{\eta})^{-1}(\widehat{\Omega}(\widebar{\beta},\widehat{\eta})-\widehat{\Omega}(\widebar{\beta},{\eta}_{0,N}))\widehat{\Omega}(\widebar{\beta},{\eta}_{0,N})^{-1}\widehat{g}(\widebar{\beta},\widehat{\eta})S_{N}^{-1T}\right\Vert\\
    &+\left\Vert NS_N^{-1}\left(\frac{\partial\widehat{G}(\widebar{\beta},\eta_{0,N})}{\partial \beta^{(1)}},...,\frac{\partial\widehat{G}(\widebar{\beta},\eta_{0,N})}{\partial \beta^{(p)}}\right)^T\widehat{\Omega}(\widebar{\beta},{\eta}_{0,N})^{-1}(\widehat{g}(\widebar{\beta},\widehat{\eta})-\widehat{g}(\widebar{\beta},{\eta}_{0,N}))S_{N}^{-1T}\right\Vert\\
    :=&J_{11}+J_{12}+J_{13}
\end{align*}
For $J_{11}$, since $\frac{\partial^2 g(\beta,\eta)}{\partial\beta^{(k)}\partial\beta^{(r)}}$ satisfies global Neyman orthogonality in a neighborhood of $\beta_0$, similar to derivations in Lemma \ref{lemma: ghat rate}, we have
\begin{align*}
\left\Vert\left(\frac{\partial\widehat{G}(\widebar{\beta},\widehat{\eta})}{\partial \beta^{(1)}},...,\frac{\partial\widehat{G}(\widebar{\beta},\widehat{\eta})}{\partial \beta^{(p)}}\right)-\left(\frac{\partial\widehat{G}(\widebar{\beta},\eta_{0,N})}{\partial \beta^{(1)}},...,\frac{\partial\widehat{G}(\widebar{\beta},\eta_{0,N})}{\partial \beta^{(p)}}\right)\right\Vert=o_p\left(\frac{1}{\sqrt{N}}\right).
\end{align*}
Since for any $\beta$ we have
\begin{align*}
    \Vert\widehat{\Omega}({\beta},\widehat{\eta})-\widebar{\Omega}(\beta,\eta_{0,N})\Vert=o_p(1),
\end{align*}
and the eigenvalues of $\widebar{\Omega}(\beta,\eta_{0,N})$ are bounded by $1/C$ and $C$ for any $\beta$, therefore the eigenvalues of $\widehat{\Omega}(\widebar{\beta},\widehat{\eta})$ are bounded by $1/2C$ and $2C$. Previously we have derived that $\sup_{\Vert\delta(\beta)\Vert\leq C}\Vert \widehat{g}(\beta,\widehat{\eta}) \Vert = O_p(\mu_N/\sqrt{N})$. Therefore,
\begin{align*}
    J_{11}\leq N\mu_N^{-2}o_p(1/\sqrt{N})O_p(\mu_N/\sqrt{N}) = o_p(1/\mu_N)=o_p(1).
\end{align*}
For $J_{12}$, we firstly notice that
\begin{align*}
    &\left\Vert NS_N^{-1}\left(\frac{\partial\widehat{G}(\widebar{\beta},\eta_{0,N})}{\partial \beta^{(1)}},...,\frac{\partial\widehat{G}(\widebar{\beta},\eta_{0,N})}{\partial \beta^{(p)}}\right)^T\right\Vert\\
    =&\left\Vert NS_N^{-1}\left(\frac{\partial\widehat{G}(\widebar{\beta},\eta_{0,N})}{\partial \beta^{(1)}},...,\frac{\partial\widehat{G}(\widebar{\beta},\eta_{0,N})}{\partial \beta^{(p)}}\right)^T-\left(\frac{\partial\widehat{G}({\beta}_0,\eta_{0,N})}{\partial \beta^{(1)}},...,\frac{\partial\widehat{G}({\beta}_0,\eta_{0,N})}{\partial \beta^{(p)}}\right)^T\right\Vert\\
    &+\left\Vert NS_N^{-1}\left(\frac{\partial\widehat{G}({\beta}_0,\eta_{0,N})}{\partial \beta^{(1)}},...,\frac{\partial\widehat{G}({\beta}_0,\eta_{0,N})}{\partial \beta^{(p)}}\right)^T\right\Vert\\
    &=o_p(\sqrt{N})+N\Vert S^{-1}_{N}\Vert O_p(\sqrt{m/N}).
\end{align*}
Therefore,
\begin{align*}
    J_{12}\leq (o_p(\sqrt{N})+N\Vert S^{-1}_{N}\Vert O_p(\sqrt{m/N}))o_p(1/\sqrt{m})O_p(\mu_N/\sqrt{N})\Vert S^{-1}_{N}\Vert=o_p(1).
\end{align*}

Similarly we can prove that $J_{13} = o_p(1)$. Therefore we have proved that
\begin{align*}
    NS_N^{-1}(J_1(\widebar{\beta},\widehat{\eta})-J_1(\widebar{\beta},{\eta}_{0,N}))S_{N}^{-1T} = o_p(1).
\end{align*}

For $J_2$ term, similarly we notice that
\begin{align*}
    J_2(\beta,\eta)=\widehat{G}(\beta,\eta)^T\widehat{\Omega}^{-1}(\beta,\eta)\widehat{G}(\beta,\eta)\widehat{\Omega}^{-1}(\beta,\eta)\widehat{g}(\beta,\eta).
\end{align*}

Therefore,

\begin{align*}
    &\Vert NS_N^{-1}(J_2(\widebar{\beta},\widehat{\eta})-J_2(\widebar{\beta},{\eta}_{0,N}))S_{N}^{-1T}\Vert \\
    \leq&\left\Vert NS_N^{-1}[\widehat{G}(\widebar{\beta},\widehat{\eta})-\widehat{G}(\widebar{\beta},{\eta}_{0,N})]\widehat{\Omega}^{-1}(\widebar{\beta},\widehat{\eta})\widehat{G}(\widebar{\beta},\widehat{\eta})\widehat{\Omega}^{-1}(\widebar{\beta},\widehat{\eta})\widehat{g}(\widebar{\beta},\widehat{\eta})S_N^{-1}\right\Vert\\
    +&\left\Vert NS_N^{-1}\widehat{G}(\widebar{\beta},{\eta}_{0,N})[\widehat{\Omega}(\widebar{\beta},\widehat{\eta})^{-1}-\widehat{\Omega}(\widebar{\beta},{\eta}_{0,N})^{-1}]\widehat{G}(\widebar{\beta},\widehat{\eta})\widehat{\Omega}^{-1}(\widebar{\beta},\widehat{\eta})\widehat{g}(\widebar{\beta},\widehat{\eta})S_N^{-1}\right\Vert\\
    +&\left\Vert NS_N^{-1}\widehat{G}(\widebar{\beta},{\eta}_{0,N})\widehat{\Omega}(\widebar{\beta},{\eta}_{0,N})^{-1}[\widehat{G}(\widebar{\beta},\widehat{\eta})-\widehat{G}(\widebar{\beta},{\eta}_{0,N})]\widehat{\Omega}^{-1}(\widebar{\beta},\widehat{\eta})\widehat{g}(\widebar{\beta},\widehat{\eta})S_N^{-1}\right\Vert\\
    +&\left\Vert NS_N^{-1}\widehat{G}(\widebar{\beta},{\eta}_{0,N})\widehat{\Omega}(\widebar{\beta},{\eta}_{0,N})^{-1}\widehat{G}(\widebar{\beta},{\eta}_{0,N})[\widehat{\Omega}^{-1}(\widebar{\beta},\widehat{\eta})-\widehat{\Omega}^{-1}(\widebar{\beta},{\eta}_{0,N})]\widehat{g}(\widebar{\beta},\widehat{\eta})S_N^{-1}\right\Vert\\
    +&\left\Vert NS_N^{-1}\widehat{G}(\widebar{\beta},{\eta}_{0,N})\widehat{\Omega}(\widebar{\beta},{\eta}_{0,N})^{-1}\widehat{G}(\widebar{\beta},{\eta}_{0,N})\widehat{\Omega}^{-1}(\widebar{\beta},{\eta}_{0,N})[\widehat{g}(\widebar{\beta},\widehat{\eta})-\widehat{g}(\widebar{\beta},{\eta}_{0,N})]S_N^{-1}\right\Vert\\
    :=&J_{21}+J_{22}+J_{23}+J_{24}+J_{25}.
\end{align*}

Similar to the previous arguments, we have
\begin{align*}
    &J_{21}\leq N\mu_N^{-2}o_P(1/\sqrt{N})O_P(\mu_N/\sqrt{N})O_P(\mu_N/\sqrt{N})=o_p(1)
\end{align*}
Now we consider the second term. Similar to previous arguments, we have
\begin{align*}
    &\Vert NS_N^{-1}\widehat{G}(\widebar{\beta},\eta_{0,N})\Vert =o_p(\sqrt{N})+N\Vert S^{-1}_{N}\Vert O_p(\sqrt{m/N}),\\
    &\Vert \widehat{\Omega}(\widebar{\beta},\widehat{\eta})^{-1}-\widehat{\Omega}(\widebar{\beta},{\eta}_{0,N})^{-1}\Vert 
    =\Vert \widehat{\Omega}(\widebar{\beta},\widehat{\eta})^{-1}(\widehat{\Omega}(\widebar{\beta},\widehat{\eta})-\widehat{\Omega}(\widebar{\beta},{\eta}_{0,N}))\widehat{\Omega}(\widebar{\beta},{\eta}_{0,N})^{-1}\Vert=o_p(1/\sqrt{m}).
\end{align*}

For $\Vert \widehat{G}(\widebar{\beta},\widehat{\eta})\Vert$, we have
\begin{align*}
    &\Vert \widehat{G}(\widebar{\beta},\widehat{\eta})\Vert\leq \Vert \widehat{G}(\widebar{\beta},\widehat{\eta})-\widehat{G}(\widebar{\beta},{\eta}_{0,N})\Vert+\Vert \widehat{G}(\widebar{\beta},{\eta}_{0,N})\Vert\\
    \leq &o_p(1/\sqrt{N})+\Vert S_N\Vert o_p(1/\sqrt{N})+O_p(\sqrt{m/N})=\Vert S_N\Vert o_p(1/\sqrt{N}).
\end{align*}
Therefore
\begin{align*}
    J_{22}\leq (o_p(\sqrt{N})+N\Vert S^{-1}_{N}\Vert O_p(\sqrt{m/N}))o_p(1/\sqrt{m})\Vert S_N\Vert o_p(1/\sqrt{N})O_p(\mu_N/\sqrt{N})\Vert S_N^{-1}\Vert=o_p(1).
\end{align*}
Similarly we can show that $J_{23}$, $J_{24}$, $J_{25}$ are $o_p(1)$. 

Using the same technique we can show that
\begin{align*}
     NS_N^{-1}(J_i(\widebar{\beta},\widehat{\eta})-J_i(\widebar{\beta},{\eta}_{0,N}))S_{N}^{-1T} = o_p(1),
\end{align*}
for $i = 3,4,5,6$. Therefore, we have
\begin{align*}
            NS_N^{-1}\Bigg( \frac{\partial^2 \widehat{Q}(\widebar{\beta},\widehat{\eta})}{\partial \beta \partial \beta^T} -\frac{\partial^2 \widehat{Q}(\widebar{\beta},{\eta}_{0,N})}{\partial \beta \partial \beta^T}\Bigg)S_{N}^{-1T} = o_p(1).
\end{align*}
        
\end{proof}
 \section{Proof for main Theorems}
\subsection{Proof for consistency}
\begin{proof}
    Our goal is to prove that $\delta(\widehat{\beta})\xrightarrow[]{p} 0$, where $\delta(\beta) = \frac{1}{\mu_N}S^T_N(\beta-\beta_0)$. By identification assumption,
    \begin{equation*}
        \Big\Vert \frac{S_N^T(\beta-\beta_0)}{\mu_N} \Big\Vert \leq C\sqrt{N}\frac{\Vert P_{0,N}[g(\beta,\eta_{0,N})]}{\mu_N}.
    \end{equation*}
    Therefore, it remains to prove that 
    \begin{equation*}
\mu_N^{-1}\sqrt{N}\Vert P_{0,N}[g(\widehat{\beta},\eta_{0,N})] \Vert = o_{p}(1).
    \end{equation*}
    Since $\widehat{\beta}$ is the minimizer of $\widehat{Q}(\beta,\widehat{\eta})$, we have
    \begin{equation*}
        \mu_N^{-2}N\widehat{Q}(\widehat{\beta},\widehat{\eta})\leq \mu_N^{-2}N\widehat{Q}(\beta_0,\widehat{\eta})
    \end{equation*}
    Later on, we will prove that
    \begin{equation}
        \sup_{\Vert\delta(\beta)\Vert\leq C}\mu_N^{-2}N \vert \widehat{Q}(\beta,\widehat{\eta})-Q(\beta,\eta_{0,N}) \vert = o_p(1). \label{eq: key uniform convergence}
    \end{equation}
    for any $\epsilon,\gamma$, let $\mathcal{A}_1=\{\Vert \delta(\widehat{\beta}) \Vert\leq C\}$, since $\delta(\widehat{\beta})=O_p(1)$, we can choose $C$ such that $P(\mathcal{A}_1)\geq 1-\epsilon$. Let $\mathcal{A}_2 = \{\sup_{\Vert \delta(\beta) \Vert\leq C}\mu_N^{-2}N\vert \widehat{Q}(\beta,\widehat{\eta})-Q(\beta,\eta_{0,N}) \vert\leq \gamma\}$, then by \ref{eq: key uniform convergence}, we have for $N$ large enough, $P(\mathcal{A}_2)>1-\epsilon$. Therefore, with probability bigger than $1-2\epsilon$,
    \begin{align*}
        \mu_N^{-2}N Q(\widehat{\beta},\eta_{0,N})\leq \mu_N^{-2}N \widehat{Q}(\widehat{\beta},\widehat{\eta}) +\gamma\leq \mu_N^{-2}N \widehat{Q}({\beta}_0,\widehat{\eta}) +\gamma\leq \mu_N^{-2}N{Q}({\beta}_0,{\eta}_0) +2\gamma,
    \end{align*}
    which implies $\mu_N^{-2}N[Q(\widehat{\beta},\eta_{0,N})-Q(\beta_0,\eta_{0,N})]=o_p(1)$ since $\epsilon$ and $\gamma$ are arbitrarily small. 

   Then we have
    \begin{align*}
        &\mu_N^{-2}N[Q(\widehat{\beta},\eta_{0,N})-Q(\beta_0,\eta_{0,N})]+m/\mu_N^2
        =\mu_N^{-2}NQ(\widehat{\beta},\eta_{0,N})\\
        =&\mu_N^{-2}N P_{0,N}g(\widehat{\beta},\eta_{0,N})P_{0,N}g^T(\widehat{\beta},\eta_{0,N})g(\widehat{\beta},\eta_{0,N})^{-1}P_{0,N}g(\widehat{\beta},\eta_{0,N})\\
        \geq & C\mu_N^{-2}N \Vert P_{0,N}g(\widehat{\beta},\eta_{0,N})\Vert^2,
    \end{align*}
    where the last inequality is due to Assumption \ref{Assumption: convergence of g-hat}. Therefore
    \begin{equation*}
        \mu_N^{-1}\sqrt{N}\Vert P_{0,N}[g(\widehat{\beta},\eta_{0,N})] \Vert = o_p(1),
    \end{equation*}
    and hence $\delta(\widehat{\beta})\xrightarrow[]{p}0$.

    To show \eqref{eq: key uniform convergence}, we make the following decomposition:
    \begin{align*}
     & \sup_{\Vert\delta(\beta)\Vert\leq C}\mu_N^{-1}N\vert \widehat{Q}(\beta,\widehat{\eta})-{Q}(\beta,\eta_{0,N}) \vert\\
    \leq & \underbrace{\sup_{\Vert\delta(\beta)\Vert\leq C}\mu_N^{-1}N \vert \widehat{Q}(\beta,\widehat{\eta})- \Tilde{Q}(\beta,\widehat{\eta})\vert}_{=:A_N} 
    + \underbrace{\sup_{\Vert\delta(\beta)\Vert\leq C}\mu_N^{-1}N \vert \Tilde{Q}(\beta,\widehat{\eta})- \Tilde{Q}(\beta,{\eta}_{0,N})\vert}_{=:B_N} 
    \\
    &+  \underbrace{\sup_{\Vert\delta(\beta)\Vert\leq C}\mu_N^{-1}N \vert \Tilde{Q}(\beta,{\eta}_{0,N})- {Q}(\beta,{\eta}_{0,N})\vert}_{=:C_N},
\end{align*}
where
\begin{align*}
    & Q(\beta,\eta) = \widebar{g}(\beta,\eta)^T\widebar{\Omega}(\beta,\eta)^{-1}\widebar{g}(\beta,\eta)/2+\frac{m}{2N},\\
    & \Tilde{Q}(\beta,\eta)  = \widehat{g}(\beta,\eta)^T\widebar{\Omega}(\beta,\eta_{0,N})^{-1}\widehat{g}(\beta,\eta)/2.
\end{align*}

Here we only discuss $A_N = o_p(1)$ and $B_N = o_p(1)$. Since $C_N$ does not involve any nuisance parameter, the proof for uniform convergence of $C_N$ is the same as \citet{newey2009generalized}.

For $A_N$ term, following the same derivation as in \citet{ye2024genius}, we have
\begin{align*}
    A_N  
    \leq &\sup_{\Vert\delta(\beta)\Vert\leq C} \left\vert \widehat{a}(\beta,\widehat{\eta})^T(\tilde{\Omega}(\beta,\widehat{\eta})-\widebar{\Omega}(\beta,\widehat{\eta}))\widehat{a}(\beta,\widehat{\eta}) \right\vert \\
    &+ \sup_{\Vert\delta(\beta)\Vert\leq C}\left\vert \widehat{a}(\beta,\widehat{\eta})^T(\widehat{\Omega}(\beta,\widehat{\eta})-\widebar{\Omega}(\beta,\widehat{\eta}))\widehat{\Omega}(\beta,\widehat{\eta})^{-1} (\widehat{\Omega}(\beta,\widehat{\eta})-\widebar{\Omega}(\beta,\widehat{\eta}))\widehat{a}(\beta,\widehat{\eta})  \right\vert\\
    \leq &\sup_{\Vert\delta(\beta)\Vert\leq C}\Vert \widehat{a}(\beta,\widehat{\eta})\Vert^2 \Big\{\sup_{\Vert\delta(\beta)\Vert\leq C}\Big\Vert \widehat{\Omega}(\beta,\widehat{\eta})- \Bar{\Omega}(\beta,\widehat{\eta})\Big\Vert+C\sup_{\Vert\delta(\beta)\Vert\leq C}\Big\Vert \widehat{\Omega}(\beta,\widehat{\eta})- \Bar{\Omega}(\beta,\widehat{\eta})\Big\Vert^2\Big\}
\end{align*}
where 
\begin{align*}
   \widehat{a}(\beta,\eta) = \mu_N^{-1}\sqrt{N}\widebar{\Omega}(\beta,\eta_{0,N})^{-1}\widehat{g}(\beta,{\eta}).
\end{align*}

Since
\begin{align*}
    &\sup_{\Vert\delta(\beta)\Vert\leq C}\Vert\widehat{a}(\beta,\widehat{\eta})\Vert^2
     \leq  \sup_{\Vert\delta(\beta)\Vert\leq C}C \mu_N^{-2}N \Vert \widehat{g}(\beta,\widehat{\eta})\Vert^2 = O_p(1),\\
     &\Vert \widehat{\Omega}(\beta_0,\widehat{\eta})-\overline{\Omega}(\beta_0,\eta_{0,N})\Vert = o_p(1),
\end{align*}
$A_N = o_p(1)$.

For $B_N$, following the same derivation of \citet{ye2024genius}, we have
\begin{align*}
    B_n &\lesssim \mu_N^{-2}N \sup_{\beta \in B} \Vert \widehat{g}(\beta,\widehat{\eta})-\widehat{g}(\beta,\eta_{0,N})\Vert \sup_{\Vert\delta(\beta)\Vert\leq C} (\Vert \widehat{g}(\beta,\widehat{\eta})\Vert+\Vert \widehat{g}(\beta,{\eta}_{0,N})\Vert)\\
    &=O_{p}\left(\frac{\sqrt{m}\delta_N}{\mu_N}\right) = o_p(1)
\end{align*}
by Lemma \ref{lemma: ghat rate}, Lemma \ref{lemma: g hat rate 3} and Lemma \ref{lemma: g rate 4}. 
\end{proof}

\subsection{Proof for asymptotic normality}

By the first order condition
\begin{align*}
    \frac{\partial \widehat{Q}(\beta,\widehat{\eta})}{\partial \beta}\Big|_{\beta = \widehat{\beta}} = 0
\end{align*}
and Taylor expansion around $\beta_0$, we have
\begin{align*}
    0 = NS_N^{-1}\frac{\partial \widehat{Q}(\beta,\widehat{\eta})}{\partial \beta}\Bigg|_{\beta = \widehat{\beta}} = NS_N^{-1}\frac{\partial \widehat{Q}(\beta,\widehat{\eta})}{\partial\beta}\Bigg|_{\beta = \beta_0} + NS_N^{-1}\frac{\partial \widehat{Q}(\beta,\widehat{\eta})}{\partial\beta \partial \beta^T}\Bigg|_{\beta = \Bar{\beta}}(S_N^{-1})^{T}[{S_N}(\widehat{\beta}-\beta_0)]
\end{align*}
Next, we aim to show 
\begin{align*}
    \left\Vert NS_N^{-1}\frac{\partial \widehat{Q}(\beta,\widehat{\eta})}{\partial\beta}\Bigg|_{\beta = \beta_0}-NS_N^{-1}\frac{\partial \widehat{Q}(\beta,{\eta}_{0,N})}{\partial\beta}\Bigg|_{\beta = \beta_0}\right\Vert=o_p(1),
\end{align*}
it suffices to show that 
\begin{align*}
    N/\mu_N\frac{\partial \tilde{Q}(\beta,\widehat{\eta})}{\partial\beta}\Bigg|_{\beta = \beta_0}=N/\mu_N\frac{\partial \widehat{Q}(\beta,{\eta}_{0,N})}{\partial\beta}\Bigg|_{\beta = \beta_0}+o_p(1)
\end{align*}
due to $\Vert S_N^{-1}\Vert\leq \mu_N^{-1}$ and Lemma \ref{Lemma: equivalence of FO derivative}.

We have
\begin{align*}
    &N\mu_N^{-1}\frac{\partial \widehat{Q}(\beta,\widehat{\eta})}{\partial\beta^{(k)}}\Bigg|_{\beta = \beta_0} = N\mu_N^{-1}\frac{\partial \tilde{Q}(\beta,\widehat{\eta})}{\partial\beta^{(k)}}\Bigg|_{\beta = \beta_0}+o_p(1)\\
    =&N\mu_N^{-1} \widebar{G}^{(k)T} \widebar{\Omega}\widehat{g}(\beta_0,\widehat{\eta})+\frac{1}{N\mu_N}\sum_{i,j = 1}^N \Tilde{U}_i^{(k)T}\widebar{\Omega}^{-1}(g(O_j;\beta_0,\widehat{\eta})-g(O_j;\beta_0,{\eta}_{0,N})) \\
    &+ \frac{1}{N\mu_N}\sum_{i,j=1}^N \Tilde{U}_i^{(k)T} \Bar{\Omega}^{-1}g(O_j;\beta_0,{\eta}_0)
    +o_p(1)
\end{align*}
where
\begin{align*}
    \Tilde{U}^{(k)}_i = G^{(k)}(\beta,\widehat{\eta})-\Bar{G}^{(k)}-\mathbb{E}[G^{(k)}_i(\beta_0,\eta_{0,N})g^T_i(\beta_0,\eta_{0,N})]\widebar{\Omega}^{-1}g(\beta_0,\widehat{\eta}).
\end{align*}
Notice that
\begin{align*}
    &\Vert N\mu_N^{-1} \widebar{G}^{(k)T} \widebar{\Omega}\hat{g}(\beta_0,\widehat{\eta})-N\mu_N^{-1} \widebar{G}^{(k)T} \widebar{\Omega}\hat{g}(\beta_0,{\eta}_{0,N})\Vert \\
    =& \Vert N\mu_N^{-1}\widebar{G}^{(k)T}\Omega^{-1}(\widehat{g}(\beta_0,\widehat{\eta})-\widehat{g}(\beta_0,{\eta}_{0,N}))\Vert = o_p(1)
\end{align*}

For the second term,
\begin{align*}
    &\frac{1}{N\mu_N}\sum_{i,j = 1}^N \Tilde{U}_i^{(k)T}\Bar{\Omega}^{-1}(g_j(\beta_0,\widehat{\eta})-g_j(\beta_0,{\eta}_{0,N})) \\
    \lesssim &  \frac{1}{N\mu_N}\sum_{i,j = 1}^N \Tilde{U}_i^{(k)T}(g_j(\beta_0,\widehat{\eta})-g_j(\beta_0,{\eta}_0)) \\
    \leq & \frac{1}{N\mu_N}\sum_{i=1}^N \Tilde{U}_i^{(k)T}\sum_{j=1}^N(g_j(\beta_0,\widehat{\eta})-g_j(\beta_0,{\eta}_{0,N}))\\
    \leq & \frac{1}{N\mu_N}\Bigg\Vert \sum_{i=1}^N \Tilde{U}_i^{(k)T}\Bigg\Vert 
 \Bigg\Vert \sum_{j=1}^N(g_j(\beta_0,\widehat{\eta})-g_j(\beta_0,{\eta}_{0,N}))\Bigg\Vert \\
 =& o_p(1)
\end{align*}
Since $\frac{1}{\sqrt{N}\mu_N}\Vert \Tilde{U}_i \Vert = O_p(1)$.

Lastly, 
\begin{align*}
    & \Bigg\Vert \frac{1}{N\mu_N}\sum_{i,j=1}^N \Tilde{U}^{(k)T}_i\Bar{\Omega}^{-1}g(O_j;\beta_0,{\eta}_{0,N})- \frac{1}{N\mu_N}\sum_{i,j=1}^N {U}_i^{(k)T} \Bar{\Omega}^{-1}g(O_j;\beta_0,{\eta}_{0,N})\Bigg\Vert \\
    \leq & \Bigg\Vert \frac{1}{N\mu_n}\sum_{i,j=1}^N \Big(\Tilde{U}^{(k)T}_i-{U}^{(k)T}_i\Big)\Bar{\Omega}^{-1}g(O_j;\beta_0,{\eta}_{0,N})\Bigg\Vert \\
    \lesssim & \Bigg\Vert \frac{1}{N\mu_N}\sum_{i=1}^N \Big(\Tilde{U}^{(k)T}_i-{U}^{(k)T}_i\Big)
    \sum_{j = 1}^N g(O_j;\beta_0,{\eta}_{0,N})\Bigg\Vert = o_p(1)
\end{align*}
Therefore, we have shown that
\begin{align*}
    NS_N^{-1}\frac{\partial \widehat{Q}(\beta,\widehat{\eta})}{\partial\beta}\Bigg|_{\beta = \beta_0}=NS_N^{-1}\frac{\partial \widehat{Q}(\beta,{\eta}_{0,N})}{\partial\beta}\Bigg|_{\beta = \beta_0}+o_p(1).
\end{align*}
Combing with 
\begin{equation*}
     NS_N^{-1}\Bigg( \frac{\partial^2 \widehat{Q}(\widebar{\beta},\widehat{\eta})}{\partial \beta \partial \beta^T}\Bigg)S_{N}^{-1T}=NS_N^{-1}\Bigg(\frac{\partial^2 \widehat{Q}(\widebar{\beta},{\eta}_{0,N})}{\partial \beta \partial \beta^T}\Bigg)S_{N}^{-1T}+ o_p(1)
\end{equation*}
and Theorem 3 of \citet{newey2009generalized}, the asymptotic normality result has been proved.

\subsection{Proof of Theorem \ref{theorem: variance estimator}}
\begin{proof}
The variance estimator can be written as
    \begin{align*}
 \frac{S_N^{-1}\widehat{V}S_N^{-1,T}}{N} = (NS_N^{-1}\widehat{H}S_N^{-1,T})^{-1}(NS_N^{-1}\widehat{D}^T\widehat{\Omega}\widehat{D}S_N^{-1,T})(NS_N^{-1}\widehat{H}S_N^{-1,T})^{-1}.
    \end{align*}

    We introduce some new notations to assist the proof:
    \begin{align*}
        \widehat{D}^{(j)}(\beta,\eta) &= \frac{1}{N}\sum_{i=1}^N\frac{\partial g_i(\beta,{\eta})}{\partial\beta_j}-\frac{1}{N}\sum_{i=1}^N\frac{\partial g_i(\beta,{\eta})}{\partial \beta_j}g_i(\beta,{\eta})^T\widehat{\Omega}(\beta,{\eta})^{-1}\widehat{g}(\beta,{\eta}),\\
        \widehat{D}(\beta,\eta) &= (\widehat{D}^{(1)},...,\widehat{D}^{(p)}), \\
        \Tilde{H} &= \frac{\partial^2\widehat{Q}(\widehat{\beta},{\eta}_{0,N})}{\partial\beta \partial \beta^T},\\
        \tilde{V} &=  \tilde{H}^{-1}\widehat{D}^T(\widehat{\beta},\eta_{0,N})\widehat{\Omega}^{-1}(\widehat{\beta},\eta_{0,N})\tilde{D}(\widehat{\beta},\eta_{0,N})\tilde{H}^{-1}.
    \end{align*}

We aim to show that
    \begin{align*}
\frac{S_N^{-1}\widehat{V}S_N^{-1}}{N} - \frac{S_N^{-1}\tilde{V}S_N^{-1}}{N}   = o_p(1).
    \end{align*}

    According to \citet{newey2009generalized},
    \begin{align*}
        & NS_N^{-1}\widehat{H}S_N^{-1} \xrightarrow[]{p} H,\\
        &NS_N^{-1,T}\tilde{D}^T\widehat{\Omega}(\widehat{\beta},\eta_{0,N})\tilde{D}S_N^{-1,T} \xrightarrow[]{p} HVH.
    \end{align*}
    Therefore, we only need to show
    \begin{enumerate}
        \item $(NS_N^{-1}\widehat{H}S_N^{-1,T})^{-1} - (NS_N^{-1}\tilde{H}S_N^{-1,T})^{-1} = o_p(1).$
        \item $NS_N^{-1}\widehat{D}^T(\widehat{\beta},\widehat{\eta})\widehat{\Omega}(\widehat{\beta},\widehat{\eta})\widehat{D}(\widehat{\beta},\widehat{\eta})S_N^{-1,T} -NS_N^{-1}\widehat{D}^T(\widehat{\beta},\eta_{0,N})\widehat{\Omega}(\widehat{\beta},\eta_{0,N})\widehat{D}(\widehat{\beta},\eta_{0,N})S_N^{-1,T} = o_p(1).$
    \end{enumerate}
I will list some key ingredients for showing those two claim and omit related details.
    For the first claim, by Lemma \ref{Lemma: convergence of second order derivative of Q}, we have
    \begin{align*}
        &(NS_N^{-1}\widehat{H}S_N^{-1,T})^{-1} - (NS_N^{-1}\tilde{H}S_N^{-1,T})^{-1}\\
        =& N^{-1}S_N^T(\widehat{H}^{-1}-\tilde{H}^{-1})S_N \\
        =& N^{-1}S_N^T\widehat{H}^{-1}(\tilde{H}-\widehat{H})\tilde{H}^{-1}S_N\\
        =&N^{-1}S_N^T\widehat{H}^{-1}S_N(NS_N^{-1}(\tilde{H}-\widehat{H})S_N^{-1,T})N^{-1}S_N^T\tilde{H}^{-1}S_N \xrightarrow[]{p} 0.
    \end{align*}

    For the second claim, we will show that
    \begin{align}
        &\sqrt{N}S_N^{-1}[\widehat{D}^T(\widehat{\beta},\widehat{\eta})-\widehat{D}^T(\widehat{\beta},{\eta}_{0,N})] = o_p(1), \label{eq: First D claim}\\
     &\sqrt{N}S_N^{-1}\widehat{D}^T(\widehat{\beta},{\eta}_{0,N}) = O_p(1). \label{eq: Second D claim}
    \end{align}
    Combing with the previous result $\Vert \widehat{\Omega}(\widehat{\beta},\widehat{\eta})-\widehat{\Omega}(\widehat{\beta},{\eta}_{0,N})\Vert = o_p(1/\sqrt{m})$, the second claim can be shown. Equation \eqref{eq: First D claim} follows from the following rate results we have already shown before:
    \begin{align*}
        &\left\Vert \frac{1}{N}\sum_{i=1}^N\frac{\partial g_i(\widehat{\beta},\widehat{\eta})}{\partial \beta_j} - \frac{1}{N}\sum_{i=1}^N\frac{\partial g_i(\widehat{\beta},{\eta}_{0,N})}{\partial \beta_j}\right\Vert  = o_p(1/\sqrt{m}),\\
        &\Vert \widehat{G}(\widehat{\beta},\widehat{\eta})-\widehat{G}(\widehat{\beta},{\eta}_{0,N})\Vert = o_p(1/\sqrt{N}), \\
        & \Vert \widehat{\Omega}(\widehat{\beta},\widehat{\eta})-\widehat{\Omega}(\widehat{\beta},{\eta}_{0,N})\Vert = o_p(1/\sqrt{m}), \\
        & \Vert \widehat{g}(\widehat{\beta},\widehat{\eta})-\widehat{g}(\widehat{\beta},{\eta}_{0,N})\Vert = o_p(1/\sqrt{N}),\\
        & \Vert \widehat{g}(\widehat{\beta},\eta_{0,N})\Vert = O_p(\mu_N/\sqrt{N}).
    \end{align*}
As for \eqref{eq: Second D claim}, it suffices to show that
\begin{align}
    &\Vert \sqrt{N}S_N^{-1}\widehat{D}^T(\widehat{\beta},{\eta}_{0,N})-\sqrt{N}S_N^{-1}\widehat{D}^T({\beta}_0,{\eta}_{0,N})\Vert = o_p(1) \label{eq: Third D claim},\\
    & \Vert \sqrt{N}S_N^{-1}\widehat{D}^T({\beta}_0,{\eta}_{0,N})\Vert = O_p(1) \label{eq: Fourth D claim}.
\end{align}

Equation \eqref{eq: Third D claim} is a direct consequence of Assumption \ref{Assumption: further restriction for moments.}. Equation \eqref{eq: Fourth D claim} is a direct consequence of $ \sqrt{m}/\mu_N  = O(1)$, $\Vert \widehat{G}(\beta_0,\eta_{0,N}) \Vert  = O_p(\sqrt{m}/\sqrt{N})$, $\Vert \widehat{g}(\beta_0,\eta_{0,N}) \Vert  = O_p(\sqrt{m}/\sqrt{N})$.

\end{proof}

\subsection{Proof of Theorem \ref{theorem: over-identification test}}
\begin{proof}
    Let $\overline{\beta}$ be a value lying on the line segment between $\beta_0$ and $\widehat{\beta}$. By mean-value theorem, 
    \begin{align*}
        &2N[\widehat{Q}(\beta_0,\widehat{\eta})-\widehat{Q}(\widehat{\beta},\widehat{\eta})]\\
        =&N(\widehat{\beta}-\beta_0)^T\left[\frac{\partial^2\widehat{Q}(\widebar{\beta},\widehat{\eta})}{\partial \beta \partial \beta^T}\right](\widehat{\beta}-\beta_0)\\
        =&(\widehat{\beta}-\beta_0)^TS_NNS_N^{-1}\left[\frac{\partial^2\widehat{Q}(\widebar{\beta},\widehat{\eta})}{\partial \beta \partial \beta^T}\right]S_N^{-1,T}S_N^T(\widehat{\beta}-\beta_0)\\
        =&(\widehat{\beta}-\beta_0)^TS_NNS_N^{-1}\left[\frac{\partial^2\widehat{Q}(\widebar{\beta},{\eta}_{0,N})}{\partial \beta \partial \beta^T}\right]S_N^{-1,T}S_N^T(\widehat{\beta}-\beta_0)+o_p(1),\\
        =&2N[\widehat{Q}(\beta_0,{\eta}_{0,N})-\widehat{Q}(\widehat{\beta},{\eta}_{0,N})]+o_p(1)
    \end{align*}
    where the third equality is due to Lemma \ref{Lemma: convergence of second order derivative of Q} and the fact that $S_N^T(\widehat{\beta}-\beta_0)=O_p(1)$. By the proof of Theorem 5 in \citet{newey2009generalized}, we have $2N[\widehat{Q}(\beta_0,{\eta}_{0,N})-\widehat{Q}(\widehat{\beta},{\eta}_{0,N})]=O_p(1)$. Therefore,
    \begin{align}
        \frac{2N[\widehat{Q}(\beta_0,\widehat{\eta})-\widehat{Q}(\widehat{\beta},\widehat{\eta})]}{\sqrt{m-p}}\xrightarrow[]{p}0.\label{eq: test, Qhat beta and Qhat beta_hat}
    \end{align}
Next, we will show that $N(\widehat{Q}(\beta_0,\widehat{\eta})-\widehat{Q}(\beta_0,{\eta}_{0,N}))=o_p(\sqrt{m})$.
\begin{align*}
     & \vert\widehat{Q}({\beta}_0,\widehat{\eta})-\widehat{Q}({\beta}_0,{\eta}_{0,N})\vert\\
    \leq & \vert (\widehat{g}(\beta_0,\widehat{\eta})-\widehat{g}(\beta_0,\eta_{0,N})) \widehat{\Omega}^{-1}(\beta_0,\widehat{\eta})\widehat{g}(\beta_0,\widehat{\eta})\vert\\
    &+\vert \widehat{g}(\beta_0,\eta_{0,N})\widehat{\Omega}(\beta_0,\widehat{\eta})^{-1}(\widehat{\Omega}(\beta_0,\widehat{\eta})-\widehat{\Omega}(\beta_0,{\eta}_{0,N}))\widehat{\Omega}^{-1}(\beta_0,{\eta}_{0,N})\widehat{g}(\beta_0,\widehat{\eta})\vert\\
    &+\vert \widehat{g}(\beta_0,\eta_{0,N}) \widehat{\Omega}^{-1}(\beta_0,\widehat{\eta})(\widehat{g}(\beta_0,\widehat{\eta})-\widehat{g}(\beta_0,{\eta}_{0,N}))\vert.
\end{align*}
Since
\begin{align*}
    &\Vert \widehat{g}(\beta_0,\widehat{\eta})-\widehat{g}(\beta_0,\eta_{0,N})\Vert = O_p\Big(\sqrt{\frac{m}{N}}\delta_N\Big),\delta_N=o(1/\sqrt{m})\\
    &\Vert \widehat{\Omega}(\beta_0,\widehat{\eta})-\widehat{\Omega}(\beta_0,{\eta}_{0,N})\Vert \leq \Vert \widehat{\Omega}(\beta_0,\widehat{\eta})-\widebar{\Omega}(\beta_0,{\eta}_{0,N})\Vert+\Vert\widebar{\Omega}(\beta_0,{\eta}_{0,N})-\widehat{\Omega}(\beta_0,{\eta}_{0,N})\Vert =o_p(1/\sqrt{m}),\\
    &\Vert \widehat{g}(\beta_0,\eta_{0,N})\Vert = O_p\Big(\sqrt{\frac{m}{N}}\Big),\Vert \widehat{g}(\beta_0,\widehat{\eta})\Vert = O_p\Big(\sqrt{\frac{m}{N}}\Big) \\
    & 1/(2C)\leq \lambda_{\min}(\widehat{\Omega}^{-1}(\beta_0,\widehat{\eta}))\leq \lambda_{\max}(\widehat{\Omega}^{-1}(\beta_0,\widehat{\eta}))\leq 2C,
\end{align*}
we have
\begin{align*}
    &N\vert (\widehat{g}(\beta_0,\widehat{\eta})-\widehat{g}(\beta_0,\eta_{0,N})) \widehat{\Omega}^{-1}(\beta_0,\widehat{\eta})\widehat{g}(\beta_0,\widehat{\eta})\vert=O_p\left({m}\delta_N\right)=o_p(\sqrt{m}),\\
    &N\vert \widehat{g}(\beta_0,\eta_{0,N})\widehat{\Omega}(\beta_0,\widehat{\eta})^{-1}(\widehat{\Omega}(\beta_0,\widehat{\eta})-\widehat{\Omega}(\beta_0,{\eta}_{0,N}))\widehat{\Omega}^{-1}(\beta_0,{\eta}_{0,N})\widehat{g}(\beta_0,\widehat{\eta})\vert=o(\sqrt{m})\\
    &N\vert \widehat{g}(\beta_0,\eta_{0,N}) \widehat{\Omega}^{-1}(\beta_0,\widehat{\eta})(\widehat{g}(\beta_0,\widehat{\eta})-\widehat{g}(\beta_0,{\eta}_{0,N}))\vert=o_p(\sqrt{m}).
\end{align*}
Therefore,
\begin{equation*}
    N(\widehat{Q}(\beta_0,\widehat{\eta})-\widehat{Q}(\beta_0,{\eta}_{0,N}))=o_p(\sqrt{m}).
\end{equation*}
Combing this result with \eqref{eq: test, Qhat beta and Qhat beta_hat}, we have
\begin{align*}
    \frac{2N[\widehat{Q}(\beta_0,{\eta}_{0,N})-\widehat{Q}(\widehat{\beta},\widehat{\eta})]}{\sqrt{m-p}}\xrightarrow[]{p}0.
\end{align*}
Hence,
\begin{align*}
    &\frac{2N\widehat{Q}(\widehat{\beta},\widehat{\eta})-(m-p)}{\sqrt{m-p}}\\
    =&\frac{2N\widehat{Q}({\beta}_0,{\eta}_{0,N})-(m-p)}{\sqrt{m-p}}+o_p(1)\\
    =&\sqrt{\frac{m}{m-p}}\frac{2N\widehat{Q}(\beta_0,\eta_{0,N})-m}{\sqrt{m}}+\frac{p}{\sqrt{m-p}}+o_p(1)\\
    \xrightarrow[]{d}&N(0,1)
\end{align*}
by proof of Theorem 5 in \citet{newey2009generalized}. By standard results for chi-squared distribution that $(\chi^2_{1-\alpha}(m)-m)/\sqrt{2m}$ converges to $1-\alpha$ quantile of standard normal as $m\rightarrow \infty$, we have
\begin{equation*}
    P(2N\widehat{Q}(\widehat{\beta},\widehat{\eta})\geq \chi^2_{1-\alpha}(m-p))=P\left(\frac{2N\widehat{Q}(\widehat{\beta},\widehat{\eta})-(m-p)}{\sqrt{2(m-p)}}\geq \frac{\chi^2_{1-\alpha}(m-p)-(m-p)}{\sqrt{2(m-p)}}\right)\rightarrow \alpha.
\end{equation*}
The Theorem has been proved.
\end{proof}

 \section{Linear moment condition case}
As an important example of separable moment, linear moment is very common in practice. That is, the estimating function have the following form:
\begin{align*}
    g(o;\beta,\eta) = G(o;\eta)^T\beta + g^{a}(o;\eta).
\end{align*}
The linear moment equation case could greatly simplify the problem. Assumption \ref{Assumption: separable score} can be simplified as follows for linear moment function:
\begin{assumption}
    \textbf{(Linear score regularity and nuisance parameters estimation).} 
    \begin{enumerate}[label=(\alph*)]
    \item Given a random subset $I$ of $[N]$ of size $N/L$, we have the nuisance parameter estimator $\widehat{\eta}$ (estimated using samples outside $I$) belongs to a realization set $\mathcal{T}_N \in T_N$ with probability $1-\Delta_N$, where $\mathcal{T}_N$ contains $\eta_{0,N}$ and satisfies the following conditions.
    \item $g(o;\beta,\eta)$ satisfies Neyman-orthogonality.
    \item The following conditions on statistical rates hold for any $j,k \in \{1,...,m\},l,r\in\{1,...,p\}$, define:
    \begin{align*}
        &r^{a,(j)}_N: = \sup_{\eta\in \mathcal{T}_N}\Big(\mathbb{E}_{P_{0,N}}[\vert g^{a,(j)}(O;\eta)-g^{a,(j)}(O;\eta_{0,N})\vert^2]\Big)^{1/2} \\
        & r^{(j,l)}_N:= \sup_{\beta \in \mathcal{B},\eta\in \mathcal{T}_N} \Big(\mathbb{E}_{P_{0,N}}\Bigg|G^{(j,l)}(\beta,\eta)-G^{(j,l)}(\beta,\eta_{0,N})\Bigg|^2\Big)^{1/2}  \\
        &\lambda^{(j)}_N:=\sup_{\beta\in \mathcal{B},t\in(0,1),\eta\in \mathcal{T}_N}\Bigg\vert \frac{\partial^2}{\partial t^2} \mathbb{E}_{P_{0,N}}g^{(j)}(O;\beta,\eta_{{0,N}}+t(\eta-\eta_{{0,N}}))\Bigg\vert\\
        & r^{a,(j),(k)}_N:= \sup_{\eta \in \mathcal{T}_N} \Big(\mathbb{E}_{P_{0,N}}\Bigg|g^{a,(j)}(\eta)g^{a,(k)}(\eta)-g^{a,(j)}(\eta_{P_{0,N}})g^{a,(k)}(\eta_{P_{0,N}})\Bigg|^2\Big)^{1/2}  \\
        & r^{(j,k,l,r)}_N:= \sup_{\beta \in \mathcal{B}} \Big(\mathbb{E}_{P_{0,N}}\Bigg|G^{(j,l)}(\eta)G^{(k,r)}(\eta)-G^{(j,l)}(\eta)G^{(k,r)}(\eta_{0,N})\Bigg|^2\Big)^{1/2} \\
        & r^{(j,k,r)}_N:= \sup_{\beta \in \mathcal{B}} \Bigg(\mathbb{E}_{P_{0,N}}\Bigg|G^{(j,l)}(\eta)g^{(k)}(\eta)-G^{(j,l)}(\eta)g^{(k)}(\eta_{0,N})\Bigg|^2\Bigg)^{1/2}
        \end{align*}
        Then the statistical rates satisfy
        \begin{align*}
        & a_N\leq  \delta_N, \lambda^{(j)}_N \leq N^{-1/2}\delta_N,\delta_N = o(1/\sqrt{m}),
    \end{align*}
    for $a_N \in\{ r^{a,(j)}_N,r^{(j,l))}_N,r^{a,(j),(k)}_N,r_N^{(j,k,l,r)},r_N^{(j,k,r)}\}$.
    \end{enumerate}\label{Assumption: linear moment function}
\end{assumption}

\section{Additive SMM example} 
We will verify the assumptions for ASMM example one by one. In this example, $g$ is linear moment condition. $G$ and $g^a$ are defined as
\begin{align*}
    &G(O;\eta) = -(Z-\eta_{Z}(X))(A-\eta_A(X)),\\
    &g^a(O;\eta) = (Z-\eta_{Z}(X))(Y-\eta_Y(X)).
\end{align*}

We first verify the regularity conditions hold for the ASMM example.

\begin{lemma}
    For all $\beta \in \mathcal{B}$, $1/C \leq \xi_{min}(\widebar{\Omega}(\beta,\eta_{0,N}))\leq \xi_{max}(\widebar{\Omega}(\beta,\eta_{0,N}))\leq C$.
\end{lemma}
\begin{proof}
    \begin{align*}
        &\widebar{\Omega}(\beta,\eta_{0})=\mathbb{E}[(Z-\eta_{Z,0}(X))(Z-\eta_{Z,0}(X))^T(Y-\eta_{Y,0}(X)-\beta A+ \beta\eta_{A,0}(X))^2]\\
        =&\mathbb{E}\left[(Z-\eta_{Z,0}(X))(Z-\eta_{Z,0}(X))^T\mathbb{E}[(Y-\eta_{Y,0}(X)-\beta A+ \beta\eta_{A,0}(X))^2|Z,X]\right]\\
        \lesssim & \mathbb{E}\left[(Z-\eta_{Z,0}(X))(Z-\eta_{Z,0}(X))^T\right]\leq CI_m
    \end{align*}
    Similarly,
    \begin{align*}
        &\widebar{\Omega}(\beta,\eta_{0,N})=\mathbb{E}[(Z-\eta_{Z,0}(X))(Z-\eta_{Z,0}(X))^T(Y-\eta_{Y,0}(X)-\beta A+ \beta\eta_{A,0}(X))^2]\\
        =&\mathbb{E}\left[(Z-\eta_{Z,0}(X))(Z-\eta_{Z,0}(X))^T\mathbb{E}[(Y-\eta_{Y,0}(X)-\beta A+ \beta\eta_{A,0}(X))^2|Z,X]\right]\\
        \gtrsim & \mathbb{E}\left[(Z-\eta_{Z,0}(X))(Z-\eta_{Z,0}(X))^T\right]\geq 1/CI_m.
    \end{align*}
    Therefore, for all $\beta \in \mathcal{B}$, $1/C \leq \xi_{min}(\widebar{\Omega}(\beta,\eta_{0,N}))\leq \xi_{max}(\widebar{\Omega}(\beta,\eta_{0,N}))\leq C$.
\end{proof}

\begin{lemma}
    (a) There is $C>0$ with $\vert \beta-\beta_0\vert\leq C\sqrt{N}\Vert \widebar{g}(\beta,\eta_{0,N}) \Vert/\mu_N$ for all $\beta \in \mathcal{B}$. (b) There is $C>0$ and $\widehat{M}=O_p(1)$ such that $\vert \beta-\beta_0\vert\leq C\sqrt{N}\Vert \widebar{g}(\beta,\eta_{0,N})\Vert/\mu_N+\widehat{M}$ for all $\beta\in \mathcal{B}$.
\end{lemma}

\begin{proof}
\noindent \textbf{Proof of (a):}
    \begin{align*}
        \widebar{g}(\beta,\eta_{0,N}) = \widebar{g}(\beta,\eta_{0,N})-\widebar{g}(\beta_0,\eta_{0,N}) = (\beta-\beta_0)\widebar{G}(\eta_{0,N}).
    \end{align*}
    Therefore,
    \begin{align*}
        \sqrt{N}\Vert \widebar{g}(\beta,\eta_{0,N})\Vert/\mu_N= \sqrt{N}\vert\beta-\beta_0\vert \sqrt{\widebar{G}^T(\eta_{0,N})\widebar{G}(\eta_{0,N})}/\mu_N \gtrsim \vert \beta-\beta_{0,N} \vert,
    \end{align*}
    where the second inequality is due to boundedness of eigenvalues of $\widebar{\Omega}(\beta,\eta_0)$ and weak asymptotics assumption.

    \noindent \textbf{Proof of (b):} 
    \begin{align*}
        &\sqrt{N}\widehat{g}(\beta,\eta_{0,N})/\mu_N\\
        =& \sqrt{N}\widehat{g}(\beta_0,\eta_{0,N})/\mu_N+\mu_N^{-1}\sqrt{N}(\beta-\beta_0)\frac{1}{N}\sum_{i=1}^N(G_i-\widebar{G})+\sqrt{N}(\beta-\beta_0)\widebar{G}/\mu_N.
    \end{align*}
    We have already shown that $\Vert\widehat{g}(\beta_0,\eta_{0,N})\Vert = O(\sqrt{m/N})$. Similarly we can drive the rate for $\Vert\frac{1}{N}\sum_{i=1}(G_i-\widebar{G})\Vert$, that is $\widehat{G}-\widebar{G}$:
    \begin{align*}
        &\mathbb{E}_{P_{0,N}}\Vert \widehat{G}-\widebar{G}\Vert^2=\mathbb{E}_{P_{0,N}}(\widehat{G}-\widebar{G})^T(\widehat{G}-\widebar{G})=\frac{1}{N}\mathbb{E}_{P_{0,N}}{G}_i^T{G}_i-\frac{1}{N}\widebar{G}^T\widebar{G}=\frac{1}{N}tr(\mathbb{E}_{P_{0,N}}{G}_i{G}^T_i)-\frac{1}{N}\widebar{G}^T\widebar{G}.
    \end{align*}
    From the weak asymptotics assumption, we know
    \begin{align*}
        \frac{1}{N}\widebar{G}^T\widebar{G} = O(\mu_N^2/N^2).
    \end{align*}
    From the assumption  $\mathbb{E}_{P_{0,N}}[A^2|Z,X]<C$
    \begin{align*}
        &\frac{1}{N}tr(\mathbb{E}_{P_{0,N}}{G}_i{G}^T_i)\\
        =&\frac{1}{N}tr(\mathbb{E}_{P_{0,N}}{G}_i{G}^T_i)=\frac{1}{N}tr\mathbb{E}_{P_{0,N}}\left[(Z-\mu_{Z,0,N}{X})(Z-\mu_{Z,0,N}(X))^T(A-\mu_{A,0,N}(X))^2\right]\\
        =&\frac{1}{N}tr\mathbb{E}_{P_{0,N}}\left[(Z-\mu_{Z,0,N}{(X)})(Z-\mu_{Z,0,N}(X))^T\mathbb{E}_{P_{0,N}}[(A-\mu_{A,0,N}(X))^2|Z,X]\right]\\
        \leq &\frac{1}{N}tr\mathbb{E}_{P_{0,N}}\left[(Z-\mu_{Z,0,N}{(X)})(Z-\mu_{Z,0,N}(X))^T\right]\lesssim m/N.
    \end{align*}
    To sum up, we have
    \begin{equation*}
        \left\Vert\frac{1}{N}\sum_{i=1}(G_i-\widebar{G})\right\Vert = O\left(\sqrt{\frac{m}{N}}\right).
    \end{equation*}
    Therefore, we have
    \begin{equation*}
        \widehat{M} := \mu_N^{-1}\sqrt{N}\sup_{\beta \in \mathcal{B}}\left\Vert \widehat{g}(\beta_0,\eta_{0,N})+(\beta-\beta_0)\frac{1}{N}\sum_{i=1}^N(G_i-\widebar{G}) \right\Vert = O_p(1),
    \end{equation*}
    where the last inequality is due to boundedness of $m^2/\mu_N$.

    Therefore,
    \begin{align*}
        \vert \beta-\beta_0\vert\leq \Vert \mu_N^{-1}\sqrt{N}(\beta-\beta_0)\widebar{G}\Vert\leq \mu_N^{-1}\sqrt{N}\Vert \widehat{g}(\beta,\eta_{0,N})\Vert+\widehat{M}.
    \end{align*}
\end{proof}

\begin{lemma}
    $\sup_{\beta\in\mathcal{B}}\mathbb{E}_{P_{0,N}}[(g(\beta,\eta_{0,N})^Tg(\beta,\eta_{0,N}))^2]/N\rightarrow 0$. \label{lemma: ASMM 3a}
\end{lemma}
\begin{proof}
    \begin{align*}
        &\sup_{\beta\in\mathcal{B}}\mathbb{E}_{P_{0,N}}[(g(\beta,\eta_{0,N})^Tg(\beta,\eta_{0,N}))^2]/N\\
        =&\sup_{\beta\in\mathcal{B}}\mathbb{E}_{P_{0,N}}[((Z-\eta_{Z,0,N}(X))^T(Z-\eta_{Z,0,N}(X)))^2(Y-\eta_{Y,0,N}(X)-\beta A+ \beta\eta_{A,0,N}(X))^4]/N\\
        \leq & m^2/N\sup_{\beta\in\mathcal{B}}\mathbb{E}_{P_{0,N}}[(Y-\eta_{Y,0,N}(X)-\beta A+ \beta\eta_{A,0,N}(X))^4]\\
        \lesssim & m^2/N \rightarrow 0,
    \end{align*}
    where the first inequality is due to boundedness of $Z$ and $X$, the second inequality is due to boundedness of $\eta_{Y,0,N}(X)$, $\eta_{A,0,N}(X)$, compactness of $\mathcal{B}$ and moment contraints $\mathbb{E}_{P_{0,N}}Y^4<C$, $\mathbb{E}_{P_{0,N}}A^4<C$.
\end{proof}

\begin{lemma}
    $\vert a^T \{\widebar{\Omega}(\beta',\eta_{0,N})-\widebar{\Omega}(\beta,\eta_{{0,N}})\}b\vert \leq C\Vert a\Vert \Vert b \Vert \Vert \beta'-\beta\Vert$.
\end{lemma}
\begin{proof}
By matrix Cauchy-Schwarz inequality \citet{gautam1999matrixcauchy}, we have 
\begin{align*}
    \mathbb{E}_{P_{0,N}}[G_ig_i(\beta,\eta_0)^T]\Omega(\beta,\eta_{0,N})^{-1}\mathbb{E}_{P_{0,N}}[g_i(\beta,\eta_{0,N})G^T_i]\leq \mathbb{E}_{P_{0,N}}[G_iG_i^T].
\end{align*}
Therefore, $\Vert \mathbb{E}[G_ig_i(\beta,\eta_{0,N})^T]\Vert\leq C$. 

Direct calculation yields:
    \begin{align*}
        &\vert a^T \{\widebar{\Omega}(\beta',\eta_{0,N})-\widebar{\Omega}(\beta,\eta_{{0,N}})\}b\vert\\
        \leq & \vert (\beta'-\beta)^2a^T \mathbb{E}_{P_{0,N}}\{G_iG_i^T\}b\vert+\vert (\beta'-\beta)a^T \mathbb{E}_{P_{0,N}}\{G_ig_i^T(\beta,\eta_{0,N})\}b\vert+\vert (\beta'-\beta)a^T \mathbb{E}_{P_{0,N}}\{g_i(\beta,\eta_{0,N})G^T_i\}b\vert\\
        \leq & \vert\beta'-\beta\vert\Vert a\Vert \Vert b\Vert (\vert \beta'-\beta\vert \Vert \mathbb{E}_{P_{0,N}}[G_iG_i^T]\Vert + 2\Vert \mathbb{E}_{P_{0,N}}[G_ig_i(\beta,\eta_{0,N})^T]\Vert)\\
        \lesssim & \vert\beta'-\beta\vert\Vert a\Vert \Vert b\Vert.
    \end{align*}
\end{proof}
\begin{lemma}
    There is $C$ and $\widehat{M}=O_p(1)$ such that for all $\beta',\beta \in \mathcal{B}$.
    \begin{align*}
        &\sqrt{N}\Vert \widebar{g}(\beta',\eta_{0,N})-\widebar{g}(\beta,\eta_{0,N})\Vert/\mu_N \leq C\vert \beta'-\beta\vert,\\
        &\sqrt{N}\Vert \widehat{g}(\beta',\eta_{0,N})-\widehat{g}(\beta_0,\eta_{0,N})\Vert/\mu_N \leq \widehat{M}\vert \beta'-\beta\vert.
    \end{align*}
    \label{Lemma: ASMM 3e}
\end{lemma}

\begin{proof}
    \begin{align*}
        \mu_N^{-1}\sqrt{N}\Vert \widebar{g}(\beta',\eta_{0,N})-\widebar{g}(\beta,\eta_{0,N})\Vert = \mu_N^{-1}\sqrt{N}\vert \beta'-\beta\vert\Vert \widebar{G}\Vert\lesssim \vert \beta'-\beta\vert.
    \end{align*}
    Similarly,
    \begin{align*}
        \mu_N^{-1}\sqrt{N}\Vert \widehat{g}(\beta',\eta_{0,N})-\widehat{g}(\beta,\eta_{0,N})\Vert = \mu_N^{-1}\sqrt{N}\vert \beta'-\beta\vert\Vert \widehat{G}\Vert\lesssim  \widehat{M}\vert\beta'-\beta\vert,
    \end{align*}
    where $\widehat{M}:=\mu_N^{-1}\sqrt{N}\Vert \widehat{G}\Vert=O_p(1)$.
    
\end{proof}

\begin{lemma}
\begin{equation*}
    (\mathbb{E}_{P_{0,N}}\Vert g_i\Vert^4+\mathbb{E}_{P_{0,N}}\Vert G_i\Vert^4)m/N\rightarrow 0.
\end{equation*}
\label{Lemma: ASMM 7b}
\end{lemma}
\begin{proof}
    The proof is similar to Lemma \ref{lemma: ASMM 3a}. The only difference is that we use the assumption $m^3/N\rightarrow 0$.
\end{proof}

\begin{lemma}
    $g$ satisfies Global Neyman orthogonality.
\end{lemma}
\begin{proof}
   Let $\eta_1=(\eta_{Y,1},\eta_{A,1},\eta_{Z,1})$ be another nuisance parameter. For simplicity of notation, the subscript $N$ will be suppressed. Then 
   \begin{align*}
       &\mathbb{E}_{P_{0,N}}\Bigg[g(O;\beta,(1-t)\eta_{0,N}+t\eta_1)\Bigg]\\
       = & \mathbb{E}_{P_{0,N}}\Bigg[\Bigg(Z-[(1-t)\eta_{Z,0,N}+t\eta_{Z,1}]\Bigg)\Bigg(Y-[(1-t)\eta_{Y,0,N}+t\eta_{Y,1}]-\beta A+\beta[(1-t)\eta_{A,0,N}+t\eta_{A,1}]]\Bigg)\Bigg]
   \end{align*}

We can calculate the first order Gateaux derivative 
\begin{align*}
    &\frac{\partial}{\partial t}\mathbb{E}_{P_{0,N}}\Bigg[g(O;\beta,(1-t)\eta_{0,N}+t\eta_1)\Bigg]\\
    =&\mathbb{E}_{P_{0,N}}\Bigg[\Bigg(\eta_{Z,0,N}-\eta_{Z,1}\Bigg)\Bigg(Y-[(1-t)\eta_{Y,0,N}+t\eta_{Y,1}]-\beta A+\beta[(1-t)\eta_{A,0,N}+t\eta_{A,1}]]\Bigg)\Bigg]+
    \\&\mathbb{E}_{P_{0,N}}\Bigg[\Bigg(Z-[(1-t)\eta_{Z,0,N}+t\eta_{Z,1}]\Bigg)\Bigg(\eta_{Y,0,N}-\eta_{Y,1}-\beta(\eta_{A,0,N}-\eta_{A,1})\Bigg)\Bigg]
\end{align*}

Therefore,
\begin{align*}
    &\frac{\partial}{\partial t}\mathbb{E}_{P_{0,N}}\Bigg[g(O;\beta,(1-t)\eta_{0,N}+t\eta_1)\Bigg]\Bigg|_{t=0}\\
    =&\mathbb{E}_{P_{0,N}}\Bigg[\Bigg(\eta_{Z,0,N}-\eta_{Z,1}\Bigg)\Bigg(Y-\eta_{Y,0,N}-\beta(A-\eta_{A,0,N})\Bigg)\Bigg]+\\
    &\mathbb{E}_{P_{0,N}}\Bigg[\Bigg(Z-\eta_{Z,0,N}\Bigg)\Bigg(\eta_{Y,0,N}-\eta_{Y,1}-\beta(\eta_{A,0,N}-\eta_{A,1})\Bigg)\Bigg]\\
    =&0
\end{align*}
Therefore, $g$ satisfies Neyman-orthogonality.
\end{proof}

\begin{lemma}
    $g$ satisfies the rate conditions in Assumption \ref{Assumption: linear moment function}.
\end{lemma}
\begin{proof}
Now we verify the rate conditions. Define the realization set to be the set of all $\eta = (\eta_{Z},\eta_{A},\eta_{Y})$ (where $\eta_{Z} = (\eta_{Z^{(1)}},...,\eta_{Z^{(m)}})$) such that
    \begin{align*}
        &\Vert \eta_{Z^{(j)}}-\eta_{Z^{(j)},0,N} \Vert_{P_{0,N},\infty} \leq C, \quad j = 1,....,m \\
        & \Vert \eta_{Y}-\eta_{Y,0} \Vert_{P_{0,N},\infty} \leq C, \Vert \eta_{A}-\eta_{A,0,N} \Vert_{P_{0,N},\infty} \leq C\\
        & \max( \Vert \eta_{A}-\eta_{A,0,N} \Vert_{P_{0,N},2},\Vert \eta_{Y}-\eta_{Y,0,N} \Vert_{P_{0,N},2})\leq \delta_N\\
        & \Vert \eta_{Z^{(j)}}-\eta_{Z^{(j)},0,N} \Vert_{P_{0,N},2}\leq \delta_N/m\\
        &\Vert \eta_{Z^{(j)}}-\eta_{Z^{(j)},0,N} \Vert_{P_{0,N},2}\times (\Vert \eta_{A}-\eta_{A,0,N} \Vert_{P_{0,N},2}+\Vert \eta_{Y}-\eta_{Y,0,N} \Vert_{P_{0,N},2})\leq \delta_N N^{-1/2}
\end{align*}

For $r^{(j,1)}_N$, we have
\begin{align*}
    r^{(j,1)}_N&:=\sup_{\eta\in \mathcal{T}_N}\Big(\mathbb{E}_{P_{0,N}}\big[\vert G^{(j)}(O;\eta)-G^{(j)}(O;\eta_{0,N})\vert^2\big]\Big)^{1/2}\\
    &\leq\sup_{\eta\in \mathcal{T}_N} \Bigg\Vert (\eta_{Z^{(j)}}(X)-\eta_{Z^{(j)},0,N}(X))(A-\eta_{A,0,N}(X))\Bigg\Vert_{P_{0,N},2}+\sup_{\eta\in \mathcal{T}_N} \Bigg\Vert (Z-\eta_{Z^{(j)}})(\eta_{A}(X)-\eta_{A,0,N}(X))\Bigg\Vert_{P_{0,N},2}\\
    &\lesssim \sup_{\eta\in \mathcal{T}_N} \Bigg\Vert \eta_{Z^{(j)}}-\eta_{Z^{(j)},0,N}\Bigg\Vert_{P_{0,N},2}+\sup_{\eta\in \mathcal{T}_N}\Bigg\Vert \eta_{A}-\eta_{A,0,N}\Bigg\Vert_{P_{0,N},2}\\
    &\leq \delta_N.
\end{align*}
The first inequality is due to inequality, the second inequality is due to boundedness of $(Z,X)$ and $\mathbb{E}_{P_{0,N}}[(A-\eta_{A,0,N}(X))^2|Z,X]$.

For $r^{a,(j)}_N$, we have
\begin{align*}
    r'^{a,(j)}_N&:=\sup_{\eta\in \mathcal{T}_N}\Big(\mathbb{E}_{P}\big[\vert g^{a,(j)}(O;\eta)-g^{a,(j)}(O;\eta_{0,N})\vert^2\big]\Big)^{1/2}\\
    &\leq \sup_{\eta\in \mathcal{T}_N} \Bigg\Vert (\eta_{Z^{(j)}}-\eta_{Z^{(j)},0,N})(Y-\eta_{Y,0,N}(X))\Bigg\Vert_{P,2}+\sup_{\eta\in \mathcal{T}_N} \Bigg\Vert (Z-\eta_{Z^{(j)}})(\eta_{Y}-\eta_{Y,0,N})\Bigg\Vert_{P_{0,N},2}\\
    &\lesssim \sup_{\eta\in \mathcal{T}_N} \Bigg\Vert \eta_{Z^{(j)}}-\eta_{Z^{(j)},0,N}\Bigg\Vert_{P_{0,N},2}+\sup_{\eta\in \mathcal{T}_N}\Bigg\Vert \eta_{Y}-\eta_{Y,0,N}\Bigg\Vert_{P_{0,N},2}\\
    &\leq \delta_N,
\end{align*}
The first inequality is due to inequality, the second inequality is due to boundedness of $(Z,X)$ and $\mathbb{E}_{P_{0,N}}[(Y-\eta_{Y,0,N}(X))^2|Z,X]$.

The second-order Gateaux derivative is

\begin{align*}
    &\frac{\partial^2}{\partial t^2}\mathbb{E}_{P_{0,N}}\Bigg[g(O;\beta,(1-t)\eta_{0,N}+t\eta_1)\Bigg]\Bigg|_{t=\tilde{t}}\\
    =&2\mathbb{E}_{P_{0,N}}\Bigg[\Bigg(\eta_{Z,0,N}(X)-\eta_{Z,1}(X)\Bigg)\Bigg(\eta_{Y,0,N}(X)-\eta_{Y,1}(X)-\beta(\eta_{A,0,N}(X)-\eta_{A,1}(X))\Bigg)\Bigg]
\end{align*}

Therefore, by triangle and Cauchy–Schwarz inequality,
\begin{align*}
    &\mathbb{E}_{P_{0,N}}\Bigg|\frac{\partial^2}{\partial t^2}\mathbb{E}_{P_{0,N}}\Bigg[g(O;\beta,(1-t)\eta_{0,N}+t\eta_1)\Bigg]\Bigg|_{t=\tilde{t}}\Bigg|\\
    \lesssim & \Vert \eta_{Z,0,N}-\eta_{Z,1}\Vert_{P_{0,N},2}\Big(\Vert \eta_{Y,0,N}-\eta_{Y,1}\Vert_{P_{0,N},2}+\Vert \eta_{A,0,N}-\eta_{A,1}\Vert_{P_{0,N},2}\Big)\\
    \leq & \delta_N N^{-1/2}.
\end{align*}

For $r^{a,(j),(k)}$ term, we have
\begin{align*}
    r'^{(j,k)}_N:=& \sup_{\eta \in \mathcal{T}_N} \Bigg(\mathbb{E}_{P_{0,N}}\Bigg|g^{a,(j)}(\eta)g^{a,(k)}(\eta)-g^{a,(j)}(\eta_{0,N})g^{a,(k)}(\eta_{0,N})\Bigg|^2\Bigg)^{1/2}\\
    \leq &\sup_{\eta \in \mathcal{T}_N}\Bigg\Vert (g^{a,(j)}(\eta)-g^{a,(j)}(\eta_{0,N}))g^{a,(k)}(\eta) \Bigg\Vert_{P_{0,N},2}+\sup_{\eta \in \mathcal{T}_N}\Bigg\Vert (g^{a,(k)}(\eta)-g^{a,(k)}(\eta_{0,N}))g^{a,(j)}(\eta_{0,N}) \Bigg\Vert_{P_{0,N},2}\\
     \leq& \sup_{\eta \in \mathcal{T}_N}\Bigg\Vert (g^{a,(j)}(\eta)-g^{a,(j)}(\eta_{0,N}))(g^{a,(k)}(\eta)-g^{a,(k)}(\eta_{0,N}) \Bigg\Vert_{P_{0,N},2} \\
    &+\sup_{\eta \in \mathcal{T}_N}\Bigg\Vert (g^{a,(k)}(\eta)-g^{a,(k)}(\eta_{0,N}))g^{a,(j)}(\eta_{0,N}) \Bigg\Vert_{P_{0,N},2}+\sup_{\eta \in \mathcal{T}_N}\Bigg\Vert (g^{a,(j)}(\eta)-g^{a,(j)}(\eta_{0,N}))g^{a,(k)}(\eta_{0,N}) \Bigg\Vert_{P_{0,N},2}.
\end{align*}

Note that

\begin{align*}
&g^{a,(j)}(O;\eta)-g^{a,(j)}(O;\eta_{0,N})\\
=&(Z^{(j)}-\eta_{Z^{(j)}}(X))(\eta_{Y,0,N}(X)-\eta_{Y}(X))+(\eta_{Z^{(j)},0,N}(X)-\eta_{Z^{(j)}}(X))(Y-\eta_{Y,0,N}(X))
\end{align*}

For the first term,
\begin{align*}
    &\sup_{\eta \in \mathcal{T}_N}\Bigg\Vert (g^{a,(j)}(\eta)-g^{a,(j)}(\eta_{0,N}))(g^{a,(k)}(\eta)-g^{a,(k)}(\eta_{0,N}) \Bigg\Vert_{P_{0,N},2}\\
    \leq & \sup_{\eta \in \mathcal{T}_N}\Bigg\Vert (Z^{(j)}-\eta_{Z^{(j)}}(X))(Z^{(k)}-\eta_{Z^{(k)}}(X))(\eta_{Y,0,N}(X)-\eta_{Y}(X))^2\Bigg\Vert_{P_{0,N},2}\\
    &+\sup_{\eta \in \mathcal{T}_N}\Bigg\Vert (Z^{(j)}-\eta_{Z^{(j)}}(X))(Z^{(k)}-\eta_{Z^{(k)}}(X))(Y-\eta_{Y,0,N}(X))^2\Bigg\Vert_{P_{0,N},2}\\
    &+\sup_{\eta \in \mathcal{T}_N}\Bigg\Vert (Z^{(j)}-\eta_{Z^{(j)}}(X))(Z^{(k)}-\eta_{Z^{(k)}}(X))(Y-\eta_{Y,0,N}(X))(\eta_{Y,0,N}(X)-\eta_{Y}(X))\Bigg\Vert_{P_{0,N},2}\\
    \lesssim & \sup_{\eta \in \mathcal{T}_N}\Bigg\Vert \eta_{Y,0,N}-\eta_{Y}\Bigg\Vert_{P_{0,N},2} +\sup_{\eta \in \mathcal{T}_N}\Bigg\Vert \eta_{Z^{(j)},0,N}-\eta_{Z^{(j)}}\Bigg\Vert_{P_{0,N},2}\leq \delta_N,
\end{align*}
where for the last inequality, I used $\Vert \eta_{Y}-\eta_{Y,0,N} \Vert_\infty \leq C$, $\mathbb{E}_{P_{0,N}}[(Y-\eta_{Y,0,N})^2|X]\leq C$, $Z$ and $X$ are bounded.

For the second term,
\begin{align*}
    &\sup_{\eta \in \mathcal{T}_N}\Bigg\Vert (g^{a,(k)}(\eta)-g^{a,(k)}(\eta_{0,N}))g^{a,(j)}(\eta_{0,N}) \Bigg\Vert_{P_{0,N},2}\\
    \leq &\sup_{\eta \in \mathcal{T}_N}\Bigg\Vert (Z^{(k)}-\eta_{Z^{(k)}}(X))(Z^{(j)}-\eta_{Z^{(j)},0,N}(X))(\eta_{Y,0,N}(X)-\eta_{Y}(X))(Y-\eta_{Y,0,N}(X)) \Bigg\Vert_{P_{0,N},2}+\\
    &+\sup_{\eta \in \mathcal{T}_N}\Bigg\Vert (\eta_{Z^{(k)}}(X)-\eta_{Z^{(k)},0,N}(X))(Z^{(j)}-\eta_{Z^{(j)},0,N}(X))(Y-\eta_{Y,0,N}(X))^2 \Bigg\Vert_{P_{0,N},2}\\
    \leq & \delta_N
\end{align*}

Similarly we can verify the remaining conditions in Assumption \ref{Assumption: linear moment function}.
\end{proof}

\begin{lemma}
    Assumption \ref{Assumption: ASN matrix estimation} is satisfied.
\end{lemma}

\begin{proof}
For the first part, we have
\footnotesize
\begin{align*}
    &\Bigg\Vert P_{0,N}[\Omega(\beta,\eta)-\Omega(\beta,\eta_{0,N})] \Bigg\Vert\\
    =&\Bigg\Vert \mathbb{E}_{P_{0,N}}(Z-\eta_Z(X))(Z-\eta_{Z}(X))^T(Y-\eta_Y(X)-\beta A+\beta \eta_A)^2-\\
    &\mathbb{E}_{P_{0,N}}(Z-\eta_{Z,0,N}(X))(Z-\eta_{Z,0,N}(X))^T(Y-\eta_{Y,0,N}(X)-\beta A+\beta \eta_{A,0,N})^2\Bigg\Vert \\
    \leq &\Bigg\Vert \mathbb{E}_{P_{0,N}}\Big[(Z-\eta_Z(X))(Z-\eta_{Z}(X))^T-(Z-\eta_{Z,0,N}(X))(Z-\eta_{Z,0,N}(X))^T(Y-\eta_Y(X)-\beta A+\beta \eta_A)^2\Big]\Bigg\Vert+\\
    &\Bigg\Vert \mathbb{E}_{P_{0,N}}[(Z-\eta_{Z,0,N}(X))(Z-\eta_{Z,0,N}(X))^T][(Y-\eta_Y(X)-\beta A+\beta \eta_A)^2-(Y-\eta_{Y,0,N}(X)-\beta A+\beta \eta_{A,0,N})^2]\Bigg\Vert
\end{align*}
\normalsize
Now we firstly consider the first term.
\scriptsize
\begin{align*}
   & \Bigg\Vert \mathbb{E}_{P_{0,N}}[(Z-\eta_Z(X))(Z-\eta_{Z}(X))^T-(Z-\eta_{Z,0,N}(X))(Z-\eta_{Z,0,N}(X))^T](Y-\eta_Y(X)-\beta A+\beta \eta_A)^2\Bigg\Vert\\
    \leq &\Bigg\Vert \mathbb{E}_{P_{0,N}}[(Z-\eta_Z(X))(Z-\eta_{Z}(X))^T-(Z-\eta_{Z,0,N}(X))(Z-\eta_{Z,0,N}(X))^T](Y-\eta_Y(X)-\beta A+\beta \eta_A)^2\Bigg\Vert_F\\
    =& \sqrt{\sum_{i,j=1}^m\Bigg(\mathbb{E}_{P_{0,N}}\Big[(Z^{(i)}-\eta_{Z^{(i)}}(X))(Z^{(j)}-\eta_{Z^{(j)}}(X))-(Z^{(i)}-\eta_{Z^{(i)},0,N}(X))(Z^{(j)}-\eta_{Z^{(j)},0,N}(X))(Y-\eta_Y(X)-\beta A+\beta \eta_A)^2\Big]\Bigg)^2}\\
    \leq &\sqrt{\sum_{i,j=1}^m\Bigg(\mathbb{E}_{P_{0,N}}\Big[\Big|(Z^{(i)}-\eta_{Z^{(i)}}(X))(Z^{(j)}-\eta_{Z^{(j)}}(X))-(Z^{(i)}-\eta_{Z^{(i)},0,N}(X))(Z^{(j)}-\eta_{Z^{(j)},0,N}(X))\Big|(Y-\eta_Y(X)-\beta A+\beta \eta_A)^2\Big]\Bigg)^2}
\end{align*}
\normalsize

Note that
\scriptsize
\begin{align*}
    &\mathbb{E}_{{P_{0,N}}}\Big[\Big|(Z^{(i)}-\eta_{Z^{(i)}}(X))(Z^{(j)}-\eta_{Z^{(j)}}(X))-(Z^{(i)}-\eta_{Z^{(i)},0,N}(X))(Z^{(j)}-\eta_{Z^{(j)},0,N}(X))\Big|(Y-\eta_Y(X)-\beta A+\beta \eta_A)^2\Big]\\
    \leq &\mathbb{E}_{P_{0,N}}\Big[\Big(\Big|(Z^{(i)}-\eta_{Z^{(i)}}(X))(\eta_{Z^{(j)}}(X)-\eta_{Z^{(j)},0,N}(X))\Big|+\Big|(Z^{(j)}-\eta_{Z^{(j)},0,N}(X))(\eta_{Z^{(i)}}(X)-\eta_{Z^{(i)},0,N}(X))\Big|\Big)(Y-\eta_Y(X)-\beta A+\beta \eta_A)^2\Big]\\
    \lesssim &\Vert \eta_{Z^{(j)}}-\eta_{Z^{(j)},0,N}\Vert_{P_{0,N},2} \Big[\sqrt{\mathbb{E}_{P_{0,N}}([Z^{(i)}-\eta_{Z^{(i)}})^2(Y-\eta_Y)^4]}\\
    &+\sqrt{\mathbb{E}_{P_{0,N}}[(Z^{(i)}-\eta_{Z^{(i)}})^2(A-\eta_A)^2(Y-\eta_Y)^2]}+\sqrt{\mathbb{E}_{P_{0,N}}[(Z^{(i)}-\eta_{Z^{(i)}})^2(A-\eta_A)^4]}\Big]\\
    &+\Vert \eta_{Z^{(i)}}-\eta_{Z^{(i)},0,N}\Vert_{P_{0,N},2} \Big[\sqrt{\mathbb{E}_{P_{0,N}}([Z^{(i)}-\eta_{Z^{(j)},0,N})^2(Y-\eta_Y)^4]}\\
    &+\sqrt{\mathbb{E}_{P_{0,N}}[(Z^{(j)}-\eta_{Z^{(j)}})^2(A-\eta_A)^2(Y-\eta_Y)^2]}+\sqrt{\mathbb{E}_{P_{0,N}}[(Z^{(j)}-\eta_{Z^{(j)}})^2(A-\eta_A)^4]}\Big]\\
    \lesssim & \frac{1}{{m}}\delta_N
\end{align*}\normalsize

We used uniformly boundedness of $(Z,X)$ in the last inequality, for example, to bound $\mathbb{E}_P([Z^{(i)}-\eta_{Z^{(i)}})^2(Y-\eta_Y)^4]$, we have
\begin{align*}
    & \mathbb{E}_{P_{0,N}}([Z^{(i)}-\eta_{Z^{(i)}})^2(Y-\eta_Y)^4]\\
    \lesssim & \mathbb{E}_{P_{0,N}}((Y-\eta_Y)^4]
    \leq  \mathbb{E}_{P_{0,N}}((Y-\eta_{Y,0,N}+\eta_{Y,0,N}-\eta_Y)^4]\\
    \lesssim & \mathbb{E}_{P_{0,N}}((Y-\eta_{Y,0,N})^4+(\eta_{Y,0,N}-\eta_Y)^4]\leq C.
\end{align*}
Similarly we can show that $\mathbb{E}_{P_{0,N}}[(Z^{(j)}-\eta_{Z^{(j)}})^2(A-\eta_A)^4]\leq C$ and $\mathbb{E}_{P_{0,N}}[(Z^{(i)}-\eta_{Z^{(i)}})^2(A-\eta_A)^2(Y-\eta_Y)^2]\leq C$.

Therefore
\begin{align*}
    \Bigg\Vert \mathbb{E}_{P_{0,N}}[(Z-\eta_Z(X))(Z-\eta_{Z}(X))^T-(Z-\eta_{Z,0,N}(X))(Z-\eta_{Z,0,N}(X))^T](Y-\eta_Y(X)-\beta A+\beta \eta_A)^2\Bigg\Vert\leq \delta_N
\end{align*}

For the second term,
\scriptsize
\begin{align*}
    &\Bigg\Vert \mathbb{E}_{P_{0,N}}[(Z-\eta_{Z,0,N}(X))(Z-\eta_{Z,0,N}(X))^T][(Y-\eta_Y(X)-\beta A+\beta \eta_A)^2-(Y-\eta_{Y,0,N}(X)-\beta A+\beta \eta_{A,0,N})^2]\Bigg\Vert\\
    \leq &\Bigg\Vert \mathbb{E}_{P_{0,N}}[(Z-\eta_{Z,0,N}(X))(Z-\eta_{Z,0,N}(X))^T][|2Y-\eta_Y(X)-\eta_{Y,0,N}(X)-2\beta A+\beta \eta_A+\beta\eta_{A,0,N}||\eta_{Y,0,N}-\eta_{Y}+\beta (\eta_A-\eta_{A,0,N})|]\Bigg\Vert\\
    \lesssim &\Bigg\Vert \mathbb{E}_{P_{0,N}}\Big[\mathbb{E}_{P_{0,N}}[(Z-\eta_{Z,0}(X))(Z-\eta_{Z,0,N}(X))^T|X]|\eta_{Y,0}-\eta_{Y}+\beta (\eta_A-\eta_{A,0,N})|\Big]\Bigg\Vert\\
    \lesssim & \mathbb{E}_{P_{0,N}} |\eta_{Y,0,N}-\eta_{Y}|+\mathbb{E}_{P_{0,N}} |\eta_{A,0,N}-\eta_{A}|\\
    \leq & \Vert \eta_{Y,0,N}-\eta_{Y}\Vert_{P_{0,N},2}+\Vert \eta_{A,0,N}-\eta_{A}\Vert_{P_{0,N},2}\leq \delta_N = o(1/\sqrt{m})
\end{align*}\normalsize

The last step is to show 

\begin{align*}
    \sup_{\beta \in B}\Big\Vert (\mathbb{P}_n-P_{0,N})M(\beta,\eta_0)\Big\Vert = o_p(1/\sqrt{m}).
\end{align*}
For $M=\Omega,\Omega^{k},\Omega^{kl},\Omega^{k,l}$. I will only show the case when $M = \Omega$. The proof for other cases are similar.
\begin{align*}
    &\Omega(\beta,\eta_0) \\
    = &(Z-\eta_{Z,0,N}(X))(Z-\eta_{Z,0,N}(X))^T(Y-\eta_{Y,0,N}-\beta A+\beta \eta_{A,0,N})^2\\
    = &(Z-\eta_{Z,0,N}(X))(Z-\eta_{Z,0,N}(X))^T[(A-\eta_{A,0,N})^2\beta^2+2\beta (A-\eta_{A,0,N})(Y-\eta_{Y,0,N})+(Y-\eta_{Y,0,N})^2].
\end{align*}
Therefore
\begin{align*}
    & \sup_{\beta \in B}\Big\Vert (\mathbb{P}_n-P_{0,N})\Omega(\beta,\eta_{0,N})\Big\Vert\\
\leq & \Big\Vert (\mathbb{P}_n-P_{0,N})(Z-\eta_{Z,0,N}(X))(Z-\eta_{Z,0,N}(X))^T(A-\eta_{A,0,N})^2\Big\Vert+\\
&\Big\Vert (\mathbb{P}_n-P_{0,N})(Z-\eta_{Z,0,N}(X))(Z-\eta_{Z,0,N}(X))^T(A-\eta_{A,0,N})(Y-\eta_{Y,0,N})\Big\Vert+\\
& \Big\Vert (\mathbb{P}_n-P_{0,N})(Z-\eta_{Z,0,N}(X))(Z-\eta_{Z,0,N}(X))^T(Y-\eta_{Y,0,N})^2\Big\Vert.
\end{align*}

I will show how to drive the rate for the first term, the second and third term can be handled similarly.

Let $S_i = \frac{1}{n}[(Z_i-\eta_{Z,0,N}(X_i))(Z_i-\eta_{Z,0,N}(X_i))^T(A_i-\eta_{A,0,N}(X_i))^2-E_{P_{0,N}}(Z_i-\eta_{Z,0,N}(X_i))(Z-\eta_{Z,0,N}(X_i))^T(A_i-\eta_{A,0,N}(X_i))^2]$.

Now we calculate the matrix variance parameter for $R=\sum_{i=1}^n S_i$:
\footnotesize
\begin{align*}
    &\nu(R)\\
    =&\Bigg\Vert \sum_{i=1}^N\mathbb{E}_{P_{0,N}}[S_iS^T_i]\Bigg\Vert\\
    =&\frac{1}{N}\Bigg\Vert \mathbb{E}_{P_{0,N}}\Bigg[(Z-\eta_{Z,0,N}(X))(Z-\eta_{Z,0,N}(X))^T(Z-\eta_{Z,0,N}(X))(Z-\eta_{Z,0,N}(X))^T(A-\eta_{A,0,N})^4\Bigg]-\\
    &\mathbb{E}_{P_{0,N}}\Bigg[(Z-\eta_{Z,0,N}(X))(Z-\eta_{Z,0,N}(X))^T(A-\eta_{A,0,N})^2\Bigg]\mathbb{E}_{P_{0,N}}\Bigg[(Z-\eta_{Z,0,N}(X))(Z-\eta_{Z,0,N}(X))^T(A-\eta_{A,0,N})^2\Bigg]\Bigg\Vert\\
    \leq & \frac{1}{N}\Bigg\Vert \mathbb{E}_{P_{0,N}}\Bigg[(Z-\eta_{Z,0,N}(X))(Z-\eta_{Z,0,N}(X))^T(Z-\eta_{Z,0,N}(X))(Z-\eta_{Z,0,N}(X))^T(A-\eta_{A,0,N})^4\Bigg]\Bigg\Vert\\
    &+\frac{1}{N}\Bigg\Vert\mathbb{E}_{P_{0,N}}\Bigg[(Z-\eta_{Z,0,N}(X))(Z-\eta_{Z,0,N}(X))^T(A-\eta_{A,0,N})^2\Bigg]\mathbb{E}_{P_{0,N}}\Bigg[(Z-\eta_{Z,0,N}(X))(Z-\eta_{Z,0,N}(X))^T(A-\eta_{A,0,N})^2\Bigg]\Bigg\Vert\\
    \leq & \frac{m}{N}\Bigg\Vert \mathbb{E}_{P_{0,N}}\Bigg[(Z-\eta_{Z,0,N}(X))(Z-\eta_{Z,0,N}(X))^T(A-\eta_{A,0,N})^4\Bigg]\Bigg\Vert+\\
    &\frac{1}{N}\Bigg\Vert\mathbb{E}_{P_{0,N}}\Bigg[(Z-\eta_{Z,0,N}(X))(Z-\eta_{Z,0,N}(X))^T(A-\eta_{A,0,N})^2\Bigg]\Bigg\Vert^2\\
    &\lesssim \frac{m}{N}
\end{align*}
\normalsize

Next we calculate the large deviation parameter.
\scriptsize
\begin{align*}
    &\Vert S_i \Vert^2=\Vert S_iS_i^T \Vert\\
    =& \frac{1}{N^2}\Bigg\Vert (Z_i-\eta_{Z,0,N}(X_i))(Z_i-\eta_{Z,0,N}(X_i))^T(Z-\eta_{Z,0,N}(X_i))(Z_i-\eta_{Z,0,N}(X_i))^T(A_i-\eta_{A,0,N}(X_i))^4\\
    &-\mathbb{E}_{P_{0,N}}[(Z_i-\eta_{Z,0,N}(X))(Z-\eta_{Z,0,N}(X))^T(A_i-\eta_{A,0,N}(X_i))^2](Z_i-\eta_{Z,0,N}(X_i))(Z_i-\eta_{Z,0,N}(X_i))^T(A_i-\eta_{A,0,N}(X_i))^2\\
    &-(Z_i-\eta_{Z,0,N}(X_i))(Z-\eta_{Z,0,N}(X_i))^T(A-\eta_{A,0,N}(X_i))^2\mathbb{E}_{P_{0,N}}[(Z_i-\eta_{Z,0,N}(X_i))(Z_i-\eta_{Z,0,N}(X_i))^T(A_i-\eta_{A,0,N}(X_i))^2]\\
    &+\mathbb{E}_{P_{0,N}}[(Z_i-\eta_{Z,0,N}(X))(Z-\eta_{Z,0,N}(X_i))^T(A_i-\eta_{A,0,N}(X_i))^2\mathbb{E}_{P_{0,N}}[(Z_i-\eta_{Z,0,N}(X_i))(Z_i-\eta_{Z,0,N}(X_i))^T(A_i-\eta_{A,0,N}(X_i))^2\Bigg\Vert \\
    \lesssim & \frac{m(A_i-\eta_{A,0,N}(X_i))^4}{N^2}\Bigg\Vert (Z_i-\eta_{Z,0,N}(X_i))(Z_i-\eta_{Z,0,N}(X_i))^T\Bigg\Vert+\frac{2(A_i-\eta_{A,0,N}(X_i))^2}{N^2}\Bigg\Vert (Z_i-\eta_{Z,0,N}(X_i))(Z_i-\eta_{Z,0,N}(X_i))^T\Bigg\Vert+\frac{1}{N^2}\\
    \lesssim & \frac{m^2}{N^2}[(A_i-\eta_{A,0,N}(X_i))^4+(A_i-\eta_{A,0,N}(X_i))^2]
\end{align*}
\normalsize
The large deviation parameter $L$ satisfies
\begin{align*}
    L^2=\mathbb{E}_{P_{0,N}}\max_i \Vert S_i\Vert\leq \frac{m^2}{N^2}\mathbb{E}_{P_{0,N}}\max_i[(A_i-\eta_{A,0,N}(X_i))^4+(A_i-\eta_{A,0,N}(X_i))^2]\lesssim \frac{m^2}{N^2}o(1) =o( \frac{m^2}{N^{2}}).
\end{align*}
The second last inequality is because we assumed $\mathbb{E}_{P_{0,N}}\vert A\vert^{4}<\infty$, therefore $\mathbb{E}_{P_{0,N}}\max_i|A|^4 =o(1)$ by Corollary 7.1 of \citet{zhang2021concentration}.

Therefore,
\begin{align*}
    \Vert R\Vert =O_p\left( \sqrt{\log(m)\frac{m}{N}}+\log(m)\sqrt{\frac{m^2}{N^{2}}}\right) = o_P\Big( \frac{1}{\sqrt{m}}\Big).
\end{align*}
\end{proof}

The proof for Theorem \ref{Theorem: ASMM} follows from the previous lemmas.
\begin{proof}[Proof of Theorem \ref{Theorem: ASMM}]
    So far we have verified all regularity condition except for Assumption \ref{Assumption: convergence of g-hat} (c) and Assumption \ref{Assumption: Further restrictions for uniform convergence}. 
    To verify Assumption \ref{Assumption: convergence of g-hat} (c), we have 
    \begin{align*}
        \sup_{\beta \in \mathcal{B}}\Vert \widehat{\Omega}(\beta,\eta_{0,N})-\widebar{\Omega}(\beta,\eta_{0,N})\Vert_{F}\leq \sup_{\beta \in \mathcal{B}}\sqrt{m}\Vert \widehat{\Omega}(\beta,\eta_{0,N})-\widebar{\Omega}(\beta,\eta_{0,N})\Vert=o_p(1).
    \end{align*}
    Similarly, since we have already had $\sup_{\beta\in \mathcal{B}}\Vert (\mathbb{P}_n-P_{0,N})A(\beta,\eta_0) \Vert = o_{P_{0,N}}(1/\sqrt{m})$, we have $\sup_{\beta\in \mathcal{B}}\Vert \widehat{A}(\beta,\eta_0)-\widebar{A}(\beta,\eta_0) \Vert = o_{P_{0,N}}(1)$ for $A = \Omega^{(k)}$ or $A = \Omega^{k,l}$. Therefore, Assumption \ref{Assumption: Further restrictions for uniform convergence} (a) has been verified. Next, we verify that $\vert a^T[\widebar{\Omega}^{(k)}(\tilde{\beta},\eta_0)-\widebar{\Omega}^{(k)}({\beta},\eta_0)]b\vert \leq C\Vert a\Vert \Vert b\Vert \Vert \tilde{\beta}-\beta\Vert$.
\begin{align*}
    &\vert a^T[\widebar{\Omega}^{(k)}(\tilde{\beta})-\widebar{\Omega}^{(k)}({\beta})]b\vert\\
    = & \vert a^T[\mathbb{E}_{P_{0,N}}(g_i(\tilde{\beta},\eta_{0,N})G_i(\eta_{0,N}))-\mathbb{E}_{P_{0,N}}(g_i({\beta},\eta_{0,N})G_i(\eta_{0,N}))]b\vert= \vert( \tilde{\beta}-\beta)a^T\mathbb{E}_{P_{0,N}}(G_iG_i^T)b\vert\\
    \lesssim & \vert\tilde{\beta}-\beta\vert\Vert a\Vert \Vert b\Vert.
\end{align*}
    Therefore, Assumption \ref{Assumption: Further restrictions for uniform convergence} is satisfied. According to Theorem \ref{Theorem: consistency and ASN}, we obtain the desired result. 
\end{proof}

\section{Multiplicative SMM example} \label{Supp: MSMM}

$g$ can be rewritten in the following way:
\begin{align*}
    g(O;\beta,\eta_N) =& (Z-\mu_{Z,N}(X))[I\{A=1\}Y-\eta_{Y,(1),N}(X)\eta_{A,(1),N}(X)]\exp(-\beta)\\
    &+(Z-\mu_{Z,N}(X))[I\{A=0\}Y-\eta_{Y,(0),N}(X)(1-\eta_{A,(1),N})(X)].
\end{align*}
Therefore, $g(O;\beta,\eta_N)$ is a separable moment function with $B=2$, $q=1$,
\begin{align*}
    g^{[1]}(o;\eta) &= (z-\eta_{Z,N}(x))[I\{a=1\}y-\eta_{Y,(1),N}(x)\eta_{A,(1),N}(x)],\\
    g^{[2]}(o;\eta) &= (z-\eta_{Z,N}(x))[I\{a=0\}y-\eta_{Y,(0),N}(x)(1-\eta_{A,(1),N})(x)],\\
    h^{[1]}(\beta)& = \exp(\beta),\\
    h^{[2]}(\beta) &= 1.
\end{align*}

Verifying conditions for MSMM is very similar to the ASMM case. Verifying assumption 2 and assumption 3 relies on the following lemma:
\begin{lemma}
    Suppose $\beta_1,\beta_2\in \mathcal{B}\subset \mathbb{R}$, where $\mathcal{B}$ is a compact set. Then there exists a constant $C_1$, $C_2$ such that 
    \begin{equation*}
        C_1\vert \beta_1-\beta_2\vert \leq \vert\exp(-\beta_1)-\exp(-\beta_2)\vert\leq C_2\vert \beta_1-\beta_2\vert.
    \end{equation*}
\end{lemma}
\begin{proof}
 Without loss of generality, we assume $\beta_1\leq \beta_2$. Since $\beta_1,\beta_2$ belongs to some compact set, we know there exist $c_1,c_2$ such that $c_1\leq \beta_1\leq \beta_2\leq c_2$. By mean value theorem, we know there exists $\beta_3$ such that $\beta_2\in[\beta_1,\beta_2]$ and $\vert\exp(-\beta_1)-\exp(-\beta_2)\vert=\exp(-\beta_3)\vert \beta_1-\beta_2\vert$. Let $C_1=\exp(-c_2)$ and $C_2=\exp(-c_1)$, we then have $ C_1\vert \beta_1-\beta_2\vert \leq \vert\exp(-\beta_1)-\exp(-\beta_2)\vert\leq C_2\vert \beta_1-\beta_2\vert$.
\end{proof}

\begin{lemma}
    For all $\beta \in \mathcal{B}$, $1/C \leq \xi_{min}(\widebar{\Omega}(\beta,\eta_{0,N}))\leq \xi_{max}(\widebar{\Omega}(\beta,\eta_{0,N}))\leq C$.
    \label{Lemma: MSMM 3b}
\end{lemma}
\begin{proof}
    \begin{align*}
        &\widebar{\Omega}(\beta,\eta_{0,N})=\mathbb{E}_{P_{0,N}}[(Z-\eta_{Z,0,N}(X))(Z-\eta_{Z,0,N}(X))^T\\
        &((AY-\eta_{Y,(1),0,N}(X)\eta_{A,(1),0,N}(X))\exp(-\beta)+(1-A)Y-\eta_{Y,(0),0,N}(X)\eta_{A,(0),0,N}(X))^2]\\
        =&\mathbb{E}_{P_{0,N}}[(Z-\eta_{Z,0,N}(X))(Z-\eta_{Z,0,N}(X))^T\\
        &\mathbb{E}_{P_{0,N}}[((AY-\eta_{Y,(1),0,N}(X)\eta_{A,(1),0,N}(X))\exp(-\beta)+(1-A)Y-\eta_{Y,(0),0,N}(X)\eta_{A,(0),0,N}(X))^2|Z,X]]
    \end{align*}
    Since with probability one,
    \footnotesize
    \begin{align*}
        1/C\leq 
        \mathbb{E}_{P_{0,N}}[((AY-\eta_{Y,(1),0,N}(X)\eta_{A,(1),0,N}(X))\exp(-\beta)+(1-A)Y-\eta_{Y,(0),0,N}(X)\eta_{A,(0),0,N}(X))^2|Z,X]\\
        \leq C
    \end{align*}
    \normalsize
    for $\beta \in \mathcal{B}$. We have for all $\beta \in \mathcal{B}$, $1/C \leq \xi_{min}(\widebar{\Omega}(\beta,\eta_{0,N}))\leq \xi_{max}(\widebar{\Omega}(\beta,\eta_{0,N}))\leq C$.
\end{proof}

\begin{lemma}
    (a) There is $C>0$ with $\vert \beta-\beta_0\vert\leq C\sqrt{N}\Vert \widebar{g}(\beta,\eta_{0,N}) \Vert/\mu_N$ for all $\beta \in \mathcal{B}$. (b) There is $C>0$ and $\widehat{M}=O_p(1)$ such that $\vert \beta-\beta_0\vert\leq C\sqrt{N}\Vert \widebar{g}(\beta,\eta_{0,N})\Vert/\mu_N+\widehat{M}$ for all $\beta\in \mathcal{B}$.
\end{lemma}
\begin{proof}
\textbf{Proof for part (a):}
    \begin{align*}
        \widebar{G} &= \mathbb{E}_{P_{0,N}}\left[(Z-\mu_{Z,0,N}(X))(I\{A=1\}Y-\eta_{Y,(1),0,N}(X)\eta_{A,(1),0,N}(X))\exp(\beta_0)\right]\\
        &=\mathbb{E}_{P_{0,N}}[g^{[1]}(O;\eta_{0,N})]\exp(\beta_0).
    \end{align*}
    Combing with the fact that $\exp(\beta)$ is bounded and $\beta \in \mathcal{B}$, $1/C \leq \xi_{min}(\widebar{\Omega}(\beta,\eta_{0,N}))\leq \xi_{max}(\widebar{\Omega}(\beta,\eta_{0,N}))\leq C$ for all $\beta \in \mathcal{B}$, we have
    \begin{align*}
       c_1\mu_N^2/N \leq \widebar{g}^{[1]}(O;\eta_0)^T\widebar{g}^{[1]}(O;\eta_0)\leq c_2\mu_N^2/N
    \end{align*}
    for some constants $c_1$ and $c_2$.
    Since 
    \begin{align*}
        \widebar{g}(\beta,\eta_{0,N}) = \widebar{g}(\beta,\eta_{0,N})-\widebar{g}(\beta_0,\eta_{0,N})=g^{[1]}(\eta_{0,N})(\exp(\beta)-\exp(\beta_0)),
    \end{align*}
    similar to the ASMM case, we have
    \begin{align*}
        \sqrt{N}\Vert \widebar{g}(\beta,\eta_{0,N})\Vert/\mu_N = \sqrt{N}/\mu_N\vert\exp(\beta)-\exp(\beta_0)\vert\sqrt{\widebar{g}^T(\eta_{0,N})\widebar{g}(\eta_{0,N})} \gtrsim \vert \beta-\beta_0\vert.   \end{align*}

        \noindent \textbf{Proof for part (b).} The proof for part (b) is similar to that of ASMM, thus omitted.
\end{proof}

\begin{lemma}
    $g^{[1]}$ and $g^{[2]}$ satisfy Neyman orthogonality. Hence $g$ satisfies Neyman orthogonality.
\end{lemma}
\begin{proof}
    Let $\eta_1=(\eta_{Y,(0),1},\eta_{Y,(1),1},\eta_{A,(1),1},\eta_{Z,1})$.  Then for $g^{[1]}(\eta)$,
    \begin{align*}
        &\mathbb{E}_{P_{0,N}}\left[g^{[1]}(O;\beta,(1-t)\eta_{0,N}+t\eta_1)\right]\\
        =&\mathbb{E}_{P_{0,N}}\big[(Z-(1-t)\eta_{Z,0,N}(X)+t\eta_{Z,1}(X))\\
        &\times (I\{A=1\}Y-((1-t)\eta_{Y,(1),0,N}(X)+t\eta_{Y,(1),1}(X))((1-t)\eta_{A,(1),0,N}(X)+t\eta_{A,(1),1}(X)))\big].
    \end{align*}
    The first order Gatuex derivative is
    \begin{align*}
        &\frac{\partial }{\partial t}\mathbb{E}_{P_{0,N}}\left[g^{[1]}(O;(1-t)\eta_{0,N}+t\eta_1)\right]\\
        =&\mathbb{E}_{P_0}\Big[(\eta_{Z,0,N}(X)-\eta_{Z,1}(X))(I\{A=1\}Y-((1-t)\eta_{Y,(1),0,N}(X)+t\eta_{Y,(1),1}(X))\\
        &\times((1-t)\eta_{A,(1),0,N}(X)+t\eta_{A,(1),1}(X)))\Big]\\
        &+(Z-(1-t)\eta_{Z,0,N}(X)+t\eta_{Z,1}(X))(\eta_{Y,(1),0,N}(X)-\eta_{Y,(1),1}(X))((1-t)\eta_{A,(1),0,N}(X)+t\eta_{A,(1),1}(X))\\
        &+(Z-(1-t)\eta_{Z,0,N}(X)+t\mu_{Z,1}(X))(\eta_{A,(1),0,N}(X)-\eta_{A,(1),1}(X))((1-t)\eta_{Y,(1),0,N}(X)+t\eta_{Y,(1),1}(X)).
    \end{align*}
    Therefore,
    \begin{align*}
        &\frac{\partial }{\partial t}\mathbb{E}_{P_{0,N}}\left[g^{[1]}(O;\beta,(1-t)\eta_{0,N}+t\eta_1)\right]\Big\vert_{t=0}\\
        =&\mathbb{E}_{P_{0,N}}\left[(\mu_{Z,0}(X)-\mu_{Z,1}(X))(I\{A=1\}Y-\eta_{Y,(1),0,N}\eta_{A,(1),0,N})\right]\\
        &+\mathbb{E}_{P_{0,N}}\left[(Z-\mu_{Z,0,N}(X))(\eta_{Y,(1),0,N}(X)-\eta_{Y,(1),1}(X))\eta_{A,(1),0,N}(X)\right]\\
        &+\mathbb{E}_{P_{0,N}}\left[(Z-\mu_{Z,0,N}(X))(\eta_{A,(1),0,N}(X)-\eta_{A,(1),1}(X))\eta_{Y,(1),0,N}(X)\right]\\
        =&0.
    \end{align*}
    Therefore, $g^{[1]}(\eta)$ satisfies Neyman orthogonality. 

    We can show $g^{[2]}(\eta)$ is Neyman orthogonal in a similar way, and hence $g(\beta,\eta)$ is Neyman orthogonal.
\end{proof}

The next result is related to the statistical rate for the second order Gateux drivative of $g$.
\begin{lemma}
For $j=1,...,m$,
    \begin{align*}
        \lambda_N^{(j)}\leq N^{-1/2}\delta_N.
    \end{align*}
\end{lemma}
\begin{proof}
    We calculate the second-order Gateux derivative for $g^{[1],(j)}(\eta)$:
    \footnotesize
    \begin{align*}
        &\frac{\partial^2}{\partial t^2}\mathbb{E}_{P_{0,N}}[g^{[1],(j)}((1-t)\eta_0+t\eta_{1})]\\
        =&3\mathbb{E}_{P_{0,N}}\left[(\eta_{Z^{(j)},0,N}(X)-\eta_{Z^{(j)},1}(X))(\eta_{Y,(1),0,N}(X)-\eta_{Y,(1),1}(X))((1-t)\eta_{A,(1),0,N}+t\eta_{A,(1),1})\right]\\
        &+3\mathbb{E}_{P_{0,N}}\left[(\eta_{Z^{(j)},0,N}(X)-\eta_{Z^{(j)},1}(X))(\eta_{A,(1),0,N}(X)-\eta_{A,(1),1}(X))((1-t)\eta_{Y,(1),0,N}+t\eta_{Y,(1),1})\right]\\
        &+2\mathbb{E}_{P_{0,N}}\left[(Z^{(j)}-(1-t)\eta_{Z^{(j)},0,N}(X)-t\eta_{Z^{(j)},1}(X))(\eta_{A,(1),0,N}(X)-\eta_{A,(1),1}(X))(\eta_{Y,(1),1}(X)-\eta_{Y,(1),0,N}(X))\right].
    \end{align*}
    \normalsize
    Therefore,
    \begin{align*}
        &\sup_{t\in (0,1),\eta\in \mathcal{T}_N}\left\vert\frac{\partial^2}{\partial t^2}\mathbb{E}_{P_{0,N}}[g^{[1],(j)}((1-t)\eta_{0,N}+t\eta_1)]\right\vert\\
        \lesssim & \mathbb{E}_{P_{0,N}}\left[\vert(\eta_{Z^{(j)},0,N}(X)-\eta_{Z^{(j)},1}(X))(\eta_{Y,(1),0,N}(X)-\eta_{Y,(1),1}(X))\vert\right]\\
&+\mathbb{E}_{P_{0,N}}\left[\vert(\eta_{Z^{(j)},0,N}(X)-\eta_{Z^{(j)},1}(X))(\eta_{A,(1),0,N}(X)-\eta_{A,(1),1}(X))\vert\right]\\
        &+\mathbb{E}_{P_{0,N}}\left[\vert(\eta_{A,(1),0,N}(X)-\eta_{A,(1)}(X))(\eta_{Y,(1)}(X)-\eta_{Y,(1),0,N}(X))\vert\right]\\
        \leq & \Vert \eta_{Z^{(j)},0,N}-\eta_{Z^{(j)},1}\Vert_{P_{0,N},2}\Vert\eta_{Y,(1),0,N}-\eta_{Y,(1),1}\Vert_{P_{0,N},2}+\Vert \eta_{Z^{(j)},0,N}-\eta_{Z^{(j)},1}\Vert_{P_{0,N},2}\Vert\eta_{A,(1),0,N}-\eta_{A,(1),1}\Vert_{P_{0,N},2}\\
        &+\Vert \eta_{Y,(1),0,N}-\eta_{Y,(1)}\Vert_{P_{0,N},2}\Vert\eta_{A,(1),0,N}-\eta_{A,(1),1}\Vert_{P_{0,N},2}\\
        & \leq N^{-1/2}\delta_N.
    \end{align*}
    Similarly, we can prove that
    \begin{equation*}
        \sup_{t\in (0,1),\eta_1\in \mathcal{T}_N}\left\vert\frac{\partial^2}{\partial t^2}\mathbb{E}_{P_{0,N}}[g^{[2],(j)}((1-t)\eta_{0,N}+t\eta_1)]\right\vert\leq N^{-1/2}\delta_N.
    \end{equation*}
    Therefore
    \begin{align*}
        \lambda^{(j)}_N&:=\sup_{\beta\in B,t\in(0,1),\eta\in \mathcal{T}_N}\Bigg\vert \frac{\partial^2}{\partial t^2} \mathbb{E}_{P_{0,N}}g^{(j)}(O;\beta,\eta_{0,N}+t(\eta-\eta_{0,N}))\Bigg\vert\\
        &\leq \sup_{t\in (0,1),\eta\in \mathcal{T}_N}\left\vert\frac{\partial^2}{\partial t^2}\mathbb{E}[g^{[1],(j)}((1-t)\eta_{0,N}+t\eta)]\right\vert\sup_{\beta\in \mathcal{B}}\exp(\beta)+\sup_{t\in (0,1),\eta\in \mathcal{T}_N}\left\vert\frac{\partial^2}{\partial t^2}\mathbb{E}[g^{[2],(j)}((1-t)\eta_{0,N}+t\eta)]\right\vert\\
        &\leq N^{-1/2}\delta_N.
    \end{align*}
\end{proof}

    \begin{lemma}
$\sup_{\beta\in\mathcal{B}}\mathbb{E}_{P_{0,N}}[(g(\beta,\eta_{0,N})^Tg(\beta,\eta_{0,N}))^2]/N\rightarrow 0$. \label{lemma: MSMM 3a}
\end{lemma}
\begin{proof}
    The proof is similar to the proof of \ref{lemma: ASMM 3a}.
\end{proof}

    \begin{lemma}
$\vert a^T \{\widebar{\Omega}(\beta',\eta_{0,N})-\widebar{\Omega}(\beta,\eta_{{0,N}})\}b\vert \leq C\Vert a\Vert \Vert b \Vert \Vert \beta'-\beta\Vert$.
\label{Lemma: MSMM 3d}
\end{lemma}
\begin{proof}
By matrix Cauchy-Schwarz inequality \citet{gautam1999matrixcauchy}, we have 
\begin{align*}
    \mathbb{E}_{P_{0,N}}[G_i(\beta,\eta_{0,N})g_i(\beta,\eta_{0,N})^T]\overline{\Omega}(\beta,\eta_{0,N})^{-1}\mathbb{E}_{P_{0,N}}[g_i(\beta,\eta_{0,N})G(\beta,\eta_{0,N})^T_i]\leq \mathbb{E}_{P_{0,N}}[G_i(\beta,\eta_{0,N})G_i(\beta,\eta_{0,N})^T].
\end{align*}
Since $\Vert G_i(\beta,\eta_{0,N})\Vert\leq C$, we have $\Vert \mathbb{E}_{P_{0,N}}[G_i(\beta,\eta_{0,N})g_i(\beta,\eta_{0,N})^T]\Vert\leq C$ and $\Vert \mathbb{E}_{P_{0,N}}[g^{[1]}_i(\beta,\eta_{0,N})g_i(\beta,\eta_{0,N})^T]\Vert\leq C$. 

Since $g(\beta',\eta_{0,N}) = g(\beta,\eta_{0,N})+g^{[1]}(\eta_{0,N})(\exp(\beta')-\exp(\beta))$, we have
    \begin{align*}
        &\vert a^T \{\widebar{\Omega}(\beta',\eta_{0,N})-\widebar{\Omega}(\beta,\eta_{{0,N}})\}b\vert\\
        \leq & \vert (\exp(\beta')-\exp(\beta))^2a^T \mathbb{E}_{P_{0,N}}\{g^{[1]}_i(\eta_{0,N})g^{[1]}_i(\eta_{0,N})^T\}b\vert\\
        &+\vert (\exp(\beta')-\exp(\beta))a^T \mathbb{E}_{P_{0,N}}\{g^{[1]}_i(\eta_{0,N})g_i(\beta,\eta_0)^T\}b\vert\\
        &+\vert (\exp(\beta')-\exp(\beta))a^T \mathbb{E}_{P_{0,N}}\{g_i(\beta,\eta_0)g^{[1]}_i(\eta_{0,N})^T_i\}b\vert\\
        \lesssim & \vert\beta'-\beta\vert\Vert a\Vert \Vert b\Vert.
    \end{align*}
\end{proof}

\begin{lemma}
    There is $C$ and $\widehat{M}=O_p(1)$ such that for all $\beta',\beta \in \mathcal{B}$.
    \begin{align*}
        &\sqrt{N}\Vert \widebar{g}(\beta',\eta_{0,N})-\widebar{g}(\beta,\eta_{0,N})\Vert/\mu_N \leq C\vert \beta'-\beta\vert,\\
        &\sqrt{N}\Vert \widehat{g}(\beta',\eta_{0,N})-\widehat{g}(\beta_0,\eta_{0,N})\Vert/\mu_N \leq C\vert \beta'-\beta\vert.
    \end{align*}
\end{lemma}

\begin{proof}
    the Proof is similar to that of Lemma \ref{Lemma: ASMM 3e}, thus omitted.
\end{proof}

\begin{lemma}
\begin{equation*}
    (\mathbb{E}_{P_{0,N}}\Vert g_i\Vert^4+\mathbb{E}_{P_{0,N}}\Vert G_i\Vert^4)m/N\rightarrow 0.
\end{equation*}
\end{lemma}
\begin{proof}
    The Proof is similar to that of Lemma \ref{Lemma: ASMM 7b}, thus omitted.
\end{proof}

\begin{lemma}
\begin{align*}
      &\left\Vert \mathbb{E}_{P_{0,N}}[G_iG_i^T]\right\Vert\leq C, \left\Vert \mathbb{E}_{P_{0,N}}\left[\frac{\partial G_i(\beta_0,\eta_{0,N})}{\partial \beta}\frac{\partial G_i(\beta_0,\eta_{0,N})^T}{\partial \beta}\right]\right\Vert\leq C,\\
      &\frac{\sqrt{N}}{\mu_N} \left\Vert \mathbb{E}_{P_{0,N}}\left[\frac{\partial G_i(\beta_0)}{\partial\beta}\right]\right\Vert\leq C.
    \end{align*}
\end{lemma}
\begin{proof}
 The proof for the first and second inequality is similar to the proof of Lemma \ref{Lemma: MSMM 3b}. To show the third one, we just need to note that $\mathbb{E}_{P_{0,N}}\left[\frac{\partial G_i(\beta_0)}{\partial\beta}\right] = \overline{G}$, then the result follows by Assumption \ref{Assumption: MSMM} (a).
\end{proof}

\begin{lemma}
    If $\Bar{\beta}\xrightarrow[]{p} \beta_0$, then for $k=1,...,p$,
    \begin{align*}
       \left \Vert \sqrt{N}/\mu_N\left[\widehat{G}(\Bar{\beta},\eta_{0,N})-\widehat{G}({\beta}_0,\eta_{0,N})\right]\right\Vert = o_{p}(1),\Bigg\Vert \sqrt{N}/\mu_N\Bigg[\frac{\partial \widehat{G}(\Bar{\beta},\eta_{0,N})}{\partial \beta^{(k)}}-\frac{\partial \widehat{G}({\beta},\eta_{0,N})}{\partial \beta^{(k)}}\Bigg]\Bigg\Vert= o_{p}(1).
    \end{align*}
\end{lemma}
\begin{proof}
     \begin{align*}
        &\Vert \sqrt{N}[\widehat{G}(\Bar{\beta},\eta_0)-\widehat{G}({\beta}_0,\eta_0)]/\mu_N \Vert\\
        =&\frac{\sqrt{N}}{\mu_N}\vert\exp(\widebar{\beta})-\exp(\beta_0)\vert\sqrt{\widehat{g}^{[1],T}(\eta_{0,N})\widehat{g}^{[1],T}(\eta_0)}\\
        =&\frac{\sqrt{N}}{\mu_N}\vert\exp(\widebar{\beta})-\exp(\beta_0)\vert\exp(-\beta_0)\sqrt{\widehat{G}^{T}(\eta_{0,N})\widehat{G}^{T}(\eta_0)}.
    \end{align*}
    Similar to the ASMM case, we can prove that
    \begin{align*}
        \Vert \widehat{G}^{T}(\eta_{0,N})\widehat{G}^{T}(\eta_{0,N})\Vert = O_P\left(\frac{\mu_N}{\sqrt{N}}\right).
    \end{align*}
    Since $\vert\exp(\widebar{\beta})-\exp(\beta_0)\vert=o_p(1)$, we have
    \begin{equation*}
        \Vert \sqrt{N}[\widehat{G}(\Bar{\beta},\eta_{0,N})-\widehat{G}({\beta}_0,\eta_{0,N})]/\mu_N \Vert = o_{p}(1).
    \end{equation*}
    The proof for
    \begin{align*}
        \Bigg\Vert \sqrt{N}\Bigg[\frac{\partial \widehat{G}(\Bar{\beta},\eta_{0,N})}{\partial \beta^{(k)}}-\frac{\partial \widehat{G}({\beta}_{0,N},\eta_{0,N})}{\partial \beta^{(k)}}\Bigg] /\mu_N\Bigg\Vert= o_{p}(1)
    \end{align*}
    is similar, thus omitted.
\end{proof}

The next lemma says the  convergence rate conditions stated in Assumption \ref{Assumption: separable score} hold.
\begin{lemma}
    Assumption \ref{Assumption: separable score} holds for the MSMM example.
\end{lemma}

\begin{proof}
    We have shown Neyman othogonality condition and $\lambda_N\leq N^{-1/2}\delta_N$. It remains to show that $z_N\leq \delta_N$ for $a_N = r^{(b_1),(j,q)}_N, r^{(b_1,b_2),(j,k,l,r)}_N$. We will show that $r^{(1),(1,1)}_N\leq \delta_N$ as example. Other proofs are similar and omitted.

    We pick $\eta_{1}$ from the realization set $\mathcal{T}_N$. For $r^{(1),(1,1)}_N$, we have
    \begin{align*}
        &r^{(1),(1,1)}\\
        =&\Vert (Z^{(1)}-\eta_{Z^{(1)},1})(AY-\eta_{Y,(1),1}(X)\eta_{A,(1),1}(X))-\\
        &(Z^{(1)}-\eta_{Z^{(1)},0,N})(AY-\eta_{Y,(1),0,N}(X)\eta_{A,(1),0,N}(X))\Vert_{P_{0,N},2}\\
        \leq &\Vert AY(\eta_{Z^{(1)},0,N}-\eta_{Z^{(1)},1})\Vert_{P_{0,N},2}+\Vert Z^{(1)}(\eta_{Y,(1),1}(X)\eta_{A,(1),1}(X))-\eta_{Y,(1),0,N}(X)\eta_{A,(1),0,N}(X)))\Vert_{P_{0,N},2}\\
        &+\Vert \eta_{Z^{(1)},1}(X)\eta_{Y,(1),1}(X)\eta_{A,(1),1}(X)-\eta_{Z^{(1)},0,N}(X)\eta_{Y,(1),0,N}(X)\eta_{A,(1),0,N}(X)))\Vert_{P_{0,N},2}
    \end{align*}

    For the first term, we have
    \begin{align*}
        &\Vert AY(\eta_{Z^{(1)},0,N}-\eta_{Z^{(1)},1})\Vert_{P_{0,N},2}\\
        \leq &\sqrt{\mathbb{E}_{P_{0,N}}(\eta_{Z^{1},0,N}(X)-\eta_{Z^{1},1}(X))^2\mathbb{E}_{P_{0,N}}[Y^2|A,X]\mathbb{E}_{P_{0,N}}[A^2|X]}\leq   \Vert\eta_{Z^{1},0,N}-\eta_{Z^{1},1}\Vert_{P_{0,N},2}.
    \end{align*}

    For the second term, we have 
    \begin{align*}
        &\Vert Z^{(1)}(\eta_{Y,(1),1}(X)\eta_{A,(1),1}(X))-\eta_{Y,(1),0,N}(X)\eta_{A,(1),0,N}(X)))\Vert_{P_{0,N},2}\\
        \leq &\Vert (\eta_{Y,(1),1}(X)\eta_{A,(1),1}(X))-\eta_{Y,(1),0,N}(X)\eta_{A,(1),0,N}(X)))\Vert_{P_{0,N},2} \\
        \leq & \Vert (\eta_{Y,(1),1}(X)-\eta_{Y,(1),0,N}(X))\eta_{A,(1),1}(X)\Vert_{P_{0,N},2}+\Vert (\eta_{A,(1),1}(X)-\eta_{A,(1),0,N}(X))\eta_{Y,(1),0,N}(X)\Vert_{P_{0,N},2}\\
        \lesssim &\delta_N
    \end{align*}

    Similarly we can show that 
    \begin{align*}
        \Vert \eta_{Z^{(1)},1}(X)\eta_{Y,(1),1}(X)\eta_{A,(1),1}(X)-\eta_{Z^{(1)},0,N}(X)\eta_{Y,(1),0,N}(X)\eta_{A,(1),0,N}(X)))\Vert_{P_{0,N},2}\leq C\delta_N.
    \end{align*} Therefore, $r^{(1),(1,1)}\leq C\delta_N$.
\end{proof}

\begin{lemma}
    Assumption \ref{Assumption: Further restrictions for uniform convergence} holds for the MSMM example.
\end{lemma}
\begin{proof}
    The proofs are similar to the proofs of Lemma \ref{Lemma: MSMM 3b} and \ref{Lemma: MSMM 3d}.
\end{proof}

\begin{lemma}
    Assumption \ref{Assumption: ASN matrix estimation} holds for MSMM example.
\end{lemma}
\begin{proof}
    The proof is similar to the proof of the ASMM case, thus omitted.
\end{proof}

The proof of Theorem \ref{Theorem: MSMM} is a direct result of Lemmas in Supplemental material \ref{Supp: MSMM} and Theorem \ref{Theorem: consistency and ASN}.

\section{Proximal causal inference example}
\subsection{Proof for the orthogonality property of the proposed moment condition}
\begin{proof}
Note that
\begin{align*}
    &\mathbb{E}_{P_{0,N}}[Y-\beta_{a,0} A-\beta_{w,0}W|A,Z,U,X]\\
    =&\mathbb{E}_{P_{0,N}}[Y(0)-\beta_{w,0}W|A,Z,U,X]\\
    =&\mathbb{E}_{P_{0,N}}[Y(0)-\beta_{w,0}W|U,X]\\
    =&\mathbb{E}_{P_{0,N}}[Y(0)-\beta_{w,0}W|X]
\end{align*}
On the other hand,
\begin{align*}
    &\mathbb{E}_{P_{0,N}}[Y-\beta_{a,0} A-\beta_{w,0}W|X]\\
    =&\mathbb{E}_{P_{0,N}}[\mathbb{E}_{P_{0,N}}[Y-\beta_{a,0} A-\beta_{w,0}W|A,Z,U,X]|X]\\
    =&\mathbb{E}_{P_{0,N}}[Y(0)-\beta_{w,0}W|X]
\end{align*}
Therefore we have
\begin{align*}
    \mathbb{E}_{P_{0,N}}[Y-\beta A-\beta_wW|A,Z,U,X] = \mathbb{E}_{P_{0,N}}[Y-\beta A-\beta_wW|X].
\end{align*}
Hence, for any function $h$ of $(A,Z,U)$, we have
\begin{equation*}
    \mathbb{E}_{P_{0,N}}h(A,Z,U)(Y-\beta_{a,0} A-\beta_{w,0}W-\mathbb{E}_{P_{0,N}}[Y-\beta_{a,0} A-\beta_{w,0}W|X]) = 0,
\end{equation*}
specifically, 
\begin{align*}
    &\mathbb{E}_{P_{0,N}}(Z-E_{P_{0,N}}[Z|X])(Y-\beta_{a,0} A-\beta_{w,0}W-\mathbb{E}_{P_{0,N}}[Y-\beta_{a,0} A-\beta_{w,0}W|X]) = 0 \\
    &\mathbb{E}_{P_{0,N}}(A-E_{P_{0,N}}[A|X])(Y-\beta_{a,0} A-\beta_{w,0}W-\mathbb{E}_{P_{0,N}}[Y-\beta_{a,0} A-\beta_{w,0}W|X]) = 0.
\end{align*}
\end{proof}

\begin{lemma}
    The moment function
    \begin{align}
g(O;\beta,\eta)=\begin{pmatrix}
  A-\eta_A(X) \\
  Z-\eta_Z(X)
\end{pmatrix}(Y-\beta_{a} A-\beta_{w}W-\eta_{Y}(X)+\beta_{a}\eta_{A}(X)+\beta_{w}\eta_{W}(X)), \label{Equation: Moment function for PSMM}
\end{align}
     satisfies global Neyman orthogonality. 
\end{lemma}

\begin{proof}
    Let 
    \begin{equation*}
        \tilde{g}(O;\beta,\eta) = (Z-\eta_Z)(Y-\eta_Y-\beta_A(A-\eta_A)-\beta_W(W-\eta_W)).
    \end{equation*}
    The first-order Gateux derivative of $\tilde{g}$ with direction $\eta_{1}$ is
    \begin{align*}
        &\frac{\partial}{\partial t}\mathbb{E}_{P_{0,N}}\left[\tilde{g}(O;\beta,
        (1-t)\eta_{0,N}+t\eta)\right]\\
        =&\mathbb{E}_{P_{0,N}}\Big[(\eta_{Z,0,N}(X)-\eta_{Z,1}(X))\\
        &\times [Y-(1-t)\eta_{Y,0,N}(X)-t\eta_{Y,1}(X)\\
        &-\beta_A(A-(1-t)\eta_{A,0,N}(X)-t\eta_{A,1}(X))-\beta_W(W-(1-t)\eta_{W,0,N}(X)-t\eta_{W,1}(X))]\\
        &+(Z-(1-t)\eta_{Z,0,N}(X)-t\eta_{Z,1}(X))\\
        &\times(\eta_{Y,0,N}(X)-\eta_{Y,1}(X)-\beta_A(\eta_{A,0,N}(X)-\eta_{A,1}(X))-\beta_{W,1}(\eta_{W,0,N}(X)-\eta_{W,1}(X)))\Big].
    \end{align*}
    Therefore,
    \begin{align*}
        \frac{\partial}{\partial t}\mathbb{E}_{P_{0,N}}\left[\tilde{g}(O;\beta,
        (1-t)\eta_{0,N}+t\eta_1)\right]\Big\vert_{t=0} = 0.
    \end{align*}
    The result still holds if we replace $Z$ in $\tilde{g}$ with $A$. Therefore,
    \begin{align*}
        \frac{\partial}{\partial t}\mathbb{E}_{P_{0,N}}\left[{g}(O;\beta,
        (1-t)\eta_{0,N}+t\eta_1)\right]\Big\vert_{t=0} = 0.
    \end{align*}
\end{proof}

\begin{lemma}
    Under Assumption \ref{Assumption: PSMM}, we have for $j = 1,...,m+1$
    \begin{equation*}
        \lambda_N^{(j)}\leq N^{-1/2}\delta_N.
    \end{equation*}
\end{lemma}
\begin{proof}
    For $j\geq 2$, we have
    \begin{align*}
        &\frac{\partial^2}{\partial t^2}\mathbb{E}_{P_{0,N}}\left[g(O;\beta,(1-t)\eta_{0,N}+t\eta)\right]\\
        =& 2\mathbb{E}_{P_{0,N}}\left[(\eta_{Z^{(j)},0,N}(X)-\eta_{Z^{(j)}}(X))(\eta_{Y,0,N}(X)-\eta_Y(X)-\beta_A(\eta_{A,0,N}-\eta_A)-\beta_W(\eta_{W,0,N}-\eta_W))\right]
    \end{align*}
    Therefore,
    \begin{align*}
\lambda_N^{(j)}\leq & \sup_{\eta\in \mathcal{T}_N,\beta \in \mathcal{B}}\mathbb{E}_{P_{0,N}}\vert (\eta_{Z^{(j)},0,N}(X)-\eta_{Z^{(j)}}(X))(\eta_{Y,0,N}(X)-\eta_Y(X)) \vert \\
&+\sup_{\beta\in\mathcal{B}}\vert \beta_A\vert\sup_{\eta\in \mathcal{T}_N,\beta \in \mathcal{B}}\mathbb{E}_{P_{0,N}}\vert (\eta_{Z^{(j)},0,N}(X)-\eta_{Z^{(j)}}(X))(\eta_{A,0,N}(X)-\eta_A(X)) \vert\\
&+\sup_{\beta\in\mathcal{B}}\vert \beta_W\vert\sup_{\eta\in \mathcal{T}_N,\beta \in \mathcal{B}}\mathbb{E}_{P_{0,N}}\vert (\eta_{Z^{(j)},0,N}(X)-\eta_{Z^{(j)},0,N}(X))(\eta_{W,0,N}(X)-\eta_W(X)) \vert\\
\leq & \sup_{\eta\in \mathcal{T}_N} \Vert \eta_{Z^{(i)},0,N}-\eta_{Z^{(i)}}\Vert_{P_{0,N},2}(\Vert \eta_{Y,0,N}-\eta_{Y}\Vert_{P_{0,N},2}+\Vert \eta_{A,0,N}-\eta_{A}\Vert_{P_{0,N},2}+\Vert \eta_{W,0,N}-\eta_{W}\Vert_{P_{0,N},2})\\
\leq & N^{-1/2}\delta_N.
    \end{align*}
    The case when $j=1$ is similar, thus omitted. 
\end{proof}

Next, we will verify the regularity conditions for the proximal causal inference example. This example is similar to the ASMM example because the moment condition of both examples are linear in $\beta$. We will verify Assumption 2,3. Other assumptions can be verified similarly.

\begin{lemma}
    For all $\beta \in \mathcal{B}$, $1/C \leq \xi_{min}(\widebar{\Omega}(\beta,\eta_{0,N}))\leq \xi_{max}(\widebar{\Omega}(\beta,\eta_{0,N}))\leq C$.
\end{lemma}
\begin{proof}
    \begin{align*}
        &\widebar{\Omega}(\beta,\eta_{0,N})=\mathbb{E}_{P_{0,N}}[(\tilde{Z}-\eta_{\tilde{Z},0,N}(X))(\tilde{Z}-\eta_{\tilde{Z},0,N}(X))^T\\
        &\times(Y-\eta_{Y,0,N}(X)-\beta_a A+ \beta_a\eta_{A,0,N}(X)-\beta_wW+\eta_{W,0,N}(X))^2]\\
        =&\mathbb{E}_{P_{0,N}}\Big[(\tilde{Z}-\eta_{\tilde{Z},0,N}(X))(\tilde{Z}-\eta_{\tilde{Z},0,N}(X))^T\\
        &\times \mathbb{E}_{P_{0,N}}[(Y-\eta_{Y,0,N}(X)-\beta_a A+ \beta_a\eta_{A,0,N}-\beta_w +\beta_w\eta_{W,0,N}(X))^2|Z,X,A]\Big]
    \end{align*}
    Since $\mathbb{E}_{P_{0,N}}[(Y-\eta_{Y,0,N}(X)-\beta_a A+ \beta_a\eta_{A,0,N}-\beta_w +\beta_w\eta_{W,0,N}(X))^2|Z,X,A]$ is bounded by $1/C$ and $C$ and $\mathbb{E}_{P_{0,N}}[(\tilde{Z}-\eta_{\tilde{Z},0,N}(X))(\tilde{Z}-\eta_{\tilde{Z},0,N}(X))^T]$ is bounded by $CI_{m+1}$ and $1/CI_{m+1}$.
    
\end{proof}

\begin{lemma}
    (a) There is a constant $C>0$ with $\Vert \delta(\beta) \Vert \leq C\sqrt{N} \Vert \widebar{g}(\beta,\eta_{0,N}) \Vert/\mu_N$ for all $\beta \in B$. (b) there is $C>0$ and $\widehat{M} = O_p(1)$ such that $\Vert \delta(\beta) \Vert \leq C\sqrt{N}\Vert \widehat{g}(\beta,\eta_{0,N})\Vert/\mu_N +\widehat{M}$ for all $\beta \in B$, where $\widehat{M}$ is $O_p(1)$.
\end{lemma}
\begin{proof}
    For (a), we have
    \begin{align*}
        \widebar{g}(\beta,\eta_{0,N}) = \widebar{g}(\beta,\eta_{0,N})-\widebar{g}(\beta_0,\eta_{0,N}) = \overline{G}(\beta-\beta_0).
    \end{align*}
    Here $\overline{G}$ is a $(m+1)\times 2$ matrix.
    Therefore,
    \begin{align*}
        &\sqrt{N}\Vert \widebar{g}(\beta,\eta_{0,N})\Vert/\mu_N\\
        =& \sqrt{N}\sqrt{(\beta-\beta_0)^TG^T(\eta_{0,N})G(\eta_{0,N})(\beta-\beta_0)}/\mu_N\\
        =&\sqrt{\delta(\beta)^TNS_N^{-1}\overline{G}^T\overline{G}S_N^{-1,T}\delta(\beta)}\\
        \gtrsim& \sqrt{\delta(\beta)^TNS_N^{-1}\overline{G}^T\overline{\Omega}^{-1}\overline{G}S_N^{-1,T}\delta(\beta)}\gtrsim \Vert \delta(\beta)\Vert.
    \end{align*}

    To prove part (b), we first have the following expansion:
    \begin{align*}
        &\sqrt{N}\widehat{g}(\beta,\eta_{0,N})/\mu_N\\
        =& \sqrt{N}\widehat{g}(\beta_0,\eta_{0,N})/\mu_N+\mu_N^{-1}\sqrt{N}\frac{1}{N}\sum_{i=1}^N(G_i-\widebar{G})(\beta-\beta_0)+\sqrt{N}\widebar{G}(\beta-\beta_0)/\mu_N\\
        =&\sqrt{N}\widehat{g}(\beta_0,\eta_{0,N})/\mu_N+\mu_N^{-1}\sqrt{N}\frac{1}{N}\sum_{i=1}^N(G_i-\widebar{G})(\beta-\beta_0)+\sqrt{N}\widebar{G}S_N^{T,-1}\delta(\beta).
    \end{align*}
    Therefore,
    \begin{align*}
        &\Vert \delta(\beta)\Vert\\
        \lesssim & \Vert \sqrt{N}\widebar{G}S_N^{T,-1}\delta(\beta)\Vert\\
        \leq & \sqrt{N}\Vert \widehat{g}(\beta,\eta_{0,N})\Vert/\mu_N +\underbrace{\left\Vert \sqrt{N}\widehat{g}(\beta_0,\eta_{0,N})/\mu_N+\mu_N^{-1}\sqrt{N}\frac{1}{N}\sum_{i=1}^N(G_i-\widebar{G})(\beta-\beta_0)\right\Vert}_{\widehat{M}}.
    \end{align*}
    Since we have already shown that $\Vert\widehat{g}(\beta_0,\eta_0)\Vert = O_p(\sqrt{m/N})$, therefore $\sqrt{N}\Vert \widehat{g}(\beta,\eta_{0,N})\Vert/\mu_N = O_p(\sqrt{m}/\mu_N) = O(1)$. It remains to show that
    \begin{equation*}
        \sup_{\beta \in \mathcal{B}}\left\Vert\mu_N^{-1}\sqrt{N}\frac{1}{N}\sum_{i=1}^N(G_i-\widebar{G})(\beta-\beta_0)\right\Vert = O_p(1).
    \end{equation*}
Since 
    \begin{align*}
       &\mathbb{E}_{P_{0,N}}\Vert(\widehat{G}-\overline{G})(\beta-\beta_0)\Vert^2 \\
       =&\frac{1}{N}\text{tr}\mathbb{E}_{P_{0,N}}G_i(\beta-\beta_0)(\beta-\beta_0)^TG^T_i-\frac{1}{N}(\beta-\beta_0)\overline{G}^T\overline{G}(\beta-\beta_0).
    \end{align*}
    Similar to the ASMM example, we have 
    \begin{align*}
        &\frac{1}{N}\text{tr}\mathbb{E}_{P_{0,N}}G_i(\beta-\beta_0)(\beta-\beta_0)^TG^T_i = O(
        m/N)\\
        &\frac{1}{N}(\beta-\beta_0)\overline{G}^T\overline{G}(\beta-\beta_0) = O(\mu_N^2/N^2).
    \end{align*}
    Therefore, $\widehat{M} = O_p(1)$.
\end{proof}

\begin{lemma} $\sup_{\beta\in\mathcal{B}}\mathbb{E}_{P_{0,N}}[(g(\beta,\eta_{0,N})^Tg(\beta,\eta_{0,N}))^2]/N\rightarrow 0$. \label{lemma: PSMM 3a}
\end{lemma}
\begin{proof}
    Omitted.
\end{proof}

\begin{lemma}
    $\vert a^T \{\widebar{\Omega}(\beta',\eta_{0,N})-\widebar{\Omega}(\beta,\eta_{{0,N}})\}b\vert \leq C\Vert a\Vert \Vert b \Vert \Vert \beta'-\beta\Vert$.
\end{lemma}
\begin{proof}
\begin{align*}
    g_i(\beta,\eta_{0,N}) = G_i\beta+(\tilde{Z}_i-\eta_{\tilde{Z},0,N}(X_i))(Y_i-\eta_{Y,0,N}(X_i))
\end{align*}
Therefore,
\begin{align*}
&g_i(\beta,\eta_{0,N})g_i(\beta,\eta_{0,N})^T\\
    =&G_i\beta\beta^T G_i^T+2(\tilde{Z}_i-\eta_{\tilde{Z},0,N}(X_i))(Y_i-\eta_{Y,0,N}(X_i))\beta^TG_i^T\\
    &+(\tilde{Z}_i-\eta_{\tilde{Z},0,N}(X_i))(\tilde{Z}_i-\eta_{\tilde{Z},0,N}(X_i))^T(Y_i-\eta_{Y,0,N}(X_i))^2
\end{align*}
We start from bounding $\vert a^T \mathbb{E}_{P_{0,N}}[G_i^T\beta'\beta'^TG_i^T]b^T -  a^T \mathbb{E}_{P_{0,N}}[G_i^T\beta\beta^TG_i^T]b^T \vert$:
\scriptsize
\begin{align*}
    &\vert a^T \mathbb{E}_{P_{0,N}}[G_i^T\beta'\beta'^TG_i^T]b^T -  a^T \mathbb{E}_{P_{0,N}}[G_i^T\beta\beta^TG_i^T]b^T \vert\\
    =&\left\vert a^T \mathbb{E}_{P_{0,N}}\left[(\tilde{Z}-\eta_{\tilde{Z},0,N})\begin{pmatrix}
        (A-\eta_{A,0,N}(X))\\
        (W-\eta_{W,0,N}(X))
    \end{pmatrix}^T\beta'\beta'^T\begin{pmatrix}
        (A-\eta_{A,0,N}(X))\\
        (W-\eta_{W,0,N}(X))
    \end{pmatrix}(\tilde{Z}-\eta_{\tilde{Z},0,N})^T\right]b^T \right.\\
    - &\left.a^T \mathbb{E}_{P_{0,N}}\left[(\tilde{Z}-\eta_{\tilde{Z},0,N})\begin{pmatrix}
        (A-\eta_{A,0,N}(X))\\
        (W-\eta_{W,0,N}(X))
    \end{pmatrix}^T\beta\beta^T\begin{pmatrix}
        (A-\eta_{A,0,N}(X))\\
        (W-\eta_{W,0,N}(X))
    \end{pmatrix}(\tilde{Z}-\eta_{\tilde{Z},0,N})^T\right]b^T \right\vert\\
    =&\left\vert a^T \mathbb{E}_{P_{0,N}}\left[(\tilde{Z}-\eta_{\tilde{Z},0,N})\mathbb{E}_{P_{0,N}}\left[\begin{pmatrix}
        (A-\eta_{A,0,N}(X))\\
        (W-\eta_{W,0,N}(X))
    \end{pmatrix}^T\beta'\beta'^T\begin{pmatrix}
        (A-\eta_{A,0,N}(X))\\
        (W-\eta_{W,0,N}(X))
    \end{pmatrix}\Bigg|A,Z,X\right](\tilde{Z}-\eta_{\tilde{Z},0,N})^T\right]b^T \right.\\
    - &\left.a^T \mathbb{E}_{P_{0,N}}\left[(\tilde{Z}-\eta_{\tilde{Z},0,N}(X))\mathbb{E}_{P_{0,N}}\left[\begin{pmatrix}
        (A-\eta_{A,0,N}(X))\\
        (W-\eta_{W,0,N}(X))
    \end{pmatrix}^T\beta\beta^T\begin{pmatrix}
        (A-\eta_{A,0,N}(X))\\
        (W-\eta_{W,0,N}(X))
    \end{pmatrix}\Bigg|A,Z,X\right](\tilde{Z}-\eta_{\tilde{Z},0,N}(X))^T\right]b^T \right\vert\\
    \lesssim &\mathbb{E}_{P_{0,N}}[\vert a^T(\tilde{Z}-\eta_{\tilde{Z}0,N}(X)) \vert\vert b^T(\tilde{Z}-\eta_{\tilde{Z}0,N}(X)) \vert] \Vert \beta'-\beta\Vert\\
    \leq & \Vert a\Vert \Vert b \Vert \Vert \beta'-\beta\Vert
\end{align*}
\normalsize
where the second last inequality is due to boundedness of $\eta_{A,0,N}$, $\eta_{W,0,N}$ $\mathbb{E}_{P_{0,N}}[Y^4|A,X,Z]$ and $\mathbb{E}_{P_{0,N}}[W^4|A,X,Z]$. We can use the similar way to bound
\footnotesize
\begin{align*}
    \vert a^T \mathbb{E}_{P_{0,N}}[(\tilde{Z}_i-\eta_{\tilde{Z},0,N}(X_i))(Y_i-\eta_{Y,0,N}(X_i))\beta'^TG_i^T]b^T -  a^T \mathbb{E}_{P_{0,N}}[(\tilde{Z}_i-\eta_{\tilde{Z},0,N}(X_i))(Y_i-\eta_{Y,0,N}(X_i))\beta^TG_i^T]b^T \vert.
\end{align*}
\normalsize
The details are omitted.
\end{proof}

\begin{lemma}
    There is a constant $C$ and $\widehat{M}=O_p(1)$ such that for all $\beta',\beta \in \mathcal{B}$.
    \begin{align*}
        &\sqrt{N}\Vert \widebar{g}(\beta',\eta_{0,N})-\widebar{g}(\beta,\eta_{0,N})\Vert/\mu_N \leq C\vert \delta(\beta')-\delta(\beta)\vert,\\
        &\sqrt{N}\Vert \widehat{g}(\beta',\eta_{0,N})-\widehat{g}(\beta_0,\eta_{0,N})\Vert/\mu_N \leq \widehat{M}\vert \delta(\beta')-\delta(\beta)\vert.
    \end{align*}
    \label{Lemma: PSMM 3e}
\end{lemma}
\begin{proof}
    \begin{align*}
 &\mu_N^{-1}\sqrt{N}\Vert \widebar{g}(\beta',\eta_{0,N})-\widebar{g}(\beta,\eta_{0,N})\Vert = \mu_N^{-1}\sqrt{N}\Vert \overline{G}(\beta'-\beta)\Vert\\
  = &\mu_N^{-1}\sqrt{N}\sqrt{(\beta'-\beta)^TS_NS_N^{-1}\overline{G}^T\overline{G}S_N^{-1,T}S^T_N(\beta'-\beta)}
 \lesssim \Vert \delta(\beta')-\delta(\beta)\Vert.
    \end{align*}
    Similarly
    \begin{align*}
\mu_N^{-1}\sqrt{N}\Vert \widebar{g}(\beta',\eta_{0,N})-\widebar{g}(\beta,\eta_{0,N})\Vert = \mu_N^{-1}\sqrt{N}\Vert \widehat{G}(\eta_{0,N})(\beta'-\beta)\Vert\leq \sqrt{N}\Vert \widehat{G}(\eta_{0,N})\Vert\Vert S_N^{-1}\Vert \Vert \delta(\beta')-\delta(\beta)\Vert,
    \end{align*}
    where $\widehat{M} = \sqrt{N}\Vert \widehat{G}(\eta_{0,N})\Vert\Vert S_N^{-1}\Vert = O_p(1)$.
\end{proof}

 \section{Additional simulation details}
\subsection{Sampling distributions of the CUE, GMM and 2SLS in the simulation settings}


Figure \ref{fig:ASMM and MSMM} and Figure \ref{fig:PSMM} show the sampling distributions of estimates for GMM and CUE across 1000 simulations.

\begin{figure}
    \centering
\includegraphics[width=1\linewidth]{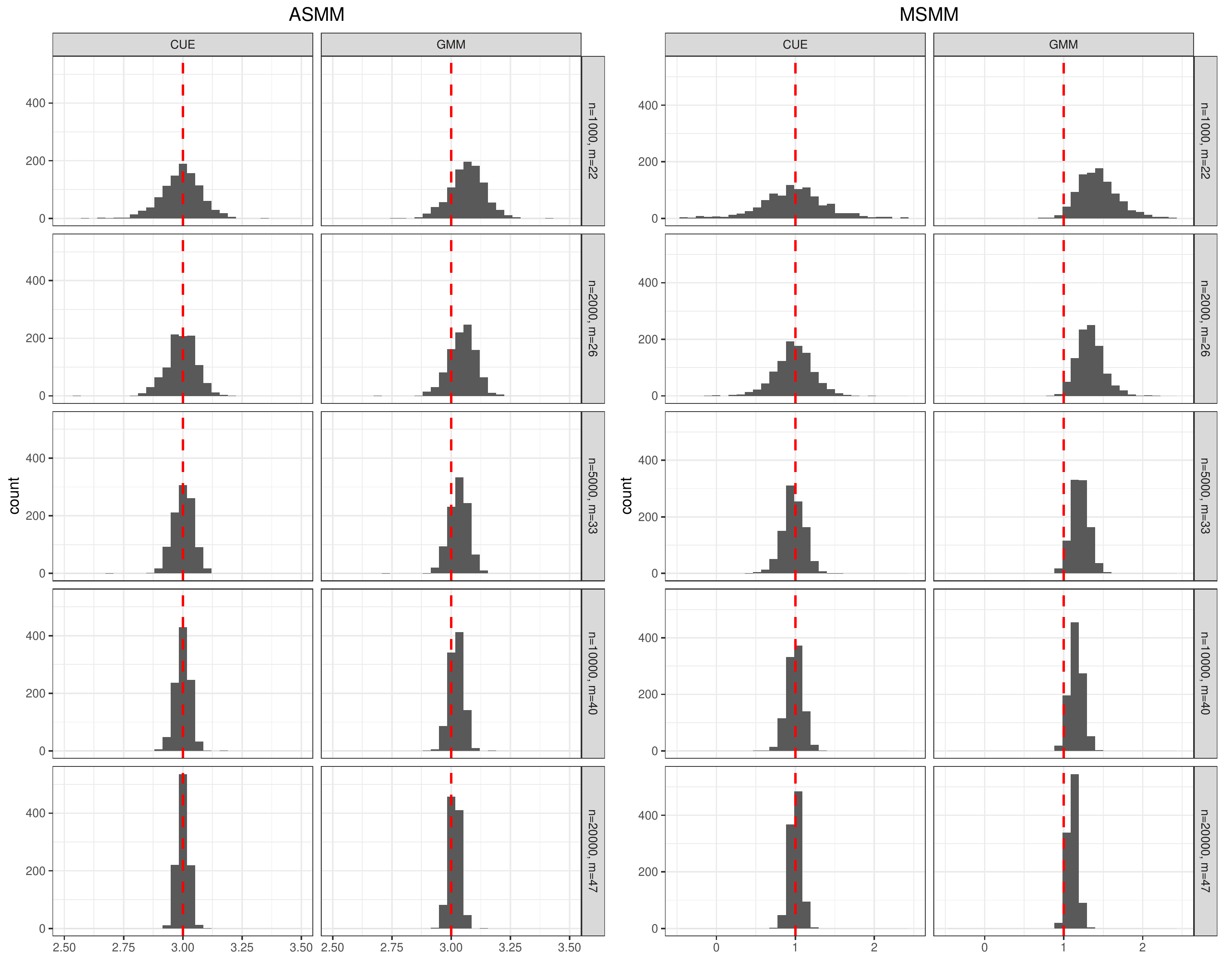}
    \caption{Sampling distribution of the CUE and GMM estimator under ASMM and MSMM settings. The dashed red lines represents the ground truth in each setting.}
    \label{fig:ASMM and MSMM}
\end{figure}

\begin{figure}
    \centering
    \includegraphics[width=1\linewidth]{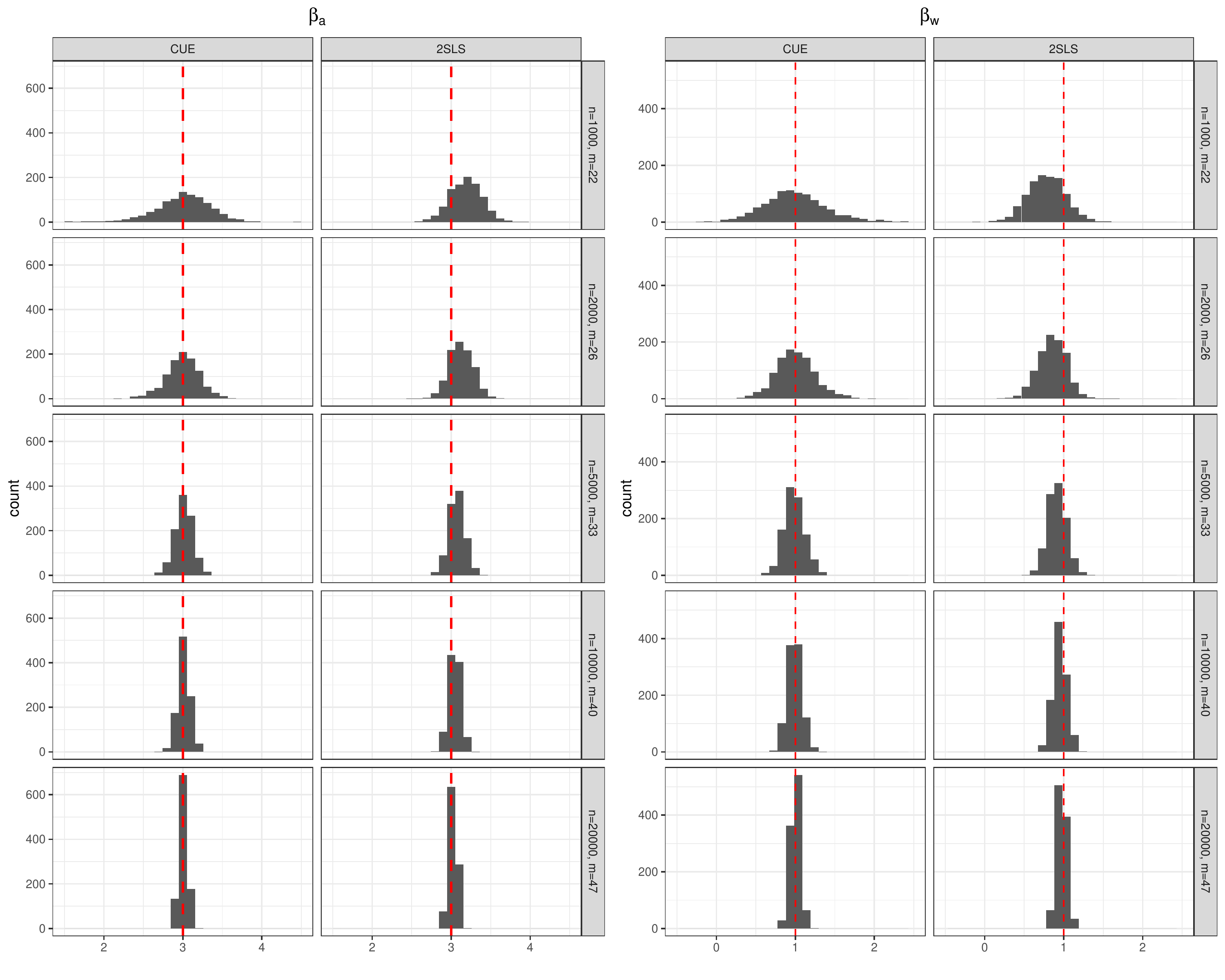}
    \caption{Sampling distribution of the CUE and 2SLS estimator under proximal causal inference settings. The dashed red lines represents the ground truth in each setting.}
    \label{fig:PSMM}
\end{figure}

\subsection{More details of Figure \ref{fig: comparison of the estimators}}

For the simulation presented in Figure \ref{fig: comparison of the estimators}, we set the sample size to $N=2000$. The data are generated following the same distribution described in Section 5.1 for the ASMM example, except that we modify the model for $A^*$ as follows:
\begin{align*}
    A^*=[\beta+3(1-\beta)]\sum_{j=1}^m N^{-1/2-j/(3m)} Z_j+X_1+\text{sin}(X_2)+\text{expit}(X_3)+U+N(0,1).
\end{align*}
Here, $\beta = 1$ corresponds to the weak identification case and $\beta = 0$ corresponds to the strong identification case. For the results using unorthogonal moment conditions, we set $\widehat{\eta}_{A,l} = \widehat{\eta}_{Y,l} = 0$.

\end{document}